\documentclass{amsart}
\usepackage[a4paper,outer=3.4cm,inner=3.4cm,total={14cm,21.4cm},bottom=3.4cm,top=3.4cm]{geometry}

\usepackage[english]{babel}
\usepackage[utf8]{inputenc}
\usepackage[foot]{amsaddr}
\usepackage{amsmath,amsthm,amssymb}
\usepackage{enumerate}
\usepackage{tikz-cd}
\usepackage[shortlabels]{enumitem}
\usepackage{mathtools}
\usepackage{latexsym}
\usepackage{thmtools}
\usepackage{array}
\usepackage[colorlinks=true]{hyperref}
\hypersetup{colorlinks, citecolor=blue, filecolor=black, linkcolor=black, urlcolor=black}
\usepackage{stmaryrd}

\usepackage[capitalize]{cleveref}
\crefformat{equation}{\normalfont(#2#1#3)}

\newtheorem{thm}{Theorem}[section]
\newtheorem{Theorem}[thm]{Theorem}
\newtheorem{Lemma}[thm]{Lemma}
\newtheorem{Proposition}[thm]{Proposition}
\newtheorem{Corollary}[thm]{Corollary}

\theoremstyle{definition}
\newtheorem{Definition}[thm]{Definition}
\newtheorem{Question}[thm]{Question}
\newtheorem{Remark}[thm]{Remark}

\newcommand{\smallmat}[4]{\left(\begin{smallmatrix} #1 & #2 \\ #3 & #4\end{smallmatrix}\right)}
\newcommand*\isomarrow{\xrightarrow{\raisebox{-0.35em}{\smash{\ensuremath{\sim}}}}}

\newcommand{\wh}{\widehat}
\newcommand{\N}{\mathbb{N}}
\newcommand{\Z}{\mathbb{Z}}
\newcommand{\Q}{\mathbb{Q}}
\newcommand{\C}{\mathbb{C}}
\newcommand{\B}{\mathcal{B}}
\newcommand{\Exp}{\mathrm{Exp}}
\newcommand{\Lie}{\operatorname{Lie}}
\newcommand{\LS}{\mathrm{LS}}
\newcommand{\BdR}{\mathrm{B}_{\dR}}

\newcommand{\m}{\mathfrak{m}}
\renewcommand{\P}{\mathbb{P}}
\renewcommand{\O}{\mathcal{O}}
\newcommand{\G}{\mathbb{G}}
\newcommand{\coker}{\operatorname{coker}}
\newcommand{\Hom}{\operatorname{Hom}}
\newcommand{\uEnd}{\underline{\mathrm{End}}}
\newcommand{\uHom}{\underline{\mathrm{Hom}}}
\newcommand{\uExt}{\underline{\mathrm{Ext}}}

\newcommand{\Ext}{\operatorname{Ext}}
\newcommand{\End}{\mathrm{End}}
\newcommand{\Aut}{\mathrm{Aut}}
\newcommand{\GL}{\mathrm{GL}}
\newcommand{\Sym}{\operatorname{Sym}}
\newcommand{\Perf}{\operatorname{Perf}}
\newcommand{\SmRig}{\mathrm{Rig}^{\mathrm{sm}}}
\newcommand{\Spec}{\operatorname{Spec}}
\newcommand{\Spa}{\operatorname{Spa}}
\newcommand{\id}{{\operatorname{id}}}
\newcommand{\cts}{{\operatorname{cts}}}
\newcommand{\an}{{\mathrm{an}}}
\newcommand{\dR}{\mathrm{dR}}
\newcommand{\hotimes}{\widehat{\otimes}}
\newcommand{\et}{{\operatorname{\acute{e}t}}}
\newcommand{\proet}{{\operatorname{pro\acute{e}t}}}
\newcommand{\HT}{\operatorname{HT}}
\newcommand{\HTlog}{\operatorname{HTlog}}

\newcommand{\Pic}{\operatorname{Pic}}
\newcommand{\uPic}{\mathbf{Pic}}
\newcommand{\wt}{\widetilde}
\newcommand{\wtOm}{\wt\Omega}

\renewcommand{\wh}{\widehat}
\renewcommand{\lim}{\varprojlim}
\newcommand{\tf}{[\tfrac{1}{p}]}
\newcommand{\can}{\mathrm{can}}

\tikzset{
	labelrotate/.style={anchor=south, rotate=90, inner sep=.5mm}} 
\makeatletter
\@namedef{subjclassname@2020}{2020 Mathematics Subject Classification}
\makeatother
\begin{document}
	\newgeometry{outer=3.4cm,inner=3.4cm,bottom=2.3cm,top=3.6cm}
	\title[A $p$-adic Simpson correspondence for smooth proper rigid varieties]{A $p$-adic Simpson correspondence\\ for smooth proper rigid varieties}
	\author{Ben Heuer}
	\address{Goethe University Frankfurt,
		Robert-Mayer-Str. 6-8,
		60325 Frankfurt am Main, Germany}
	\email{heuer@math.uni-frankfurt.de}
	\maketitle\thispagestyle{empty}
	\begin{abstract}
		For any smooth proper rigid analytic space $X$ over a complete algebraically closed extension of $\Q_p$, we construct a \mbox{$p$-adic} Simpson correspondence: an equivalence  of categories between vector bundles on Scholze's pro-\'etale site of $X$ and Higgs bundles on $X$. This generalises a result of Faltings  from smooth projective curves to any higher dimension, and further to the rigid analytic setup. The strategy is new, and is based on the study of rigid analytic moduli spaces of pro-\'etale invertible  sheaves on spectral varieties. 
	\end{abstract}
	\setcounter{tocdepth}{2}
	
	\section{Introduction}

	\subsection{Main result}
	
	Let $K$ be a complete algebraically closed extension of $\Q_p$. 
	The goal of this article is to prove the following global $p$-adic Simpson correspondence.
	
	\begin{Theorem}[{\Cref{t:p-adicSimpson-proper}}]\label{t:main-thm}
		Let $X$ be a smooth proper rigid space over $K$. 
		Then choices of a $\BdR^+/\xi^2$-lift $\mathbb X$ of $X$ and of an exponential $\Exp$ for $K$ induce an exact tensor equivalence
		\[
		\mathrm{S}_{\mathbb X,\Exp}:\big\{\text{\normalfont pro-\'etale vector bundles on }X\big\}\isomarrow \big\{\text{\normalfont Higgs bundles on }X\big\}\]
		which is natural in the datum of the pair $(X,\mathbb X)$.
	\end{Theorem}
	Here an exponential for $K$ is a continuous splitting of the $p$-adic logarithm $\log\colon 1+\mathfrak m_K\to K$, and  the two sides of the $p$-adic Simpson correspondence $
	\mathrm{S}_{\mathbb X,\Exp}$ are defined as follows:
	\begin{Definition}
		\begin{enumerate}
			\item 
		A pro-\'etale vector bundle on $X$ is a finite locally free sheaf on the pro-\'etale site $X_{\proet}$ of Scholze \cite{Scholze_p-adicHodgeForRigid} endowed with the completed structure sheaf $\widehat{\O}$.
		\item 
		A Higgs bundle on $X$ is a pair $(E,\theta)$ of an analytic vector bundle $E$ on $X$ and a morphism of $\O_X$-modules $\theta\colon E\to E\otimes \Omega_X^{1}(-1)$ satisfying $\theta\wedge \theta=0$.
		\end{enumerate}
	\end{Definition}
	
	Faltings constructed the $p$-adic Simpson correspondence in the case when $X$ is a smooth projective curve \cite{Faltings_SimpsonI}, under some further assumptions on $X$ and $K$, and formulated in terms of ``generalised representations'', which are equivalent to pro-\'etale vector bundles (see also \cite[\S1.5]{heuer-sheafified-paCS} for some background).  Since then, it has been one of the main open questions in the area  whether such a  correspondence exists in higher dimension (see e.g.\ \cite[Foreword]{AGT-p-adic-Simpson}). \Cref{t:main-thm} confirms that this is the case: It generalises  Faltings' result not only from smooth projective curves to smooth proper varieties, but further to the rigid analytic setting.
	
	Our method is quite different from that of \cite{Faltings_SimpsonI} even for curves, and we can avoid any use of semi-stable models and log-structures. Instead, we work with Scholze's perfectoid foundations of $p$-adic Hodge theory. As a consequence, our result is stronger than that of Faltings even in the case of curves, namely we need weaker choices, as we will explain below.
	
	\medskip
	Assume now additionally that $X$ is connected. 	As the name ``generalised representations'' suggests, any choice of base-point $x\in X(K)$ induces, via descent from the universal pro-finite-\'etale cover, a natural fully faithful functor 
	\[\mathrm{Rep}_{K}(\pi_1(X,x))\hookrightarrow \big\{\text{pro-\'etale vector bundles on }X\big\}\]
	from continuous representations of the \'etale fundamental group $\pi_1(X,x)$ on finite dimensional $K$-vector spaces 
	(see \cite[Thm.~5.2]{heuer-v_lb_rigid}). We thus obtain a fully faithful exact tensor functor
	\[\mathcal S_{\mathbb X,\Exp}:\mathrm{Rep}_{K}(\pi_1(X,x))\hookrightarrow \big\{\text{Higgs bundles on }X\big\}\]
	from \Cref{t:main-thm}.
	This allows us to generalise the question posed in \cite[\S5]{Faltings_SimpsonI}:\restoregeometry
	\begin{Question}\label{q:image-of-reps}
		How can we characterise the essential image of $\mathcal S_{\mathbb X,\Exp}$?
	\end{Question}
	The name ``$p$-adic Simpson correspondence'' for \Cref{t:main-thm} alludes to the famous non-abelian Hodge correspondence in complex geometry due to Corlette and Simpson \cite{SimpsonCorrespondence}.
	For a smooth \textit{projective} variety $Y$ over $\C$ with a base-point $y\in Y(\C)$, this is an equivalence of categories from  representations of the topological fundamental group $\pi_1(Y,y)$ on finite dimensional \mbox{$\C$-vector} spaces to the category of semistable Higgs bundles on $Y$ with vanishing rational Chern classes. The functor $\mathcal S_{\mathbb X,\Exp}$ is a very close analogue of this functor.
	
	\medskip
	
		\Cref{q:image-of-reps} is therefore another main question about the $p$-adic Simpson correspondence. It has been widely expected for some time that in Faltings' setting of curves over $K=\C_p$, the completion of $\overline{\Q}_p$, the essential image of $\mathcal S_{\mathbb X,\Exp}$  is given by semistable Higgs bundles of degree $0$ as over $\C$ (the degree being the only non-trivial Chern class for vector bundles on curves). However, a recent preprint of Andreatta \cite{andreatta2024padic} shows that even for curves these conditions are in general not strong enough to describe the essential image.
		 Cases in which the answer to \Cref{q:image-of-reps} is well-understood include line bundles, where it is proved in \cite[Thm.~1.1]{heuer-geometric-Simpson-Pic} that for \textit{projective} $X$ over $\C_p$, the essential image of $\mathcal S_{\mathbb X,\Exp}$ is given by Higgs line bundles with vanishing rational Chern classes. But already for line bundles on general proper $X$, or for larger fields $K$, it is shown that one in general needs stronger assumptions. In summary, these known cases indicate that \Cref{q:image-of-reps} is quite hard, and it  currently seems wide open.
	
	\subsection{$p$-adic non-abelian Hodge theory}

	Let $X$ be a smooth proper rigid space over $K$. Then we have a $p$-adic analogue of the Hodge decomposition from complex geometry: By a result of Scholze \cite[Thm.~3.20]{Scholze2012Survey}, building on the work of Tate \cite{tate1967p}, Faltings \cite{faltings1988p}  and others in the algebraic setting, the datum of a $\BdR^+/\xi^2$-lift $\mathbb X$ of $X$ induces an isomorphism
	 \begin{equation}\label{eq:Hodge-decomp}
	  H^n_{\et}(X,\Q_p)\otimes_{\Q_p} K=\textstyle\bigoplus\limits_{i+j=n}H^i(X,\Omega^j_X(-j))
	  \end{equation}
	 where $(-j)$ denotes a Tate twist.
	We note that there is a canonical such lift $\mathbb X$ if we are given a model of $X$ over a complete discretely valued subfield of $K$ with perfect residue field.
	 
	 \medskip
	 
	In complex geometry, the non-abelian Hodge correspondence of Corlette and Simpson is a non-abelian categorical generalisation of the Hodge decomposition.
	  Starting with the pioneering work of Deninger--Werner \cite{DeningerWerner_vb_p-adic_curves} and Faltings \cite{Faltings_SimpsonI}, it has been a much studied question  what a $p$-adic version of the non-abelian Hodge correspondence could look like. Our \Cref{t:main-thm} can now be regarded as providing such a ``$p$-adic non-abelian Hodge correspondence''. There are several ways to explain this interpretation:
	 The first perspective, emphasized by Abbes--Gros  \cite{AbbesGros-II}, is that  the $p$-adic Simpson corrspondence  should generalise the $p$-adic Hodge decomposition \Cref{eq:Hodge-decomp} to more general coefficients. Indeed, we prove:
	
	\begin{Theorem}\label{t:intro-cohom-comparison}
		In the setting of \Cref{t:main-thm}, let $V$ be a pro-\'etale vector bundle on $X$ and let $(E,\theta)=
		\mathrm{S}_{\mathbb X,\Exp}(V)$ be the associated Higgs bundle. Then there is a natural isomorphism 
		\[ R\Gamma(X_{\proet},V)=R\Gamma_{\mathrm{Higgs}}(X,(E,\theta)).\]
	\end{Theorem}
	Here $R\Gamma(X_{\proet},-)$ is pro-\'etale cohomology, while $R\Gamma_{\mathrm{Higgs}}(X,-)$ is Dolbeault cohomology (\Cref{d:Dolbeault-cohom}). In the simplest case of $V=\O_X$, the left hand side equals $R\Gamma_{\et}(X,\Q_p)\otimes_{\Q_p} K$ by Scholzes' Primitive Comparison Theorem \cite[Thm~5.1]{Scholze_p-adicHodgeForRigid}, while the right hand side is equal to Hodge cohomology. Hence, this case recovers the Hodge--Tate decomposition  \Cref{eq:Hodge-decomp}.
	
	\medskip
	
	There is a second, very different way in which we can regard \Cref{t:main-thm} as a generalisation of the Hodge decomposition \Cref{eq:Hodge-decomp} to more general coefficients: In the spirit of Simpson's perspective on non-abelian Hodge theory \cite[\S0]{SimpsonCorrespondence}, the functor $
	\mathrm{S}_{\mathbb X,\Exp}$ can be interpreted as a ``categorical Hodge decomposition for non-abelian coefficient groups'': 
	If $G$ is any rigid group, then 
	the set of isomorphism classes of $G$-torsors on $X_{\proet}$ is given by $H^1_{\proet}(X,G)$. For $G=\GL_r$, this classifies the pro-\'etale vector bundles of rank $r$ on $X$ up to isomorphism.

	For $G=\G_a$, we instead have $H^1_{\proet}(X,\G_a)=H^1_{\et}(X,\Q_p)\otimes_{\Q_p} K$ by the aforementioned Primitive Comparison Theorem. Second, there is a general notion of $G$-Higgs bundles for rigid groups $G$, and the set of isomorphism class of $\G_a$-Higgs bundles is $H^1(X,\O)\oplus H^0(X,\Omega^1_X(-1))$. From this perspective, the isomorphism \Cref{eq:Hodge-decomp} for $n=1$ can be regarded as matching up pro-\'etale $\G_a$-torsors and $\G_a$-Higgs bundles, and indeed one can upgrade this bijection to an equivalence between these two categories. In this sense, \Cref{t:main-thm} provides a categorified generalised Hodge decomposition in degree $n=1$ for $\GL_r$-coefficients.
	
	\subsection{Comparison to previous results}

	Faltings' article \cite{Faltings_SimpsonI} has been very influential and has sparked a great deal of activity in recent years. Its line of argument can roughly be divided into three parts: The local correspondence  between ``small'' objects in terms of a toric chart, the global correspondence between ``small'' objects in terms of a lift, and finally the global $p$-adic Simpson correspondence in the case of projective curves. 
	
	\medskip
	
	The first two steps, which we shall summarise as the ``small correspondence'', have since been the subject of extensive studies, and can now be regarded as being well-understood: This started with the work of Abbes--Gros and Tsuji \cite{AGT-p-adic-Simpson}\cite{Tsuji-localSimpson}, who have reinterpreted and studied in great detail the small correspondence for certain semi-stable schemes with log structures. More recently, the small correspondence has been studied for rigid analytic $X$ under various additional technical assumptions, including the case when $X$ is arithmetic and the pro-\'etale bundle comes from a $\Q_p$-local system due to Liu--Zhu \cite{LiuZhu_RiemannHilbert}, and the case when $X$ has good reduction due to Wang \cite{Wang-Simpson}. Moreover, there are new approaches based on prismatic crystals \cite{MinWang22}\cite{MinWang-Hodge-Tate}, leading to  a relative version of the  small correspondence in families \cite{anschutz2023hodgetate}.

	\medskip

	In contrast, the final step in \cite{Faltings_SimpsonI}, the $p$-adic Simpson correspondence for projective curves, is much less well-understood: It is deduced from the small correspondence by descent from finite \'etale covers, using a subtle construction of ``twisted pullback'' which has recently been studied further by Xu \cite{xu2022parallel}. There are two main reasons why this strategy is limited to curves: Firstly, the descent step relies on the fact  that one can make global differentials on curves $p$-adically small by passing to finite \'etale covers. Second, it uses a semistable reduction assumption, but in higher dimension one does not know if semistable models exist for any finite \'etale covers. For these reasons, this strategy does not generalise to higher dimension.
	
	\medskip
	
	Our approach to the $p$-adic Simpson correspondence is quite different from that of \cite{Faltings_SimpsonI}, and in particular it is logically independent of the global correspondence for small objects. As a consequence, even in the case of curves, \Cref{t:main-thm} is in fact more general than Faltings' result:
	Firstly, due to the different technical foundations, the base field $K$ in \Cref{t:main-thm} is more general, as Faltings assumes that $X$ admits a model $X_0$ over a discretely valued non-archimedean field $L\subseteq K$ with perfect residue field.
	More importantly, however, a new aspect is that our $p$-adic Simpson functor $	\mathrm{S}_{\mathbb X,\Exp}$ depends on a lift of $X$ to $\BdR^+/\xi^2$, whereas in \cite{Faltings_SimpsonI} it is important to instead choose a lift of a semi-stable model over $\O_K$ to the integral subring $\mathrm{A}_{\inf}/\xi^2\subseteq \BdR^+/\xi^2$, because such a datum is necessary for the global small correspondence. 
	
	\medskip

	The relevance of this improvement is that for a smooth proper variety $X_0$ over $L$, the base-change to $K$ always admits a canonical lift to $\BdR^+/\xi^2$. Indeed, one always has a canonical map $L\to \BdR^+/\xi^2$ along which one can base-change, but this does not restrict to a map $\O_L\to \mathrm{A}_{\inf}/\xi^2$ unless $L$ is absolutely unramified. Consequently,  in contrast to Faltings' result, we can eliminate the choice of lift in the important special case that $X$ admits a model $X_0$ over $L$, making it more canonical, in close analogy to the Hodge--Tate decomposition \Cref{eq:Hodge-decomp}.

	\medskip
	
	Apart from Faltings' result for curves, the only other previously known cases of a $p$-adic Simpson correspondence for proper $X$ beyond the small case are the case of line bundles, i.e.\ rank one \cite{heuer-v_lb_rigid},  and the case of projective space $X=\P^n$ \cite[Cor.~8.28]{anschutz2023hodgetate}. Moreover, there are partial results, e.g.\ in the case of vanishing Higgs field \cite{DeningerWerner_Simpson}\cite{wuerthen_vb_on_rigid_var}. But our result is new even when $X$ is an abelian variety of dimension $>1$.

	The cohomological comparison \Cref{t:intro-cohom-comparison} was previously known for the small correspondence under various additional hypothesis: These include the algebraic settings of Abbes--Gros and Tsuji \cite{AGT-p-adic-Simpson}\cite{AbbesGros-II}, the case of good reduction \cite[Thm.~1.1]{Wang-Simpson}\cite{anschutz2023hodgetate}, and  arithmetic settings of Galois-equivariant pro-\'etale vector bundles, namely for $\Q_p$-local systems due to Liu--Zhu \cite{LiuZhu_RiemannHilbert}, and more generally by Min--Wang \cite{MinWang-Hodge-Tate}. For curves, Faltings deduced it from the small case. Beyond these cases, this result is new, already for line bundles.
	
	\subsection{Strategy}
	Conceptually, the approach of this article is rooted in the idea of using $p$-adic analytic moduli spaces to study the $p$-adic Simpson correspondence, initiated in \cite{heuer-geometric-Simpson-Pic}\cite{heuer-sheafified-paCS}. 
	
	In this article, we apply this perspective to the spectral cover, an object going back to the work of Hitchin \cite{HitchinFibration}:
	Let $(E,\theta)$ be a Higgs bundle on $X$. To simplify notation, we set $\wtOm:=\Omega_X^1(-1)$. The datum of $\theta$ is then equivalent to an $\O_X$-algebra homomorphism \[\Sym^\bullet_{\O_X}\wtOm^\vee\to \uEnd(E)\] on $X_{\et}$. Let $B\subseteq \uEnd(E)$ be its image.
	 This is a coherent commutative $\O_X$-algebra and the Higgs field $\theta$ is encoded by the $B$-action on $E$ via the natural map $\Sym^\bullet\wtOm^\vee\to B$. We call the finite cover $X':=\Spa_{\O_X}B\to X$ the spectral variety\footnote{This is slightly different to the spectral variety studied in the context of the Hitchin fibration. Roughly, we replace the characteristic polynomial by the minimal polynomial. But this difference is not essential.} of $(E,\theta)$. Consider now Scholze's pro-\'etale site $X_{\proet}$ of \cite[\S3]{Scholze_p-adicHodgeForRigid} endowed with the completed structure sheaf. The very basic idea for defining $\mathrm{S}^{-1}_{\mathbb X,\Exp}$ is to use the morphism of ringed sites $\nu:X_\proet\to X_{\et}$ and to send 
	 \begin{alignat*}{2}
	 \mathrm{S}^{-1}_{\mathbb X,\Exp}:\big\{\text{\normalfont Higgs bundles on }X\big\}&\to&&\big\{\text{\normalfont pro-\'etale vector bundles on }X\big\} \\
	(E,\theta)\,\,&\mapsto&& \,\,\,\nu^{\ast}E\otimes_{\mathcal B}\mathcal L_{\theta}
	\end{alignat*}
	where $\mathcal B:=\nu^{\ast}B$ and
	where $\mathcal L_{\theta}$ will be a certain invertible (i.e.\ locally free of rank 1) \mbox{$\mathcal B$-module} on $X_{\proet}$ that depends on $(E,\theta)$. 
	In order to define $\mathcal L_{\theta}$, the key idea is to show that the moduli functor of invertible $\mathcal B$-modules is represented by a rigid group variety. To make this precise, let $\SmRig_{K,\et}$ be the site of smooth rigid spaces over $K$ with the \'etale topology and let
	\[\uPic_{X'}\colon\SmRig_{K,\et}\to\mathrm{Ab},\quad \text{sheafification of } \big(S\mapsto  H^1_{\an}(X'\times S,\O^\times)\big)\]
	be the rigid analytic Picard functor of $X'$. We define the pro-\'etale Picard functor of $\mathcal B$ as
	\[\uPic_{\B,\proet}\colon\SmRig_{K,\et}\to \mathrm{Ab},\quad \text{sheafification of } \big(S\mapsto  H^1_{\proet}(X\times S,\B^\times)/H^1_{\proet}(S,\pi_{S,\proet\ast}\B^\times)\big)\]
	where $\pi_S:X\times S\to S$ denotes the projection.
	Our main technical result is now the following ``multiplicative Hodge--Tate sequence for $\B$'':

	\begin{Theorem}\label{t:intro-ses}
		There is a short exact sequence of abelian sheaves on  $\SmRig_{K,\et}$
		\begin{equation}\label{eq:ses-intro}
		0\to \uPic_{X'}\to  \uPic_{\B,\proet} \to H^0(X,B\otimes_{\O_X} \wtOm)\otimes_K \G_a\to 0.
		\end{equation}
		If $\uPic_{X'}$ is representable,
		then \eqref{eq:ses-intro} is represented by a short exact sequence of rigid group varieties over $K$. The associated exact sequence of Lie algebras is the Hodge--Tate sequence
		\begin{equation}\label{eq:HT-ses-B}
		0\to H^1_{\an}(X,B)\to H^1_{\proet}(X,\B)\to H^0(X,B\otimes \wtOm)\to 0.
		\end{equation}
	\end{Theorem}
	
	We can now explain how we use the choices in \Cref{t:main-thm}:
	The $\BdR^+/\xi^2$-lift $\mathbb X$ induces a splitting $s_{\mathbb X}$ of \eqref{eq:HT-ses-B}. The natural map $\wtOm^\vee\to B$ induces a section $\tau_B\in H^0(X,B\otimes \wtOm)$. We now use an observation from \cite{HMZ} based on Fargues' work on $p$-divisible rigid groups \cite{Fargues-groupes-analytiques}:
	
	\begin{Theorem}[{\cite[Thm.~6.12]{HMZ}}]\label{t:intro-exp}
		For any commutative rigid group $H$ over $K$ such that $[p]\colon H\to H$ is surjective,  any exponential for $K$ induces a natural Lie exponential
		\[\Exp_H\colon \Lie(H)\to H(K).\]
	\end{Theorem}
	For algebraic groups $H$ this is shown in \cite[p.~856f]{Faltings_SimpsonI} for the definition of twisted pullback. We instead apply it to $H=\uPic^0_{\B,\proet}$, which is not algebraic even if $X$ is a curve, to get an invertible $\mathcal B$-module $\mathcal L_\theta:=\Exp_H(s_{\mathbb X}(\tau_B))\in H^1_\proet(X,\mathcal B^\times)$. In summary, we used the diagram
	\[
	\begin{tikzcd}[row sep = 0.5cm]
		\uPic_{X'}(K) \arrow[r]                 & {\uPic_{\B,\proet}}(K)\arrow[r]                    & {H^0(X,\wtOm\otimes B)\phantom{\ni \tau_B.}}                                         \\
		{H^1_{\et}(X,B)} \arrow[r] \arrow[u] & {H^1_\proet(X,\B)} \arrow[u,"\Exp_H"] \arrow[r] & {H^0(X,\wtOm\otimes B)}\ni \tau_B. \arrow[u,equal] \arrow[l, "s_{\mathbb X}"' xshift=7.5, bend right=20,start anchor=north west, end anchor = north east, yshift = -3]
	\end{tikzcd}\]
	
	 In order to make the construction canonical and functorial, we develop a notion of rigidifications of invertible pro-\'etale $\mathcal B$-modules that makes $\mathcal L_\theta$ unique up to unique isomorphism.
	
		The perspective of $p$-adic moduli spaces enters twice in the construction of $\mathcal L_\theta$: Once in order to prove right-exactness in \Cref{t:intro-ses}, and independently when we invoke \Cref{t:intro-exp}.

	The assumption in \Cref{t:intro-ses} that $\uPic_{X'}$ is representable is satisfied when $X$ is algebraic.
	As it currently seems a bit unclear whether it is true in general (see \S\ref{s:2.2}), we also give a more general construction of $\mathcal L_{\theta}$ motivated by \cite[Thm~1.2]{heuer-geometric-Simpson-Pic}: we show that the torsor $\uPic_{\B,\proet}\to \mathcal A_B$ has a reduction of structure group to $\uPic_{X'}[p^\infty]$ that is always representable.
	
	The construction of the functor $\mathrm{S}_{\mathbb X,\Exp}$ from pro-\'etale vector bundles to Higgs bundles is similar: Let $V$ be any pro-\'etale vector bundle on $X$. Then we use that by a construction of Rodr\'iguez Camargo \cite[Thm.~1.0.3]{camargo2022geometric}, one can endow $V$ with a canonical Higgs field
	\[ \theta_V:V\to V\otimes \wtOm.\]
The image $B$ of the induced morphism $\Sym^\bullet \wtOm^\vee\to \nu_\ast\uEnd(V)$ is a coherent $\O_X$-algebra. Set $\B:=\nu^\ast B$, then as before, we can use \Cref{t:intro-ses} and the choices to define an invertible $\B$-module $\mathcal L_{V}$. We then show that $(V,\theta_V)\otimes_{\mathcal B}\mathcal L_{V}^{-1} $ is an analytic Higgs bundle.

\medskip

	Conceptually, this strategy builds on our previous works on moduli spaces \cite{heuer-geometric-Simpson-Pic}\cite{heuer-diamantine-Picard}\cite{heuer-sheafified-paCS}, e.g.\ the simplest  case of $B=\O_X$ of \Cref{t:intro-ses} is closely related to \cite[Thm.~1.3]{heuer-diamantine-Picard}. However, it does not logically rely on their main results: With the exception of an easier special case of \cite[Thm.~6.5]{heuer-sheafified-paCS}, it only uses very limited technical input from these articles.

	During the final stages of this work, we learnt that Daxin Xu was independently studying moduli of invertible $B$-modules in the case of curves. In our subsequent joint work \cite{HX}, we build on this circle of ideas to compare moduli spaces in $p$-adic non-abelian Hodge theory.
	
	\subsection*{Acknowledgements}
	We thank Johannes Ansch\"utz, Bhargav Bhatt, Gabriel Dospinescu, Tongmu He, Alexander Petrov, Yupeng Wang, Annette Werner and Mingjia Zhang for very helpful discussions and for their comments on earlier versions of this article. We thank Gerd Faltings for suggesting an improved treatment of rigidifications that eliminated choices of base points from \Cref{t:p-adicSimpson-proper}, making it completely natural. We moreover thank Arthur-C\'esar Le Bras, Ruochuan Liu, Juan Esteban Rodr\'iguez Camargo,  Peter Scholze, Matti Würthen and Daxin Xu for very helpful conversations. Finally, we thank the referee for very helpful comments.
	This project was funded by Deutsche Forschungsgemeinschaft (DFG, German Research Foundation) -- Project-ID 444845124 -- TRR 326.
	\section{The relative pro-\'etale Picard variety of the spectral cover}
	\subsection{Setup}
	Let $K$ be a complete algebraically closed non-archimedean field over $\Q_p$ as in the introduction. Throughout, we work in the setting of \cite{Scholze_p-adicHodgeForRigid} and thus work with analytic adic spaces in the sense of Huber: By a rigid space, we mean an adic space locally of topologically finite type over $\Spa(K,\O_K)$. We use the following notation from \cite[Def.~2.10]{heuer-sheafified-paCS}:

	\begin{Definition}\label{d:wtOm}
		For any smooth rigid space $X$ over $K$, we write $\wtOm_X:=\Omega_{X|K}^1(-1)$ for the Tate twist of the sheaf of K\"ahler differentials on $X$. We also just write $\wtOm=\wtOm_X$ if $X$ is clear from the context. For any $k\in \N$, we set $\wtOm_X^k=\wedge^k_{\O_X}\wtOm_X$. We can always make a choice of $p$-power unit roots in $K$ to trivialise the Tate twist, but it is better to remember it to get the correct notion of functoriality in $K$, especially in situations where there is a Galois action.
	\end{Definition}
	For any rigid space $X$, we denote by $\SmRig_X$ the category of smooth rigid spaces over $X$, and by $\Perf_{X}$ the category of perfectoid spaces over $X$ in the sense of \cite{perfectoid-spaces}. If $X=\Spa(K)$, we also write this category as $\Perf_K$. We endow it with the \'etale topology, or with the finer v-topology, to make them into sites. We indicate the topology by an index, e.g.\ $\Perf_{K,\et}$.
	
	For any rigid space $X$, there is an associated diamond $X^\diamondsuit$ in the sense of \cite[\S15]{etale-cohomology-of-diamonds}, which we can equivalently regard as a sheaf $X^\diamondsuit:\Perf_{K,v}\to \mathrm{Sets}$. Since the resulting functor $-^\diamondsuit:\SmRig_K\to \Perf_{K,v}$ is fully faithful, we will freely switch back and forth between smooth rigid spaces and their associated diamonds, and therefore usually drop the $-^\diamondsuit$ from notation.

For any rigid space $X$, we denote by $X_\proet$ the pro-\'etale site  defined by Scholze in \cite[\S3]{Scholze_p-adicHodgeForRigid}, endowed with the completed structure sheaf. It will later (in \Cref{p:RnuB^*}) be important that we use this original version of the pro-\'etale site (also referred to as the ``flattened pro-\'etale site''), rather than newer definitions. That being said, this technical subtlety is irrelevant for the statement of \Cref{t:main-thm}, since vector bundles agree among all pro-\'etale sites as well as the v-site (due to a result of Kedlaya--Liu \cite[Thm.~3.5.8]{KedlayaLiu-II}, see \cite[\S1.1]{heuer-v_lb_rigid} for more details).
	
\subsection{The relative pro-\'etale Picard variety of a coherent algebra}\label{s:2.2}

Let $X$ be a proper rigid space over $K$. Then we have the rigid Picard functor
\[ \uPic_{X}:\SmRig_{K,\et}\to \mathrm{Ab}\]
from the category of smooth rigid spaces over $K$ to the category of abelian groups,
defined as the \'etale sheafification of the presheaf
\[Y\mapsto H^1_{\an}(X\times Y,\O^\times).\]
Here on the right hand side, we can equivalently use the \'etale topology by \cite[Prop.~8.2.3]{FvdP}.

It is conjectured that $\uPic_{X}$ is always represented by a smooth rigid group variety. This is known e.g.\ when $X$ is algebraisable, i.e.\ when there is a proper scheme $X_0$ over $K$ such that $X=X_0^\an$: In this case, $\uPic_{X_0}$ is represented by a locally finite type group scheme over $K$ by Theorems of Grothendieck, Murre and Oort, see \cite[\S8.2]{BLR-II}. We then have
$\uPic_{X}=\uPic_{X_0}^\an$
by K\"opf's relative GAGA Theorem \cite{Koepf_GAGA}. 
There are other cases in which $\uPic_X$ is known to be representable. We mention \cite{HartlLutk}\cite{warner2017adic} and also refer to \cite[\S1]{heuer-diamantine-Picard} for an overview.

\begin{Definition}\label{d:Pic_proet}
	Let $B$ be a (commutative) coherent $\O_X$-algebra on $X_{\et}$ and let $\mathcal B:=\nu^\ast B$ be the associated sheaf on $X_{\proet}$. For any smooth rigid space $S$, let us write $\pi_S:X\times S\to S$ for the natural projection. Then we define the pro-\'etale Picard functor of $B$
	\[\uPic_{\B,\proet}:\SmRig_{K,\et}\to \mathrm{Ab}\] to be the \'etale sheafification of the presheaf
	\[S\mapsto H^1_\proet(X\times S,\mathcal B^\times)/H^1_{\proet}(S,\pi_{S,\proet\ast}\B^\times).\]
	Here and in the following, we also regard $B$ and $\mathcal B$ as sheaves on $X\times S$ via pullback from $X$.
\end{Definition}
\begin{Remark}
	We will see in \Cref{p:explicit-description-of-Pic_B} that  the above presheaf is already a sheaf before the sheafification, at least when $B$ is $\O_X$-torsionfree. But we do not need this in this section.
\end{Remark}
The main goal of this section is to prove the following result:
	\begin{Theorem}\label{t:rel-vPic-spectral-var}
		Let $\pi:X\to \Spa(K)$ be a smooth proper rigid space. Let $B$ be a coherent $\O_X$-algebra on $X_{\et}$ and let $f:X':=\Spa_{\O_X}(B)\to X$ be the associated finite morphism. We denote by $\mathcal B:=\nu^\ast B$ the associated sheaf on $X_{\proet}$.  Then we have:
		\begin{enumerate}
			\item
			 There is a canonical short exact sequence of abelian sheaves on $\SmRig_{K,\et}$
			\[
		 0\to \uPic_{X'}\to  \uPic_{\B,\proet} \xrightarrow{\HTlog} H^0(X,B\otimes_{\O_X} \wtOm_X)\otimes_K \G_a\to 0\]
		which is functorial in $B$ and $X$.
		\item The sequence in (1) becomes split over an open subgroup of $H^0(X,B\otimes \wtOm_X)\otimes \G_a$.
	\end{enumerate}
	Assume now furthermore that the rigid analytic Picard functor $\uPic_{X'}$ is representable by a rigid group. For example, this is the case when $X$ is algebraisable. Then we moreover have:
	
	\begin{enumerate}\setcounter{enumi}{2}
		\item  The sequence in (1) is representable by a sequence of rigid group varieties.
		 
		\item The induced sequence of Lie algebras over $K$ obtained by passing to tangent spaces at the identity is canonically isomorphic to the Hodge--Tate sequence of $B$
		\[ 0\to H^1_{\an}(X,B)\to H^1_\proet(X,\B)\to H^0(X,B\otimes \wtOm_X)\to 0.\]
		\item The multiplication map $[p]$ on the identity component of $\uPic_{\B,\proet}$ is finite \'etale.
	\end{enumerate}
	\end{Theorem}
	\begin{Remark}
	In comparison to the rigid analytic Picard functor $\uPic_{X}$, we have made the following changes in the definition of $\uPic_{\B,\proet}$: 
	\begin{itemize}
		\item We have replaced the analytic (or \'etale) topology  with the pro-\'etale topology.
		\item We have replaced $\O^\times$ with units in a coherent $\O_X$-algebra $B$.
	\end{itemize}
	There is a technical variant of  $\uPic_{\B,\proet}$ where we additionally make the following changes:
	\begin{itemize}
		\item We replace the test category $\SmRig_{X,\et}$ by $\Perf_{K,\et}$.
		\item We replace the pro-\'etale topology with the v-topology.
	\end{itemize}
	It is possible to relate these two versions to each other via ``diamondification'' \cite[\S2]{heuer-diamantine-Picard}. But for the purpose of this article it is much easier for technical reasons (and arguably more natural)  to instead work with rigid test objects. Nevertheless, both approaches would work. We note that up to this non-essential technical difference, we can then recover the ``diamantine v-Picard variety'' from \cite[Thm.~1.3]{heuer-diamantine-Picard} as the simplest special case of $B=\O_X$.
	\end{Remark}
	\begin{Remark}
		In contrast to the natural map $H^1_\et(X,B^\times)\to H^1_\et(X',\O_{X'}^\times)$, which is an isomorphism,
		the natural map $H^1_\proet(X,\B^\times)\to H^1_\proet(X',\O_{X'}^\times)$ is usually not an isomorphism, i.e.\ it makes a difference whether we consider invertible $\B$-modules on $X_\proet$ or invertible $\O_{X'}$-modules on $X'_\proet$. For example, $X'$ could be non-reduced, but as perfectoid algebras are reduced, $\O_{X'}$ on $X'_\proet$ does not see the non-reduced structure. In contrast, the algebra $\mathcal B$ on $X_\proet$ keeps this structure. This motivates the relative setup of \Cref{t:rel-vPic-spectral-var}.
	\end{Remark}
		
	\subsection{The  Leray exact sequence of the sheaf $\mathcal B^\times$}
	Let $X$ be any quasi-compact smooth rigid space and let $B$ be a coherent $\O_X$-module on $X_\an$. We are mostly interested in the case that $B$ is a coherent $\O_X$-algebra, but it seems worth recording some of the results of this section more generally also for modules as this is possible without extra work. We denote by
	\[ \nu:X_\proet\to X_{\et},\quad \lambda:X_\proet\to X_{\an}\]
	the natural morphisms of sites.	
	As before, we set $\B:=\lambda^\ast B$.
	
	\begin{Lemma}\label{l:coherent-integral-sumodule}
		There is a finitely presented $\O_X^+$-module $B_0$ with an isomorphism $B_0\tf\isomarrow B$ of $\O_X$-modules. If $B$ is moreover an $\O_X$-algebra, we can arrange for $B_0$ to be an $\O_X^+$-algebra.
	\end{Lemma}
	\begin{proof}
		When $B$ is an $\O_X$-algebra, let $f:X':=\Spa_{\O_X}(B)\to X$ be the finite morphism defined by $B$.
		We use the results of \cite{BLR-II} on existence of formal models: These guarantee that we can find a formal model $\mathfrak f:\mathfrak X'\to \mathfrak X$ of $f$ which is finite (\cite[Prop.~5.9.17]{AbbesEGR}, or combine \cite[3.3.8.(b) and Thm.~3.3.12]{Lutkebohmert_RigidCurves} using that quasi-finite and proper implies finite). Then $\mathfrak B:=\mathfrak f_{\ast}\O_{\mathfrak X'}$ is a coherent $\O_{\mathfrak X}$-module (\cite[Thm.~2.11.5]{AbbesEGR}) whose rigid  generic fibre is $B$. Let $\eta:X_{\an}\to \mathfrak X$ be the natural morphism of ringed spaces, then  $\eta^{-1}\mathfrak B\otimes_{\eta^{-1}\O_{\mathfrak X}}\O_{X_{\an}}^+$ has the desired properties.
		
		When $B$ is only a coherent $\O_X$-module, we instead use \cite[Prop.~4.8.18.(ii)]{AbbesEGR}.
	\end{proof} 
	For our coherent $\O_X$-module $B$, let us fix a choice of an $\O^+_X$-module $B_0$ like in \Cref{l:coherent-integral-sumodule}. We also write $B_0$ for the associated $\O^+$-module on $X_{\et}$ obtained by pullback. Note that $B_0$ is not necessarily $p$-torsionfree. We let $B_{0}[p^\infty]\subseteq B_0$ be the $p$-power torsion submodule and define an $\O_X^+$-module
	\[ B^+:=B_{0}/B_{0}[p^\infty]\]
	on $X_\et$.
	Equivalently, this is the image of the given map $B_0\to B$ on $X_\et$. We caution that in general, the sheaf $B_0$ may have unbounded $p$-torsion, and thus $B^+$ may not be coherent.
	
	When $B$ is an $\O_X$-algebra, we always choose $B_0$ to be an $\O_X^+$-algebra, then so is $B^+$.
	
	   We make the analogous definitions also on $X_\proet$, namely we define
	\[\mathcal B_0:=\lambda^{\ast}B_0=\nu^{-1} B_0\otimes_{\nu^{-1}\O^+_{X_{\et}}}\O^+_{X_\proet}\to  \mathcal B.\]
	By right-exactness of $\lambda^{\ast}$, this is still finitely presented in the sense that locally in $X_{\proet}$, there is an exact sequence
	$\O^{+m}\to \O^{+n}\to \B_0\to 0.$
	Like before, we set 
	\[ \mathcal B^+:=\mathcal B_{0}/\mathcal B_{0}[p^\infty]\subseteq \B.\]
	
	\begin{Lemma}\label{l:approx-prop-for-B^+/p}
		\begin{enumerate}
		\item For any $k\in \Z$, we have $\B/p^k\B^+=\nu^{-1}(B/p^kB^+)$ on $X_{\proet}$. As a consequence, we have
		$R\nu_{\ast}(\B/p^k\B^+)=B/p^kB^+$.
		\item When $B$ is a coherent $\O_X$-algebra, we have $\B^\times/(1+p\B^+)=\nu^{-1}(B^\times/(1+pB^+))$ on $X_{\proet}$. In particular, we have
		$R\nu_{\ast}(\B^\times/(1+p \B^+))=B^\times/(1+p B^+)$.
		\item Moreover, the exponential then defines an isomorphism of sheaves of groups on $X_\proet$
		\[ \exp:p\B^+\isomarrow 1+p \B^+,\quad x\mapsto \textstyle\sum_{n=0}^\infty\frac{x^n}{n!}\]
		if $p>2$, or $\exp:4\B^+\isomarrow 1+4 \B^+$ if $p=2$. 
		\end{enumerate}
	\end{Lemma}
	In order to avoid having to distinguish $p>2$ and $p=2$ throughout this article, we sometimes summarise \Cref{l:approx-prop-for-B^+/p}.3 by writing $\exp$ as a map $\exp:2p\B^+\isomarrow 1+2p \B^+$.
	\begin{proof}
		\begin{enumerate}[wide, labelwidth=!, labelindent=0pt]
			\item As there is a natural map from right to left,  the statement is local on $X_{\et}$. We can therefore assume that $B_0$ admits a presentation on  $X_{\et}$ of the form
		\[ \O^{+m}_X\to \O^{+n}_X\to B_0\to 0.\]
		This stays right-exact after tensoring with the locally constant sheaf $\Q_p/p^k\Z_p$. Note that
		\[ B_0\otimes_{\Z_p}\Q_p/p^k\Z_p=B/p^kB^+\]
		as we can see from applying $B_0\otimes-$ to the short exact sequence  $p^k\Z_p\to \Q_p\to \Q_p/p^k\Z_p$. Hence we obtain a right-exact sequence
		\begin{equation}\label{eq:pres-B/p^kB+}
		(\O_X/p^k\O^{+}_X)^m\to (\O_X/p^k\O^{+}_X)^n\to B/p^kB^+\to 0.
		\end{equation}
		The same argument also applies on $X_\proet$ to give an analogous presentation of
		\[\B/p^k\B^+=\B_0\otimes_{\Z_p}\Q_p/p^k\Z_p=\nu^\ast(B_0\otimes_{\Z_p}\Q_p/p^k\Z_p)=\nu^\ast(B/p^kB^+).\]
		
		Consider now the natural transformation $\nu^{-1}\to \nu^{\ast}$. This induces an isomorphism on $\O_X/p^k\O^{+}_X$. Applied to the right-exact sequence \Cref{eq:pres-B/p^kB+},  we deduce from this that
		 \[\nu^{-1}(B/p^kB^+)\isomarrow \nu^{\ast}(B/p^kB^+)=\B/p^k\B^+.\]
		
		The second part of the statement now follows from  \cite[Lemma 3.16 and Cor.~3.17]{Scholze_p-adicHodgeForRigid}.
		\setcounter{enumi}{2}
		\item The statement is local on $X$, so we may assume that $X=\Spa(R,R^+)$ is affinoid and that there is a surjection $q:\O^{+m}\to B_0$. It suffices to see that for any affinoid perfectoid $U\in X_\proet$ that lives over some affinoid $V\in X_\et$, the image $S=S_{U\to V}$ of the natural map
		\[\O^{+m}_X(V)\otimes_{\O_X^+(V)}\O_X^+(U)\xrightarrow{q\otimes \id} B(V) \otimes_{\O_X(V)}\O_X(U)\]
		 is $p$-adically complete: This implies that we get a bijective exponential on $S$ (see e.g.\ \cite[Lemma~2.18]{heuer-v_lb_rigid}), and we obtain $\B^+$ from the algebras $S_{U\to V}$ by sheafification in $U\to V$. 
		 
		 As $S$ is finitely generated over $\O_X^+(U)$, it is clearly bounded in the Banach $K$-algebra $B(V) \otimes_{\O_X(V)}\O_X(U)$. To see that it is $p$-adically complete, it thus suffices to see that $S$  is moreover closed. For this we use that $q:\O^m_X(V)\to B(V)$ is surjective, hence $S$
		 is open by the Banach Open Mapping Theorem. As $S$ is an additive subgroup, it is thus also closed.
		
		Note that the same argument applied on $X_\et$ yields an isomorphism $2pB^+\simeq 1+2pB^+$. This shows that $1+pB^+\subseteq B^\times$ (even for $p=2$ since $(1+2a)(1-2a)\in 1+4B^+$ for $a\in B^+$).
		\setcounter{enumi}{1}
		\item This follows from (1) by the same argument as in \cite[Lemma~2.17]{heuer-v_lb_rigid}: Namely, let $Z\approx \varprojlim_i Z_i$ be any pro-\'etale affinoid perfectoid tilde-limit in $X_{\proet}$ with $Z_i\in X_{\et}$. We claim that 
		\[ \B^\times(Z)/(1+p\B^+)(Z)=\textstyle\varinjlim_i B^\times(Z_i)/(1+pB^+)(Z_i).\]
		This shows the desired statement by \cite[Lemma 3.16]{Scholze_p-adicHodgeForRigid}.
		
		Let $f\in \B^\times(Z)$ and let $k\in \N$ be large enough so that $p^kf\in p\B^+(Z)$ and $p^kf^{-1}\in p\B^+(Z)$. Then by (1) we can pass to an \'etale cover of $Z_i$ for some $i$ to find $g$ and $g'$ in $\varinjlim B(Z_i)$ whose images under the map $\phi:\varinjlim B(Z_i)\to \B(Z)$ satisfy $\phi(g)-f\in p^k\B^+(Z)$ and $\phi(g')-f^{-1}\in p^k\B^+(Z)$. Then $\phi(gg'-1)\in p\B^+(Z)$, hence $gg'-1\in pB^+(Z_i)$ for some $i$ large enough by part (1). As we have seen in the proof of (3) that $1+pB^+(Z_i)$ is a subgroup of $B^\times(Z_i)$, this shows that $g \in B^\times(Z_i)$. Since $\phi(g)f^{-1}\in 1+p\B^+(Z)$, this shows surjectivity.

		 Injectivity follows from (1) by a similar argument.
		\qedhere
	\end{enumerate}
	\end{proof}

	\begin{Proposition}\label{p:RnuB^*}
		Let $X$ be a smooth rigid space and let $\nu:X_{\proet}\to X_{\et}$ be the natural map. Let $B$ be a coherent $\O_X$-module on $X$. We consider the functor \[\nu^{\ast}:=\O_{X_{\proet}}\otimes_{\nu^{-1}\O_{X_\et}}\nu^{-1}-:\mathrm{Mod}(X_{\et},\O_{X_\et})\to \mathrm{Mod}(X_{\proet},\O_{X_\proet})\] where we recall that we denote by $\O=\O_{X_{\proet}}$ the completed structure sheaf. Let $\mathcal B=\nu^{\ast}B$.
		\begin{enumerate}[label=(\roman*)]
			\item We have $\mathcal B=\nu^{\ast}B= L\nu^{\ast} B$ via the natural map.
			\item For any $n\in \mathbb Z_{\geq 0}$, we have $R^n\nu_{\ast}\B=
			B\otimes_{\O_X} \wtOm^n_X$.
		\item If $B$ is a coherent $\O_X$-algebra, then $\nu_{\ast}\B^\times= B^\times$ and $R^n\nu_{\ast}\B^\times=B\otimes \wtOm^n_X$ for $n\geq 1$.		\end{enumerate}
	\end{Proposition}
	\begin{Remark}
	Relations like $L\nu^{\ast}B=\nu^\ast B$ are the reason why $X_{\proet}$ is also called the ``flattened pro-\'etale site''. Note that (i) becomes false if we replace  $X_{\proet}$ by $X_v$, which for example includes points $\Spa(K)\to X$. That said, (ii) still has a chance to be true for $X_v$.
	\end{Remark}
	\begin{proof}
		\begin{enumerate}[wide, labelwidth=!, labelindent=5pt,label=(\roman*)]
			\item 
		All statements are local on $X$, so we may assume that $X=\Spa(R)$ is affinoid. Since $R$ is then regular, there exists a finite free resolution
		\[0\to \O_X^{n_1}\to \dots\to \O_X^{n_l}\to B\to 0.\]
		We need to show that this stays exact after applying $\nu^{\ast}$ as this computes $L\nu^{\ast}$.
		
		To do so, we use that a basis of $X_{\proet}$ is given by objects that are of the form
		\[ U_3\xrightarrow{f_3} U_2\xrightarrow{f_2} U_1\xrightarrow{f_1} X,\quad \text{where}\]
		\begin{enumerate}[wide, labelwidth=!, labelindent=12pt]
			\item $f_1$ is an \'etale morphism,
			\item $f_2$ is obtained from a toric chart $U_1\to \mathbb T^d$ as the pullback of the affinoid perfectoid toric cover  $\mathbb T^d_\infty\to \mathbb T^d$, and
			\item $f_3$ is a pro-finite-\'etale map of affinoid perfectoid  objects.
		\end{enumerate}
		It therefore suffices to prove that $f_3^{\ast}f_2^\ast f_1^\ast-$ preserves the exactness.
		
		The functor $f_1^{\ast}$ is exact since $f_1$ is a flat morphism of rigid spaces. Hence we may assume without loss of generality that $U_1=X$.
		
		To show that $f_2$ preserves the exactness, write $\mathbb T^d=\Spa(A,A^+)$ and $\mathbb T^d_\infty=\Spa(A_\infty,A^+_\infty)$. We need to see that $\hotimes_AA_\infty$ preserves the exactness. For this we use that after rescaling the transition morphisms, we can find an integral model
		\begin{equation}\label{eq:int-version-of-res-of-B}
		0\to \O_X^{+,n_1}\to \dots\to \O_X^{+,n_l}\to B^+\to 0
		\end{equation}
		of our complex which has bounded $p$-torsion cohomology. Since $A^+\to A^+_\infty$ is faithfully flat mod $p^n$, the complex still has bounded $p$-torsion cohomology after applying $\hotimes_{A^+}A_\infty^+$, and thus is exact after inverting $p$.
		
		Finally, the map $f_3$ is of the form $\Spa(D,D^+)\to \Spa(C,C^+)$ for perfectoid algebras $C$ and $D$ such that $C^+/p^k\to D^+/p^k$ is a filtered colimit of almost finite \'etale maps by almost purity \cite[Thm.~5.2]{perfectoid-spaces} for any $k\in \N$.  In particular, it is almost faithfully flat mod $p^k$. Thus the same argument as for $f_2$ shows that $f_3^\ast-$ preserves exactness.
\item
		We have $R^n\nu_{\ast}\O=\wtOm_X^n=\Omega_X^{n}(-n)$ by \cite[Prop.~3.23]{Scholze2012Survey}. Since $B$ is a perfect complex on the smooth rigid space $X$, it follows from part (i) and the projection formula that
		\[ R\nu_{\ast}\B=R\nu_{\ast}L\nu^\ast B=B\otimes^L_{\O_X} R\nu_{\ast}\mathcal O=B\otimes_{\O_X} R\nu_{\ast}\mathcal O\]
		where as before we denote by $\O$ the completed structure sheaf on $X_{\proet}$.
		\item That $\nu_{\ast}\B^\times=B^\times$ follows from (ii) by taking units. For the second part, we now use the isomorphism $\exp$ from \Cref{l:approx-prop-for-B^+/p}.(3). This induces a short exact sequence 
		\[
		 0\to \B^+\xrightarrow{\exp(p\cdot -)} \B^\times\to \B^\times/(1+p\B^+)\to 0\]
		 for $p>2$, and similarly for $p=2$ using $\exp(4\cdot -)$ instead.
		Applying $R\nu_{\ast}$ and using \Cref{l:approx-prop-for-B^+/p}.(2), this shows that the exponential defines an isomorphism
		$R^n\nu_{\ast}\B^+=R^n\nu_{\ast}\B^\times$
		for $n> 0$. Second, using the short exact sequence
		\[ 0\to \B^+\to  \B\to \B/\B^+\to 0,\]
		we see from \Cref{l:approx-prop-for-B^+/p}.(1) and 
	 part (ii) that
		\[R^n\nu_{\ast}\B^+=R^n\nu_{\ast}\B=B\otimes_{\O_X} R\nu_{\ast}\mathcal O.\qedhere\]
		\end{enumerate}
	\end{proof}
	\begin{Corollary}\label{c:Leray-seq-for-B^x}
		\begin{enumerate}
		\item When $B$ is a coherent $\O_X$-module, we have a left-exact sequence of $B(X)$-modules, functorial in $X$ and $B$, which we call the Hodge--Tate sequence of $B$:
		\[0\to  H^1_{\et}(X,B)\to  H^1_{\proet}(X,\B)\xrightarrow{\HT_B} H^0(X,\wtOm_X\otimes B).\]
		
		\item When $B$ is a coherent $\O_X$-algebra, we have a left-exact sequence of abelian groups, functorial in $X$ and $B$, which we call the multiplicative Hodge--Tate sequence of $B$:
		\[0\to  H^1_{\et}(X,B^\times)\to  H^1_{\proet}(X,\B^\times)\xrightarrow{\HTlog} H^0(X,\wtOm_X\otimes B).\]
		\end{enumerate}
	\end{Corollary}
	\begin{proof}
		We form the Leray sequence and use \Cref{p:RnuB^*}.(ii) and (iii).
	\end{proof}
	Applying the Corollary to $X\times Y$ for any $Y\in\SmRig_{K,\et}$, it follows that for varying $Y$, we obtain short exact sequences, functorial in $X$ and $Y$,
	\[ 0\to H^1_{\et}(X\times Y,B^\times)\to H^1_{\proet}(X\times Y,\B^\times)\to H^0(X\times Y,\wtOm_{X\times Y}\otimes_{\O_X} B).\]
	Recall now that  we have the product formula for differentials $\Omega^1_{X\times Y}=\pi_X^\ast\Omega^1_X\oplus \pi_Y^{\ast}\Omega_Y^1$
	where $\pi_X:X\times Y\to X$ and $\pi_Y:X\times Y\to Y$ denote the respective projections. We now use that $X$ is proper. By Kiehl's Proper Mapping Theorem and the resulting rigid version of ``cohomology and base-change'' \cite[Cor.~3.10, Thm.~3.18.(b)]{heuer-relative-HT}, this shows that
	\[H^0(X\times Y,\wtOm_{X\times Y}\otimes  B)=\Big(H^0(X,\wtOm_X\otimes B)\otimes_K\O(Y) \Big) \oplus \Big(H^0(Y,\wtOm_Y)\otimes_K B(X)\Big).\]
	Similarly, the pro-étale version of ``cohomology and base-change'' \cite[Thm.~3.18.(c)]{heuer-relative-HT} shows
	\begin{equation}\label{eq:coh-bc}
	\pi_{Y\ast}\B=B(X)\otimes_K \O_{Y_\proet}. 
	\end{equation}
	Using functoriality of the above sequence applied to the morphism $\pi_Y:X\times Y\to Y$, we thus obtain the following morphism of left-exact sequences:
	\[\begin{tikzcd}
		0 \arrow[r] & {H^1_{\et}(Y,\pi_{Y\ast}B^\times)} \arrow[d] \arrow[r] & {H^1_{\proet}(Y,\pi_{Y\ast}\B^\times)} \arrow[d] \arrow[r] & {H^0(Y,\wtOm_Y)\otimes_KB(X)} \arrow[d] \\
		0 \arrow[r] & {H^1_{\et}(X\times Y,B^\times)} \arrow[r] \arrow[r]  & {H^1_{\proet}(X\times Y,\B^\times)} \arrow[r] & {H^0(X\times Y,\wtOm_{X\times Y}\otimes B)}.
	\end{tikzcd}\]
	Recall from \Cref{p:RnuB^*} that the top line arises from the Leray sequence of the morphism $Y_\proet\to Y_\et$ for the sheaf $\pi_{Y\ast}\B^\times$. As the fourth term in this sequence is $H^2_{\et}(Y,\pi_{Y\ast}B^\times)$, we see that the top right map is surjective after sheafification in $Y$. Upon sheafification, it follows that on cokernels, we obtain a left-exact sequence of sheaves on $\SmRig_{K,\et}$
	\[ 0\to \uPic_{X'}\to\uPic_{\B,\proet}\to  H^0(X,\wtOm_X\otimes B)\otimes_K \G_a\]
	as described in
	 \Cref{t:rel-vPic-spectral-var}. 
	 
	To get the desired short exact sequence, we are thus left to prove the right-exactness, which is more difficult. As a preparation, we first discuss the partial splitting. For this we use:

	\subsection{The Higgs--Tate torsor of Abbes--Gros}
	We now define the analogue of the Higgs--Tate torsor of Abbes--Gros \cite[II.10.3]{AGT-p-adic-Simpson} in the analytic setting of the pro-\'etale site.
	
	\begin{Definition}\label{d:Higgs-Tate}
		Let $X$ be a smooth rigid space over $K$ and let $\mathbb X$ be a flat lift of $X$ to $\BdR^+/\xi^2$.  Via the homeomorphism $|\mathbb X|=|X|$, we may regard $\O_{\mathbb X}$ as a sheaf on $X_{\an}$. We define a pro-\'etale sheaf $L_{\mathbb X}$ on $X_{\proet}$ as the subsheaf of $\uHom(\lambda^{-1}\O_{\mathbb X},\mathbb B_{\dR}^+/\xi^2)$ defined as follows:
		\[L_{\mathbb X}:=\left\{\text{\begin{tabular}{l} {\parbox{4.5cm}{homomorphisms $\widetilde{\varphi}$ of sheaves of $\BdR^+/\xi^2$-algebras on $X_{\proet}$ making the following diagram of sheaves commutative:}}\end{tabular}}
		\begin{tikzcd}[row sep =0.55cm]
			\lambda^{-1}\O_{\mathbb X} \arrow[d] \arrow[r,"\widetilde\varphi"] & \mathbb B_{\dR}^+/\xi^2 \arrow[d] \\
			\lambda^{-1}\mathcal O_X \arrow[r]             & \O_{X_{\proet}}    
		\end{tikzcd}\right\}.\]
	\end{Definition}
	We note that such a flat lift $\mathbb X$  of $X$ to $\BdR^+/\xi^2$ always exists when $X$ is proper, see \cite[Prop.~7.4.4]{Guo_HodgeTate}. 
	\begin{Lemma}\label{l:Faltings-ext}
		\begin{enumerate}
			\item 
		$L_{\mathbb X}$ is a pro-\'etale torsor under $\nu^\ast\wtOm_X^{\vee}$ on $X_{\proet}$. 
		\item 	$L_{\mathbb X}$ is natural in $(X,\mathbb X)$ in the following sense: Let $f:Y\to X$ be a morphism of smooth rigid spaces over $K$ with a lift $\tilde{f}:\mathbb Y\to \mathbb X$ to $\BdR^+/\xi^2$. Then $\tilde{f}$ induces an isomorphism of torsors between $f^\ast L_{\mathbb X}$ and the pushout of  $L_{\mathbb Y}$ along $\nu^\ast\wtOm_Y^\vee\to f^{\ast}\nu^\ast\wtOm_X^\vee$.
		\end{enumerate}
	\end{Lemma} 
	This is closely related to the discussion in \cite[\S2.2]{Wang-Simpson}, see also \Cref{r:Faltings-extension}.
	\begin{proof}
		For part 1, assume first that $X=\Spa(R)$ is affinoid with a lift $\mathbb X=\Spa(\widetilde{R})$. Let $Y\to X$ be any affinoid perfectoid object of $X_\proet$, and write $S=\O(Y)$. Then $L_{\mathbb X}(Y)$ describes the dotted morphisms making the following diagram commutative:
		\[
		\begin{tikzcd}
			\BdR^+/\xi^2\arrow[r]\arrow[d]&\widetilde{R} \arrow[d] \arrow[r, dashed,"\wt\varphi"] & A_{\inf}(S)/\xi^2\tf \arrow[d] \\	K\arrow[r]&
			R \arrow[r]                                          & S.                            
		\end{tikzcd}\]
		As $\wt{R}$ is formally smooth over $\BdR^+/\xi^2$ and the rightmost vertical map is a square-zero thickening, such a lift $\wt\varphi$ always exists. The kernel of $\mathbb B_{\dR}^+/\xi^2\to \mathcal O_{X}$ is given by the Tate twist $\mathcal O_{X}(1)$. Hence, by deformation theory, for any two such lifts, the difference is a derivation $\widetilde{R}\to S(1)$, or equivalently, an $R$-linear morphism $\wtOm^1_R\to S$. 
		The resulting action of $\nu^\ast\wtOm^\vee_X(Y)$ on $L_{\mathbb X}(Y)$ is clearly natural in $Y\in X_\proet$. Hence  it defines a $\nu^\ast\wtOm_X^{\vee}$-action on $L_{\mathbb X}$.
		
		 The above local consideration then shows that  $L_{\mathbb X}$ is a torsor under this $\nu^\ast\wtOm^{\vee}_X$-action.
		 
		 For part 2, we first note that $f^{\ast}L_{\mathbb X}$ on $Y_\proet$  is clearly given by the homomorphisms $f^{-1}\lambda^{-1}\mathcal O_{\mathbb X}\to \mathbb B_{\dR}^+/\xi^2$ lifting the map $f^{-1}\lambda^{-1}\mathcal O_X\to \lambda^{-1}\mathcal O_Y\to \mathcal O_{Y_\proet}$. The isomorphism is now induced by the map $L_{\mathbb Y}\to f^{\ast}L_{\mathbb X}$ defined by sending any $\tilde{\varphi}:\lambda^{-1}\mathcal O_{\mathbb Y}\to \mathbb B_{\dR}^+/\xi^2$ in $L_{\mathbb Y}$ to its composition with $\tilde{f}:f^{-1}\lambda^{-1}\mathcal O_{\mathbb X}\to \lambda^{-1}\mathcal O_{\mathbb Y}$. This induces the desired isomorphism of torsors because this map is linear over $\nu^\ast\wtOm_X^\vee\to f^{\ast}\nu^\ast\wtOm_X^\vee$.
		\end{proof}
	
	\begin{Definition}\label{d:LXB}
		Let $X$ be a smooth rigid space, let $\mathbb X$ be a flat $\BdR^+/\xi^2$-lift of $X$ and let $B$ be a coherent $\O_X$-module. For any  $\theta\in H^0(X,B\otimes \wtOm)$, or equivalently any morphism $\theta:\wtOm^\vee\to B$, we obtain 
		an associated pro-\'etale torsor under $\mathcal B:=\nu^{\ast}B$ on $X_{\proet}$ as the pushout along $\nu^\ast\theta$,
		\[ L_{\mathbb X,B,\theta}:= L_{\mathbb X}\times^{\nu^\ast\wtOm^\vee}\B.\]
		In particular, for any  $\Sym^\bullet_{\O_X}\wtOm^\vee$-algebra $B$ on $X_{\et}$ that is coherent over $\O_X$, we obtain such a torsor associated to the composition $\wtOm^\vee\to\Sym^\bullet_{\O_X}\wtOm^\vee\to B$. We simply denote it by $L_{\mathbb X,B}$.
	\end{Definition}
	\begin{Proposition}\label{p:splitting-induced-by-HT-torsor}
		Let $X$ be a smooth rigid space, let $\mathbb X$ be a flat $\BdR^+/\xi^2$-lift of $X$ and let $B$ be a coherent $\O_X$-module. Sending any $\theta\in H^0(X,B\otimes\wtOm)$  to the class of $L_{\mathbb X,B,\theta}$ defines a natural section
		\[ s_{\mathbb X,B}:H^0(X,B\otimes\wtOm)\to H^1_{\proet}(X,\B)\]
		of the Hodge--Tate map $\HT_B$ from \Cref{c:Leray-seq-for-B^x}.1. The map  $s_{\mathbb X,B}$ is functorial in $\mathbb X$ and $B$.
	\end{Proposition}
	\begin{proof}
		It is clear from the construction that $(B,\theta)\mapsto L_{\mathbb X,B,\theta}$ is functorial in $B$. To see that $\HT_B(s_{\mathbb X,B}(\theta))=\theta$, we may thus reduce to the universal case $B=\wtOm^\vee$, $\theta=\id\in H^0(X,\wtOm\otimes \wtOm^\vee)$. 
		
		The construction is moreover functorial in $\mathbb X$. In particular the statement is local, so we may reduce to the case that $X=\Spa(R)$ is toric. We can then make $L_{\mathbb X}$ explicit, as follows: 
		
		Fix a toric chart $X\to \mathbb T^d$ and denote by $T_1,\dots,T_d$ the induced coordinates on $X$.  We then have the standard basis $\tfrac{dT_1}{T_1},\dots,\tfrac{dT_d}{T_d}$ of $\Omega_R^1$. Let $\partial_1,\dots,\partial_d\in \Omega_R^{1\vee}$  be its dual basis. Let $X_\infty=\Spa(R_\infty)\to X$ be the toric cover.
	For any lift $\mathbb X=\Spa(\wt R)$ of $X$, there exists by formal smoothness a lift of the map $K\langle T_1^\pm,\dots,T_d^\pm\rangle\to R$ induced by the chart to an \'etale morphism
	\[\BdR^+/\xi^2\langle T_1^\pm,\dots,T_d^\pm\rangle\to \wt R\]
	 of $\BdR^+/\xi^2$-algebras.
	Any choice of such a lift induces a section of $L_{\mathbb X}(X_\infty)$, as follows:  The morphism
$\BdR^+/\xi^2\langle T_1^\pm,\dots,T_d^\pm\rangle\to \mathbb B^+_{\dR}/\xi^2(X_\infty)$, $T_i\mapsto [T_i^{1/p^\infty}]$
	extends by formal \'etaleness to a unique morphism
	\[\wt\varphi:\widetilde{R}\to \mathbb B^+_{\dR}/\xi^2(X_\infty)\]
	lifting the map $R\to R_\infty$. Let $\Delta$ be the Galois group of $X_\infty \to X$. We write $c_i:\Delta\to \Z_p(1)=\varprojlim_{n\in\N} \mu_{p^n}(K)$ for the map determined by saying that for any $\gamma\in \Delta$, we have
	\[ \gamma \cdot [T_i^{1/p^\infty}]= [c_i(\gamma)]\cdot [T_i^{1/p^\infty}].\]
	
	\begin{Lemma}\label{l:explicit-description}
		Under the identification $L_{\mathbb X}(X_\infty)=\Omega^{\vee}_R\otimes_RR_\infty(1)$ induced by the section $\wt\varphi$, the $\Delta$-action on the right is given by the continuous 1-cocycle
		\[\textstyle \Delta\to \wtOm_R^\vee=\Omega_R^{1\vee}(1),\quad \gamma\mapsto \sum_{i=1}^d c_i(\gamma)\cdot \partial_i.\]
	\end{Lemma}
	\begin{proof}
		Let $R_\infty\{1\}:=\ker(B^+_{\dR}/\xi^2(R_\infty)\to R_\infty)$.
		For any $\sum a_i\partial_i\in \Hom(\Omega_R,R_\infty\{1\})$, the corresponding element of $L_{\mathbb X}(R_\infty)$ is uniquely characterised (by formal \'etaleness) by saying that it sends $T_i\mapsto [T_i^{1/p^\infty}]+a_i$. In order to describe the $\gamma$-action, we thus need to compute
		\[ \gamma([T_i^{1/p^\infty}]+a_i)-[T_i^{1/p^\infty}] =(c_i(\gamma)-1)[T_i^{1/p^\infty}]+\gamma a_i\]
		where $c_i(\gamma)-1\in R_\infty\{1\}$. Indeed, recall that for any $\varepsilon \in \Z_p(1)$ and $a\in \Z_p$, we have
		$\varepsilon^{a}-1\equiv a(\varepsilon-1)\bmod \xi^2$, which induces a canonical isomorphism $R(1)\to R\{1\}, r \varepsilon \to r(\varepsilon-1)$. Under this identification, we see that the cocycle has the desired description.
	\end{proof}
	It now suffices to consider for any $\frac{dT_i}{T_i}\in H^0(X,\Omega^1)$ with associated map $\wtOm^\vee\to \O(1)$ the $\O(1)$-torsor $L_{\mathbb X}\times^{\nu^\ast\wtOm^\vee} \O(1)$. We need to see that its image under $\HT(1)$ is $\frac{dT_i}{T_i}$. By \Cref{l:explicit-description}, the associated cocycle $\Delta\to \wtOm^\vee_R\to R(1)$ is of the form $\gamma\mapsto c_i(\gamma)$.
	The  statement now follows from \Cref{l:explicit-description} by the characterisation of the map $\HT(1):H^1_{\proet}(X,\O(1))\to H^0(X,\Omega)$  given in \cite[Lemma 3.24]{Scholze2012Survey}: Indeed, sending $T_i\in \O_{X_\et}^\times$ around the bottom left corner of the diagram described in the cited lemma defines $\HT(1)^{-1}(\frac{dT_i}{T_i})\in H^1_{\proet}(X,\O(1))$. Going around the top right yields the class defined by the 1-cocycle $\Delta\to R(1)$, $\gamma\mapsto c_i(\gamma)$.
	\end{proof}
	
	\begin{Remark}\label{r:Faltings-extension}
		Assume that $X$ has a model $X_0$ over a discretely valued field $L\subseteq K$ with perfect  residue field and that $\mathbb X=X_0\times_{L}\BdR^+/\xi^2$. Then one can prove that the associated class in $H^1_{\proet}(X,\nu^\ast\wtOm_X^{\vee})$ of $L_{\mathbb X}$, considered as an extension of $\O$ by $\nu^\ast\wtOm_X^{\vee}$, is dual to the Faltings extension
		\[ 0\to \O\to E\to \nu^\ast\wtOm_X^{1}\to 0\]
		on $X_\proet$ from \cite[Corollary 6.14]{Scholze_p-adicHodgeForRigid}. Indeed,
		locally on $X$, this can be seen by comparing the proof of \Cref{l:explicit-description} to that of \cite[Lemma~3.24]{Scholze2012Survey}.
		
		In fact, more generally, one can use $L_{\mathbb X}^\vee$ to define a Faltings extension without assuming existence of a model $X_0$, instead depending on a lift $\mathbb X$. See \cite[\S2.2]{Wang-Simpson} for a related discussion.
	\end{Remark}
	
	\subsection{The partial splitting}\label{s:partial-splitting}
	We can now also construct the partial splitting.
	\begin{proof}[Proof of \Cref{t:rel-vPic-spectral-var}.2]
		We define $\B^+$ like after \Cref{l:coherent-integral-sumodule} and consider the composition
		\[ H^1_\proet(X,2p\B^+)\to  H^1_\proet(X,\B)\xrightarrow{\HT} H^0(X,\wtOm_X\otimes B).\]
		The image of the first map is an $\O_K$-submodule $H^+$ of the finite dimensional $K$-vector space $H^1_\proet(X,\B)$ such that $H^+\tf = H^1_\proet(X,\B)$. Hence it contains an open $\O_K$-sublattice $H_0$. We may choose a lift $H_0\to  H^1_\proet(X,2p\B^+)$ to regard $H_0$ as a submodule of $H^1_\proet(X,2p\B^+)$.
		
		Let now $s:H^0(X,\wtOm_X\otimes B)\to H^1_\proet(X,\B)$ be any splitting, e.g.\ by \Cref{p:splitting-induced-by-HT-torsor} this could be induced by the Higgs--Tate torsor $L_{\mathbb X}$ for any choice of a $\BdR^+/\xi^2$-lift $\mathbb X$ of $X$. Let $\Lambda\subseteq H^0(X,\wtOm_X\otimes B)$ be the preimage of $H_0$ under $s$ and set $\Gamma:=s(\Lambda)\subseteq H_0\subseteq H^1_\proet(X,2p\B^+)$. Summarising the construction, $\Lambda$ and $\Gamma$ are finite free $\O_K$-modules and we have a diagram
		\begin{equation}\label{eq:diag-expl-loc-splitting}
			\begin{tikzcd}[column sep = 0.2cm]
				\Gamma\otimes_{\O_K} \G_a^+ \arrow[d] \arrow[r] \arrow[r] & {H^1_\proet(X,2p\B^+)\otimes\G_a^+} \arrow[d,"\HT"] \arrow[r] & R^1\pi_{\proet*}(2p\B^+) \arrow[d,"\HT"] \arrow[r,"\exp"] & R^1\pi_{\proet*}\B^\times \arrow[d,"\HTlog"] \\
				\Lambda\otimes_{\O_K} \G_a^+ \arrow[r, hook] \arrow[u, bend left,"s"]                           & {H^0(X,\wtOm\otimes B)\otimes \G_a} \arrow[r,equal]                 & {H^0(X,\wtOm\otimes B)\otimes \G_a} \arrow[r,equal]        & {H^0(X,\wtOm\otimes B)\otimes \G_a.}       
			\end{tikzcd}
		\end{equation}
		The composition of $s$ with the top row defines a splitting of $\HTlog$ over $\Lambda\otimes \G_a^+$.
	\end{proof}

	\subsection{Right-exactness via the exponential map}
	In order to prove \Cref{t:rel-vPic-spectral-var}.1, it remains to see that the following natural morphism is surjective:
	\[ \uPic_{\B,\proet}\to H^0(X,B\otimes \wtOm_X)\otimes \G_a.\]
	For this, we will give a geometric argument, using the following version of Proper Base Change:
	
	\begin{Proposition}\label{p:proper-bc}
		Let $g:Z\to \Spa(K)$ be any proper rigid space over $K$. Then for any $n,m\in \N$, we have an isomorphism of sheaves on $\SmRig_{K,\et}$
		\[ R^mg_{\et\ast}(\Z/n\Z)=\underline{H^m_{\et}(Z,\Z/n\Z)}.\]
	\end{Proposition}
	\begin{proof}
		Let $h:Y\to \Spa(K)$ be any object in $\SmRig_{K,\et}$ and let $g':Y\times Z\to Y$ be the projection. Then we need to see that on $Y_\et$, we have $R^mg'_{\et\ast}(\Z/n\Z)=h^\ast R^mg_{\ast}(\Z/n\Z)$.		This is an instance of the rigid analytic Proper Base Change from \cite[Thm.~3.15]{BhattHansen_Zar_constructible}.
	\end{proof}
	In the following, let us write $\pi':X'\to \Spa(K)$ for the structure map of $X'$. For simplicity, we also denote by $\pi'$ the induced morphism of sites $\SmRig_{X,\et}\to \SmRig_{K,\et}$.
	\begin{Corollary}\label{l:[p]-on-Pic-fet}
		If $\uPic_{X'}=R^1\pi'_{\ast}\O^\times$ is representable, then $[p]:\uPic^0_{X'}\to \uPic^0_{X'}$ is finite \'etale. More precisely, it is an \'etale torsor under $\underline{H^1_\et(X',\mu_{p})}$.
	\end{Corollary}
	\begin{proof}
		The Kummer sequence induces an exact sequence
		\[ 1\to R^1\pi'_{\et\ast}\mu_{p}\to \uPic_{X'}\xrightarrow{[p]} \uPic_{X'}\to R^2\pi'_{\et\ast}\mu_{p}.\]
		By \Cref{p:proper-bc},
		the last map goes from a rigid group to a constant group, so it sends $\uPic^0_{X'}$ to $0$. Hence $[p]$ is surjective on $\uPic^0_{X'}$. It is also \'etale by \cite[Lemme 1]{Fargues-groupes-analytiques} as it induces multiplication by $p$ on tangent spaces. By \cite[Lemme 5]{Fargues-groupes-analytiques}, it remains to see that $\ker[p]$ is finite. This follows from the fact that $H^1_\et(X',\mu_p)$ is finite by \cite[Thm 3.17]{Scholze2012Survey}.
	\end{proof}
	
	Recall that the left-exact sequence of \Cref{t:rel-vPic-spectral-var} arises from the Leray sequence for the projection $X_\proet\to X_{\et}$. It can therefore be continued to a 4-term exact sequence on $\SmRig_{K,\et}$
	\[ 	0\to R^1\pi_{\et \ast}B^\times\to  \uPic_{\B,\proet} \xrightarrow{\HTlog} H^0(X,B\otimes \wtOm_X)\otimes \G_a\xrightarrow{\partial} R^2\pi_{\et \ast}B^\times,\]
	where again we write $\pi:\SmRig_{X,\et}\to \SmRig_{K,\et}$ for the morphism of sites induced by $X\to \Spa(K)$.
	We will show that the boundary map $\partial$ vanishes. For this we use the following description:
	
	\begin{Lemma}\label{l:et-LB-X'xY}
		Let $f:Y'\to Y$ be a finite morphism of rigid spaces over $K$. Then for $n\in\N$,
		\[ H^n_\et(Y,f_{\et\ast}\O^\times)=H^n_\et(Y',\O^\times).\]
	\end{Lemma}
	\begin{proof}
		The functor $f_{\et\ast}$ is exact by \cite[Prop.~2.6.3]{huber2013etale}, hence $Rf_{\et\ast}\O^\times=f_{\et\ast}\O^\times$.
	\end{proof}
	After sheafifying in the rigid space $Y$, we can thus identify $\partial$ with a map
	\[  \partial:A:=H^0(X,B\otimes \wtOm)\otimes_K \G_a\to R^2\pi'_{\et\ast}\O^\times.\]
	
	By the splitting of $\HTlog$ over an open subgroup of $ A$ from \S\ref{s:partial-splitting}, we know that any $x\in A(Y)$ is in the kernel of $\partial$ after multiplying by $p^n$ for some $n$. Hence $\partial$ factors through a map
	\[\partial':A=H^0(X,B\otimes \wtOm)\otimes \G_a\to R^2\pi'_{\et\ast}\O^\times[p^\infty].\]
	We claim that any such homomorphism vanishes. To see this, we consider the Kummer sequence on $X'_{\et}$. Using \Cref{p:proper-bc}, we see that this induces a long exact sequence
		\[ \uPic_{X'}\to  \underline{H^2_{\et}(X',\mu_{p^n})}\to R^2\pi'_{\et\ast}\O^\times\xrightarrow{p^n} R^2\pi'_{\et\ast}\O^\times.\] 
		Taking the colimit over $n$, it follows that we have an \'etale surjection
		\[\underline{H^2_{\et}(X',\mu_{p^\infty})}\rightarrow R^2\pi'_{\et}\O^\times[p^\infty].\]
		
		\begin{Lemma}\label{l:surjection-from-locally-constant}
			Let $Q$ be a sheaf on $\SmRig_{K,\et}$ that admits a surjection $\underline{H}\to Q$ from a locally constant sheaf. Then any map $h:Z\to Q$ from a connected rigid space $Z$ is constant.
		\end{Lemma}
		\begin{proof}
			That $h$ is an \'etale surjection means  that there is an \'etale cover $Z'\to Z$ by a rigid space and a map $h'$ fitting into a commutative diagram
		\[
		\begin{tikzcd}[row sep = 0.55cm]
			Z' \arrow[d,"h'"] \arrow[r] & Z \arrow[d,"h"] \\
			{\underline{H}} \arrow[r]           & Q.       
		\end{tikzcd}\]
		Then $h'$ is locally constant by the assumption on $\underline{H}$. Since the \'etale map $Z'\to Z$ is open  \cite[Prop.~1.7.8]{huber2013etale}, it follows that $h$ is constant locally on $Z$, hence constant as $Z$ is connected.
		\end{proof}
		Applying this to the  homomorphism $\partial'$ from the connected rigid group variety  $A$, we deduce that $\partial'=0$, hence $\partial=0$.
		This finishes the proof of \Cref{t:rel-vPic-spectral-var}.1 \qed
	
	\begin{Remark}
		Evaluating at $K$-points, we deduce that there is a short exact sequence
		\[0\to \Pic(X')\to H^1_{\proet}(X,\B^\times)\to H^0(X,B\times\wtOm)\to 0.\]
		Note, however, that our proof crucially uses that this sequence can be upgraded to sheaves as it relies on a geometric argument to see the vanishing of the boundary map on the right.
	\end{Remark}
	\subsection{Representability of $\uPic_{\B,\proet}$}
	We can now complete the proof of \Cref{t:rel-vPic-spectral-var}:
	\begin{proof}[Proof of \Cref{t:rel-vPic-spectral-var}, parts 3,4,5]
		Assume that $\uPic_{X'}=R^1\pi_{\et\ast}B^\times$  is representable by a rigid group. Then the same argument as for \cite[Cor.~2.9]{heuer-diamantine-Picard} shows that also $R^1\pi_{\proet\ast}\B^\times$ is representable: Let $\Lambda$ be as in \S\ref{s:partial-splitting} and for any $n\in \N$, let $U_n:=\HTlog^{-1}(p^{1-n}\Lambda\otimes \G_a^+)\subseteq \uPic_{\B,\proet}$, then
		\[ \textstyle\uPic_{\B,\proet}=\bigcup_{n\in \N}U_n.\]
		Since the short exact sequence is split over $\Lambda$, the sheaf $U_1\cong  \Lambda\otimes \G_a^+\times \uPic_{X'}$ is representable.
		
		The general case follows inductively: We have a morphism of short exact sequences
		\begin{equation}\label{eq:ind-proof-Pic_B,proet-rep}
			\begin{tikzcd}[row sep = 0.50cm]
			0 \arrow[r] & {\uPic_{X',\et}} \arrow[d, "{[p]}"] \arrow[r] & {U_n} \arrow[d, "{[p]}"] \arrow[r] & p^{1-n}\Lambda \arrow[d, "{\cdot p}","\sim"labelrotate] \arrow[r] & 0 \\
			0 \arrow[r] & {\uPic_{X',\et}} \arrow[r] & {U_{n-1}} \arrow[r] & p^{1-(n-1)}\Lambda\arrow[r] & 0.
		\end{tikzcd}
		\end{equation}
		The morphism on the left is locally on the target a finite \'etale $H^1(X',\mu_{p^n})$-torsor by \Cref{l:[p]-on-Pic-fet}, hence the same is true for $U_n\to U_{n-1}$. Any such torsor is representable: For example, we can argue in v-sheaves and use that $U_{n-1,\et}=U^\diamondsuit_{n-1,\et}$ by \cite[Lemma~15.6]{etale-cohomology-of-diamonds} to see that any finite \'etale torsor is representable in $U_{n-1,\et}$, hence also on $\SmRig_{K,\et}$.
		
		\medskip
		
		If $X$ is algebraisable, then so is the coherent $\O_X$-module $B$ by rigid GAGA, and it follows that $X'$ is algebraisable, i.e.\ there is a proper scheme $X'_0\to \Spec(K)$ such that $X'=X'^{\an}_0$. As explained in \S\ref{s:2.2}, it follows that $\uPic_{X'}$ is  representable  in this case.
		
		Part (5) follows from the above diagram \eqref{eq:ind-proof-Pic_B,proet-rep} and \Cref{l:[p]-on-Pic-fet}.
	
		Finally, for part (4), we first note the following well-known fact:
		\begin{Lemma}
			If $\uPic_{X'}$ is representable, then $\Lie\uPic_{X'}=H^1_{\an}(X',\O)=H^1_{\an}(X,B)$.
		\end{Lemma}
		\begin{proof}
			This is standard if the Picard functor is defined on all rigid spaces by testing on $K[X]/X^2$. But the statement is still true if we only know representability on $\SmRig_{K,\et}$: We first note that for any commutative rigid group $G$, \cite[Cor.~4]{Fargues-groupes-analytiques} implies that we have
			\begin{equation}\label{eq:Lie-as-Hom}
			\Lie G=\Hom(\G_a^+,G)\tf.
			\end{equation}
			Second, there is a natural map 
			$H^1(X',\O)\to \Hom(\G_a^+,\uPic_{X'})\tf$
			induced by inverting $p$ on the map  $\exp:R^1\pi'_{\et\ast}2p\O^+\to R^1\pi'_{\et\ast}\O^\times$. Conversely, any homomorphism $\G_a^+\to \uPic_{X'}$ defines an element in $H^1(X',1+\m \O^+)$ by \cite[Satz~1 and 2]{Gerritzen-ZerlegungPicard}, whose image under $\log:(1+\m \O^+)\to \O$ defines an element in $H^1(X',\O)$. Thus $ \Hom(\G_a^+,\uPic_{X'})\tf=H^1(X',\O)$.
		\end{proof}
		Part (4) about tangent spaces now follows from the local splitting constructed in the diagram \eqref{eq:diag-expl-loc-splitting}, and the relation to the Hodge--Tate sequence explained by \Cref{p:splitting-induced-by-HT-torsor}.
	\end{proof}
	\section{Invertible $\B$-modules via the exponential}
	We continue with the setup of \Cref{t:rel-vPic-spectral-var}, that is, $X$ is a smooth proper rigid space, $B$ is a coherent $\O_X$-algebra on $X_{\et}$ and $\mathcal B=\nu^{\ast}B$ on $X_{\proet}$. From \S3.2 onwards, we will additionally assume that $B$ is $\O_X$-torsionfree: This assumption is harmless for our applications because $B$ will always live inside the $\O_X$-torsionfree module $\uEnd(E)$ for some vector bundle $E$ on $X$.
	
	The first theme of \S3 is to relate invertible $\B$-modules to torsors under the additive group $\B$ via the exponential. The relation is furnished by the exponential map
	$\exp:2p\B^+\to \B^\times$ from \Cref{l:approx-prop-for-B^+/p}.(3)
	which by applying $-\otimes_{\Z}\Z\tf$ on the level of sheaves of groups yields a natural map
	\[\textstyle\exp:\B\to \B^\times\tf :=\varinjlim_{x\mapsto x^p}\B^\times.\]
	In \S3.1, we will explain how this map can be used to construct reductions of structure groups of $\uPic_{\B,\proet}$ which are always representable. Once this is achieved, our goal will be to use $p$-adic exponentials of rigid groups in order to functorially exponentiate the Higgs--Tate torsor to obtain invertible $\B$-modules, splitting $\HTlog$ on $K$-points. 
	In order to make this exponentiation functorial, our main task in this section will be to equip  invertible $\B$-modules with extra data that allow us to identify all objects in an isomorphism class up to \textit{unique} isomorphism. This is done in \S3.2 by introducing an appropriate notion of rigidifications.

	\subsection{Reduction of structure groups}\label{s:red-str-grps}
		Motivated by the map $\exp:\B\to \B^\times\tf$, we now pass from invertible $\B$-modules to $\B^\times\tf$-torsors. 
	Our starting point is  the functor $\uPic_{\B,\proet}\tf$. Even if $\uPic_{\B,\proet}$ is representable, this is typically no longer represented by a rigid group. But using that $X$ is quasi-compact, we can still regard it as a moduli functor of isomorphism classes of pro-\'etale $\B^\times\tf$-torsors. Since any rigid vector group is uniquely divisible, we still have a morphism
	\[ \HTlog:=\HTlog\tf:\uPic_{\B,\proet}\tf\to \mathcal A_B:=H^0(X,\wtOm_X\otimes B)\otimes \G_a.\]
	We now explain the relation to torsors under the additive group $\B$:
	For any pro-\'etale $\B$-torsor $M$, let us write 
	\[M^{\exp}:=M\times^\B\B^\times\tf\]
	for the pushout along the exponential. We apply this to the following $\B$-torsor: 
	\begin{Definition}\label{d:L_XB-e_X}
	Recall from \Cref{d:Higgs-Tate} that the datum of a $\BdR^+/\xi^2$-lift $\mathbb X$ induces a $\nu^\ast\wtOm_X^{\vee}$-torsor $L_{\mathbb X}$ on $X$. By pushout along the tautological map $\nu^\ast\wtOm_X^{\vee}\to \B$ over $\mathcal A_B=H^0(X,B\otimes \wtOm_X)\otimes \G_a$, we obtain from \Cref{d:LXB} a pro-\'etale $\B$-torsor $L_{\mathbb X,B}$ over $X\times \mathcal A_B$. 
	Sending the pro-\'etale $\B$-torsor $L_{\mathbb X,B}$ to
	$L^{\exp}_{\mathbb X,B}$
	on $X\times \mathcal A_B$ now defines  a map
	\[ e_{\mathbb X}:\mathcal A_B\to \uPic_{\B,\proet}\tf \]
	which is a natural splitting
	of $\HTlog\tf$ by \Cref{p:splitting-induced-by-HT-torsor}. All in all, we obtain a commutative diagram
	\begin{equation}\label{eq:def-e_X}
		\begin{tikzcd}[column sep = 2.5cm]
			\uPic_{\B,\proet}\tf \arrow[rd, "{\HTlog[\frac{1}{p}]}" description]              &                                           \\
			\uPic_{\B,\proet} \arrow[r, "\HTlog" description] \arrow[u,"\can"] & \mathcal A_B. \arrow[lu, "e_{\mathbb X}"', bend right]
		\end{tikzcd}
	\end{equation}
	where $\can:	\uPic_{\B,\proet} \to \uPic_{\B,\proet}\tf $ is the canonical map. 
	\end{Definition}
	
	In particular, this shows that the $\uPic_{X'}$-torsor 	$\uPic_{\B,\proet}\to \mathcal A_B$ of \Cref{t:rel-vPic-spectral-var} is split after inverting $p$, or in other words that the associated element in $H^1_{\et}(\mathcal A_B,	\uPic_{X'}\tf)$ is trivial. It follows formally, using long exact sequences, that $\uPic_{\B,\proet}$ admits a reduction of structure group to  $\uPic_{X'}[p^\infty]$. In fact, a more canonical construction is possible using the lift $\mathbb X$:
	\begin{Theorem}\label{t:red-structure-group}
		Let $X$ be a smooth proper rigid space over $K$. Let $B$ be a coherent $\O_X$-algebra and set $\B:=\nu^\ast B$.
		Let $\mathbb X$ be a $\BdR^+/\xi^2$-lift of $X$. Consider the abelian sheaf on $\SmRig_{K,\et}$
		\[
		\begin{tikzcd}[column sep = 1.5cm]
			\mathcal P_{\mathbb X}=\mathcal P_{\mathbb X,B}:=\mathrm{eq}\Big(\uPic_{\B,\proet}\arrow[r, "\can", shift left] \arrow[r, "e_{\mathbb X}\circ \HTlog"', shift right] & \uPic_{\B,\proet}\tf\Big).
		\end{tikzcd}\]
		Then the natural maps define a morphism of short exact sequences
		\[\begin{tikzcd}
			0\arrow[r]& \uPic_{X'}[p^\infty]\arrow[r]\arrow[d]& \mathcal P_{\mathbb X}\arrow[r,"\HTlog"]\arrow[d]& \mathcal A_B\arrow[r]\arrow[d,equal]& 0\\
			0\arrow[r]& \uPic_{X'}\arrow[r]& \uPic_{\B,\proet}\arrow[r,"\HTlog"]& \mathcal A_B\arrow[r]& 0.
		\end{tikzcd}\]
		The first sequence is always representable by a short exact sequence of rigid groups. In particular, we have $\Lie\mathcal P_{\mathbb X} =\mathcal A_B(K)$. 
		The connected component of the identity $\mathcal P^0_{\mathbb X}$ is $p$-divisible and the morphism $\HTlog:\mathcal P^0_{\mathbb X}\to  \mathcal A_B$ is still surjective.
	\end{Theorem}
	\begin{proof}
		We first prove left-exactness:
		Let $x,y \in \mathcal P_{\mathbb X}(Y)$ be any sections over $Y\in \SmRig_{K,\et}$.
		Since $ e_{\mathbb X}$ is a section of $\HTlog\tf$, it is in particular injective. We thus have equivalences
		\[ \HTlog(x)=\HTlog(y)\Leftrightarrow e_{\mathbb X}\circ \HTlog(x)=e_{\mathbb X}\circ \HTlog(y)\Leftrightarrow \can(x)=\can(y)\]
		\[\Leftrightarrow x \cdot y^{-1}\in \uPic_{\B,\proet}[p^\infty]=\ker(\can).\]
		Considering multiplication by $p$ on the sequence of \Cref{t:rel-vPic-spectral-var}, we see that
		\[\uPic_{\B,\proet}[p^\infty]=\uPic_{X'}[p^\infty].\]
		This gives the desired left-exact sequence.
		
		To see the right-exactness, we use that by considering the cokernel of $[p]$ on the sequence of \Cref{t:rel-vPic-spectral-var}, and using \Cref{p:proper-bc}, there is a short exact sequence
		\[\uPic_{\B,\proet}\xrightarrow{\can} \uPic_{\B,\proet}\tf\to \underline{H^2_{\et}(X',\mu_{p^\infty})}.\]
		Regarding $e_{\mathbb X}$ as an element of $\uPic_{\B,\proet}\tf(\mathcal A_B)$, we see that its image in the third term is a homomorphism $\mathcal A_B\to \underline{H^2_{\et}(X',\mu_{p^\infty})}$ which has to be trivial since $\mathcal A_B$ is connected. It follows that there is an \'etale cover $\mathcal A'\to \mathcal A_B$ over which $e_{\mathbb X}$ lifts to a map $s:\mathcal A'\to \uPic_{\B,\proet}$.  Then by construction, and using that $\HTlog= \HTlog\tf  \circ \can$, we see that
		\[ e_{\mathbb X}\circ\HTlog(s)=e_{\mathbb X}\circ\HTlog\tf  \circ \can(s)=e_{\mathbb X}\circ\HTlog\tf \circ e_{\mathbb X}=e_{\mathbb X} =\can(s).\]
		Hence $s\in \mathcal P_{\mathbb X}(\mathcal A')$, and we see that $\HTlog:\mathcal P_{\mathbb X}\to \mathcal A_B$ is surjective.
		
		To see the representability, we first note that by the Kummer sequence and \Cref{p:proper-bc}, we see that
		 \[\uPic_{X'}[p^\infty]=R^1\pi'_{\ast}\mu_{p^\infty}=\underline{H^1_{\et}(X',\mu_{p^\infty})}\]
		 is representable by a locally constant rigid group. It follows that $\mathcal P_{\mathbb X}\to \mathcal A_B$ is \'etale in $\mathcal A_B^\diamondsuit$, and we deduce that it is representable by a rigid space. The statement about Lie algebras is immediate from $\Lie(\uPic_{X'}[p^\infty])=0$.
		 
		Finally, to see the claim regarding $\mathcal P^0_{\mathbb X}$, we consider multiplication by $p^n$ on the short exact sequence of the first part. Using \Cref{p:proper-bc}, we obtain an exact sequence
		\[ \mathcal P_{\mathbb X}\xrightarrow{p} \mathcal P_{\mathbb X}\to \underline{H^2_{\et}(X',\mu_{p})}.\]
		Let $N:=\mathcal P_{\mathbb X}/\mathcal P_{\mathbb X}^\circ$ be the group of connected components, then $Q:=\coker(\mathcal P^0_{\mathbb X}\xrightarrow{p} \mathcal P^0_{\mathbb X})$ sits in an exact sequence
		\[ \underline{N[p]}\to Q\to \underline{H^2_{\et}(X',\mu_{p})}.\]
		Using \Cref{l:surjection-from-locally-constant}, we deduce that the map $\mathcal P^0_{\mathbb X}\to Q$ is trivial, hence $\mathcal P^0_{\mathbb X}$ is $p$-divisible.
		
		Finally, since $\HTlog:\mathcal P^0_{\mathbb X}\to \mathcal A_B$ is still locally split over an open subgroup of $\mathcal A_B$, it follows from the $p$-divisibility of $\mathcal P^0$ that the map $\mathcal P^0_{\mathbb X}\to \mathcal A_B$ is still surjective.
	\end{proof}
	We thus get a smaller moduli space of invertible $\B$-modules which is always representable, at the expense of using the choice of $\mathbb X$.
	This will allow us to circumvent the question whether $\uPic_{X'}$ is representable while still obtaining a rigid moduli space of invertible $\B$-modules.

	\subsection{Rigidifying invertible $\B$-modules}
	While the rigid group $G$ representing the functor $\uPic_{\B,\proet}$ is unique up to unique isomorphism if it exists, we have stopped short of giving a canonical universal $\B$-module on $(X\times G)_\proet$. Instead, we only get a universal class in $H^1_\proet(X\times G,\B^\times)$ yielding an \textit{isomorphism class} of such objects. For Picard functors, when $X$ is connected, the usual way to rectify this issue is to add a rigidification at a base point $x\in X(K)$ to the moduli problem: For  any $S\in \SmRig_{K}$ and any sheaf $L$ on $(X\times S)_\proet$, we denote by $L_{|x\times S}$ (or $L_{|x}$ when $S=\Spa(K)$) the sheaf on $S_\proet$ obtained by pullback of $L$ along $x\times \id:S\to X\times S$. Then in our case, a rigidification at $x$ would be an isomorphism 
	\[ \alpha:L|_{x\times S}\isomarrow B_{|x} \otimes_K\O_S\]
	on $S_\proet$ for any invertible $\B$-module $L$ on $X\times S$. Here we regard the fibre $B_{|x}$ of $B$ at $x$ simply as a $K$-vector space.
	This gives a good notion of rigidifications when $B=\O_X$ and $X$ is connected, for the same reason that this works for the classical Picard functor: We have 
	\[\mathrm{Aut}_{\O_X}( L)=(\O(X)\otimes_K \O(S))^\times=(\O_{X|x}\otimes_K \O(S))^\times=\Aut_{\O_{X|x}}( L_{|x\times S}),\]
	hence the automorphisms of  $L$ correspond one-to-one to the automorphisms of $L|_{x\times S}$. Thus an isomorphism class of rigidified line bundles determines a line bundle up to unique isomorphism.
	
	However, this simple approach no longer works for general $B$, as the map
	$B(X)\to B_{|x}$
	is in general neither injective nor surjective. The failure to be injective means that $ L_{|x\times S}$ cannot detect all automorphisms of $ L$, which is problematic for defining moduli problems.

	The goal of this section is to solve this issue	by defining a more elaborate notion of rigidifications for invertible $\B$-modules in  $\mathcal P_{\mathbb X}$. This follows an idea suggested to us by Gerd Faltings.

	\begin{Lemma}\label{l:finite-K-alg-units}
		Let $A$ be a finite $K$-algebra such that $\Spec(A)$ is connected. Let $\eta(A)$ be the nilradical. Then for any $K$-algebra $R$, there is a canonical and functorial decomposition
		\[  (A\otimes_K R)^\times\cong (\eta(A)\otimes_K R)\times R^\times.\]
	\end{Lemma}
	\begin{proof}
		Since $K$ is algebraically closed, there is a surjective homomorphism of $K$-algebras $A\to K$ splitting the structure map. We thus have a natural exact sequence of $A$-modules
		\[ 0\to \eta(A)\to A\to K\to 0 \]
		that is functorially split.
		After tensoring with $R$ and
		passing to units, this induces a split exact sequence
		\[ 0\to 1+\eta(A)\otimes_K R\to (A\otimes_K R)^\times\to R^\times\to 0.\]
		The group isomorphism $\log: 1+\eta(A)\otimes_K R\to \eta(A)\otimes_K R$ yields the desired decomposition.
	\end{proof}
	
	\begin{Lemma}\label{l:map-on-units-split}
		Let $A\hookrightarrow A'$ be an injective homomorphism of finite $K$-algebras.
		\begin{enumerate}
			\item The presheaf $(A\otimes \O)^\times$ on $\SmRig_{K,\et}$ defined by
			\[ (A\otimes \O)^\times:S\mapsto (A\otimes_K\O(S))^\times\]
			is represented by a rigid group of the form $\G_a^r\times \G_m^{s}$ for some $r,s\in \N$.
			\item The following induced morphism of sheaves on $\SmRig_{K,\et}$ 
			is split injective:
			\[ (A\otimes \O)^\times\to (A'\otimes \O)^\times\] Its cokernel is represented by a rigid group of the form $\G_a^r\times \G_m^{s}$ for some $r,s\in \N$.
			\item 
			For any $S\in \SmRig_{K,\et}$, the following natural morphism on $S_\proet$ is split injective: \[(A\otimes_K\O_{S_\proet})^\times\to (A'\otimes_K\O_{S_\proet})^\times\]
		\end{enumerate}
	\end{Lemma}
	\begin{proof}
		Dealing with each connected component of $\Spec(A)$ separately, we can reduce to the case that $\Spec(A)$ is connected. Then the first part is immediate from \Cref{l:finite-K-alg-units}.
		
		For part 2, decompose $A'=\prod_{i=1}^n A'_i$ into $K$-algebras such that each $\Spec(A'_i)$ is connected. By \Cref{l:finite-K-alg-units}, the morphism in question is then the direct sum of
		\[\textstyle \eta(A)\otimes \O\to \prod_{i=1}^n \eta(A'_i)\otimes \O \quad \text{and} \quad \O^\times \to \prod_{i=1}^n \O^\times.\] 
		As $\eta(A)\hookrightarrow \prod_{i=1}^n\eta(A'_i)$ is an injective morphism of $K$-vector spaces, the first map is injective with cokernel isomorphic to $\O^m$ for some $m$. The second is injective with cokernel isomorphic to $(\O^\times)^{n-1}$. Thus the cokernel is represented by $\G_a^m\times \G_m^{n-1}$.
		
		Part 3 can be seen in a similar way, namely this morphism is the direct sum of split maps
		\[\textstyle \eta(A)\otimes \O_{S_\proet}\to \prod_{i=1}^n \eta(A'_i)\otimes \O_{S_\proet} \quad \text{and} \quad \O_{S_\proet}^\times \to \prod_{i=1}^n \O_{S_\proet}^\times.\qedhere\]
	\end{proof}
	
	\begin{Lemma}\label{l:injective-into-fibre}
		Let $B$ be an $\O_X$-torsionfree coherent $\O_X$-module. Then there is a finite set $M\subseteq X(K)$ such that $H^0(X,B)\hookrightarrow \prod_{x\in M}B_{|x}$ is injective, where $B_{|x}$ is the fibre of $B$ over $x$.
	\end{Lemma}
	\begin{proof}
		Let $0\neq s\in H^0(X,B)$ and let $U=\Spa(R)\subseteq X$ be any non-empty affinoid open subspace such that $s_{|U}\neq 0$. Then $B_{|U}$ is associated to a coherent $R$-module $M$. By  \cite[051S]{StacksProject}, we can find $0\neq f\in R$ such that $M_f\cong R^r_f$ for some $r\in \N$. Since $B$ is $\O_X$-torsionfree, $M$ is $R$-torsionfree, hence the image of $s$ in $M_f$ is non-zero. Then the image of $s$ in $B_{|x}$ is non-zero for any $x\in U(f\neq 0)$ away from the joint vanishing locus of the coordinates of $s$ in $R^r_f$.
		
		We now initially set $M=\{x\}$ and apply this argument inductively to non-zero elements in the kernel of $H^0(X,B)\to \prod_{x\in M}B_{|x}$. Enlarging $M$ in this way, this map becomes injective after finitely many steps since $H^0(X,B)$ is finite dimensional by Kiehl's Theorem.
	\end{proof}
	We note that in general, we cannot use $M=\{x\}$ for just one $x\in X$ (a counter-example is given by line bundles on $\P^1$, where $H^0(X,B)$ can have larger dimension than each $B_{|x}\cong K$).
	
	Moreover, 
	the assumption that $B$ is $\O_X$-torsionfree is necessary (for a counter-example, take $X=\P^1$ and $B=\O_X/I^2$ where $I$ is the sheaf of functions vanishing at $0\in \P^1$). 
	\begin{Definition}\label{d:rigidifying-set}
		\begin{enumerate}
			\item 
		A subset $M\subseteq X(K)$ is called $B$-rigidifying if it is finite and satisfies the statement of \Cref{l:injective-into-fibre}. More generally, it will later (in the proof of \Cref{t:inv-B-module-assoc-to-HF}) be convenient to allow multiplicities: We say that a multiset $M$ of elements of $X(K)$ is $B$-rigidifying if $M$ is finite and $ H^0(X,B)\hookrightarrow \prod_{x\in M}B_{|x}$ is injective.

		\item Given a lift $\mathbb X$ of $X$ to $\BdR^+/\xi^2$ and any finite set $\mathbb M\subseteq \mathbb X(\BdR^+/\xi^2)$, we define the reduction of $\mathbb M$ to be the multiset given by the image of $\mathbb M$ under the reduction map $ \mathbb X(\BdR^+/\xi^2)\to X(K)$. Here we count with multiplicities, which is the motivation to allow multisets. We say that $\mathbb M$ is $B$-rigidifying if its  reduction $M$ is  $B$-rigidifying.
		\end{enumerate}
	\end{Definition}
	
	Throughout the rest of this subsection, let $B$ be an $\O_X$-torsionfree coherent $\O_X$-algebra on the smooth proper rigid space $X$ over $K$, and set $\mathcal B:=\nu^{\ast}B$ on $X_{\proet}$.
	
	As a first application of rigidifying sets, we can use them to prove:
	\begin{Proposition}\label{p:explicit-description-of-Pic_B}
		Let $S\in \SmRig_{K}$. With notations as in \Cref{d:Pic_proet}, the natural maps induce isomorphisms
		\[ \uPic_{B,\proet}(S)=H^1_{\proet}(X\times S,\B^\times)/H^1_{\proet}(S,\pi_{S,\proet\ast}\B^\times)=H^0(S,R^1\pi_{S,\proet\ast}\B^\times).\]
	\end{Proposition}
	\begin{proof}
		The Leray sequence of $\pi_{S,\proet}$ yields a 5-term exact sequence of pro-\'etale cohomologies
		\[H^1(S,\pi_{S\ast}\B^\times)\to H^1(X\times S,\B^\times)\to H^0(S,R^1\pi_{S\ast}\B^\times)\to  H^2(S,\pi_{S\ast}\B^\times)\to H^2(X\times S,\B^\times).\]
		It suffices to prove that the last morphism is injective. To this end, recall from \eqref{eq:coh-bc} that $ \pi_{S\ast}\B=H^0(X,B)\otimes_K\O_{S_\proet}$. By \Cref{l:injective-into-fibre}, there is a $B$-rigidifying set $M\subseteq X(K)$. The specialisation at each $x\in M$ then defines a morphism
		\[\textstyle H^2_{\proet}(S,\pi_{S\ast}\B^\times)\to H^2_{\proet}(X\times S,\B^\times)\to \prod_{x\in M} H^2_{\proet}(S,(B_{|x}\otimes_K \O_{S_\proet})^\times).\]
		Setting  $A:=H^0(X,B)$ and $A':=\prod_{x\in M}  B_{|x}$, we deduce from \Cref{l:map-on-units-split}.3 that this composition is split injective. Hence the first map is injective, as we wanted to see.
	\end{proof}
	
	\begin{Lemma}\label{l:QM}
		For any $B$-rigidifying multiset $M$ and any $S\in \SmRig_{K,\et}$, the specialisation map
		\[\textstyle s_M(S):H^0(X\times S,\B)\to \prod_{x\in M}B_{|x}\otimes_K\O(S)\]
		is injective. The presheaf on $\SmRig_{K,\et}$ defined by the cokernel of the associated map on units
		\[ Q^M:S\mapsto \big(\textstyle\prod_{x\in M}B_{|x}\otimes_K\O(S)\big)^\times/ H^0(X\times S,\B^\times),\]
		is a sheaf represented by a rigid group of the form $Q^M\cong \G_a^r\times \G_m^s$ for some $r,s\geq 0$.
	\end{Lemma}
	\begin{proof}
		Let $A:=H^0(X,B)$ and $A':=\prod_{x\in M}B_{|x}$, then \Cref{eq:coh-bc} shows that $s_M$ is of the form
		\[ A\otimes_K\O\to A'\otimes_K\O\]
		for the injective morphism of finite $K$-algebras $A\to A'$.  This shows part 1. Second, $Q^M$ is the cokernel of the units of this map, hence part 2 now follows from \Cref{l:map-on-units-split}.2.
	\end{proof}
	
	To simplify notation in the following, we denote $L_{|x\times S}$ by $L_{|x}$ when $S$ is clear from context.
	
	\begin{Definition}
		Let $M$ be a $B$-rigidifying multiset of $K$-points of $X$. Let $S\in \SmRig_{K,\et}$. Then an $M$-rigidified invertible $\B$-module on $X\times S$ is a pair $\mathcal L=(L,(\alpha_x)_{x\in M})$ consisting of an invertible $\B$-module $L$ on $X\times S$ and for each $x\in M$ an isomorphism of $\mathcal B_{|x}$-modules on $S_\proet$
		\[ \alpha_x: L_{|x}\isomarrow \mathcal B_{|x}.\]
		An isomorphism of $M$-rigidified invertible $\B$-modules $\psi:( L,(\alpha_x)_{x\in M})\isomarrow ( L',(\alpha_x')_{x\in M})$ is an isomorphism of  $\B$-modules $\psi: L\to  L'$ such that $\alpha_x'\circ \psi_{|x}=\alpha_x$ for each $x\in M$.
	\end{Definition}
	\begin{Remark}\label{r:tensor-rigidif-B-mods}
	For any two $M$-rigidified invertible $\B$-modules $\mathcal L=( L,(\alpha_x)_{x\in M})$,  $\mathcal L'=( L',(\alpha_x')_{x\in M})$, the tensor product $\mathcal L\otimes \mathcal L':=( L\otimes_{\mathcal B}  L',(\alpha_x\otimes \alpha'_{x})_{x\in M})$ is again an $M$-rigidified invertible $\B$-module. The ``trivial'' $M$-rigidified $\B$-module $(\mathcal B,(\id\colon\mathcal B_{|x}\to \mathcal B_{|x})_{x\in M})$ defines a neutral element and  $( L,(\alpha_x)_{x\in M})$ has an inverse $( L^{\vee},(\alpha_x^{\vee -1})_{x\in M})$ with respect to $\otimes$.
	\end{Remark}
	
	\begin{Lemma}\label{l:isom-between-rigidified-modules-unique}
		Let $\mathcal L$ and $\mathcal L'$  be two isomorphic $M$-rigidified invertible $\B$-modules on $X\times S$. Then the isomorphism $\mathcal L\to \mathcal L'$ is unique.
	\end{Lemma}
	\begin{proof}
		Let $\beta_1,\beta_2:\mathcal L\to \mathcal L'$ be two isomorphisms of $M$-rigidified invertible $\B$-modules. It suffices to prove that $\beta:=\beta_1^{-1}\circ\beta_2:\mathcal L\to \mathcal L$ is the identity. Write $\mathcal L=(L,(\alpha_x)_{x\in M})$, then the underlying isomorphism $\beta:L\to L$ is given by an element $\beta\in H^0(X\times S,\B^\times)$ subject to the condition that $\alpha_x=\alpha_x\cdot \beta_{|x}$ for all $x$. The latter is equivalent to $\beta_{|x}=1$, hence $\beta$ is in the kernel of $H^0(X\times S,\B^\times)\to \prod_{x\in M}(B_{|x}\otimes_K\O(S))^\times$. This is trivial by \Cref{l:QM}.
	\end{proof}
	
	\begin{Definition}
		\begin{enumerate}
			\item 
			Let $M$ be any $B$-rigidifying multiset. The $M$-rigidified pro-\'etale Picard functor of $B$ is the presheaf
			\[\uPic^M_{\B,\proet}\colon\SmRig_{K,\et}\to \mathrm{Ab}\]
			sending $S$ to  the set of isomorphism classes of $M$-rigidified invertible $\B$-modules on $(X\times S)_\proet$. By \Cref{r:tensor-rigidif-B-mods}, this has a group structure induced by $\otimes$. 
			\item The forgetful functor $(L,(\alpha_x)_{x\in M})\mapsto L$ defines a morphism of presheaves
			\[\uPic^M_{\B,\proet}\to  \uPic_{\B,\proet}.\]
			We denote by $\mathcal P^M_{\mathbb X}\to \mathcal P_{\mathbb X}$ (or $\mathcal P^M_{\mathbb X,\B}$ to indicate $\B$) the fibre of this morphism over the sheaf $\mathcal P_{\mathbb X}\to \uPic_{\B,\proet}$ from \Cref{t:red-structure-group}.
		\end{enumerate}
	\end{Definition}
	\begin{Lemma}
		The presheaf $\uPic^M_{\B,\proet}$ is a sheaf. In particular, $\mathcal P^M_{\mathbb X}$ is a sheaf.
	\end{Lemma}
	\begin{proof}
		If $\varphi:S'\to S$ is a cover in $\SmRig_{K,\et}$ and  $\mathcal L$ is an $M$-rigidified invertible $\B$-module on $X\times S'$ with an isomorphism $\psi:\pi_1^\ast \mathcal L\isomarrow \pi_2^\ast \mathcal L$ over $X\times S'\times_SS'$, then $\psi$ automatically satisfies the cocycle condition by \Cref{l:isom-between-rigidified-modules-unique}, so $\mathcal L$ descends to  $S$.
		
		Similarly, let $\mathcal L$ be any $M$-rigidified invertible $\B$-module on $X\times S$ and assume that there is an isomorphism $\psi:\varphi^{\ast}\mathcal L\isomarrow \mathcal (\B,(\id)_{x\in M})$ over $X\times S'$. Then  $\pi_1^{\ast}\psi=\pi_2^{\ast}\psi$ on $X\times S'\times_SS'$ by \Cref{l:isom-between-rigidified-modules-unique}, hence $\psi$ descends to an isomorphism $\mathcal L\isomarrow (\B,(\id)_{x\in M})$ over $X\times S$.
	\end{proof}
	
	\begin{Theorem}\label{t:rigidifed-PicB}Let $B$ be an $\O_X$-torsionfree coherent $\O_X$-algebra on the smooth proper rigid space $X$ and $\mathcal B=\nu^{\ast}B$ on $X_{\proet}$.
		Let $M$ be any $B$-rigidifying multiset of points of $X(K)$. 
		\begin{enumerate}
			\item  The forgetful map defines a natural short exact sequence of abelian sheaves on $\SmRig_{K,\et}$
			\[ 0\to Q^M\to \uPic^M_{\B,\proet}\to \uPic_{\B,\proet} \to 0\]
			where the first map is given by sending $\alpha=(\alpha_x)_{x\in M}$ to $(\mathcal B,(\alpha_x:B_{|x}\xrightarrow{\cdot \alpha_x} B_{|x})_{x\in M})$. In particular, $\uPic^M_{\B,\proet}$ is representable by a rigid group if $\uPic_{\B,\proet}$ is.
			\item Pullback along $\mathcal P_{\mathbb X}\to \uPic_{\B,\proet}$ induces a short exact sequence on $\SmRig_{K,\an}$
			\[ 0\to Q^M\to \mathcal P^M_{\mathbb X}\to \mathcal P_{\mathbb X}\to 0.\]
			In particular, $\mathcal P^M_{\mathbb X}$ is always representable by a rigid group.
		\end{enumerate}
	\end{Theorem}
	\begin{proof}
		
		We first consider the morphism $Q^M\to \uPic^M_{\B,\proet}$ in part 1. Let $S\in \SmRig_{K}$ and $\alpha, \alpha' \in \prod_{x\in X}(B_{|x}\otimes \O(S))^\times$, then an isomorphism between the associated $M$-rigidified invertible $\B$-modules $(\B,(\alpha_x))$ and $(\B,(\alpha'_x))$ on $X\times S$ is the datum of an isomorphism $\beta:\B\to \B$ over $X\times S$ such that $\alpha_x=\alpha'_x\circ \beta_{|x}$ for all $x\in M$. Here $\beta$ is necessarily given by multiplication by an element in $H^0(X\times S,\B^\times)$. This shows that the first map is well-defined and injective.
		
		Next, we show  exactness in the middle: Let $\mathcal L=(L,(\alpha_x)_{x\in M})$ be an $M$-rigidified invertible $\B$-module on $X\times S$ that goes to $1\in \uPic_{\B,\proet}(S)$. This means that the isomorphism class of $L$  has a preimage $L_0$ under the map $H^1_\proet(S,\pi_{S,\proet\ast}\B^\times)\to H^1_{\proet}(X\times S,\B^\times)$. The existence of the rigidifications $\alpha_x$ means that $L_0$ lies in the kernel of the specialisation map
		\[\textstyle H^1_\proet(S,\pi_{S,\proet\ast}\B^\times)\to \prod_{x\in M} H^1_\proet(S,\B_{|x}^\times).\]
		Using  the identification $\pi_{S,\proet\ast}\B=B(X)\otimes_K \O_{S_\proet}$ from \eqref{eq:coh-bc}, we now see from \Cref{l:map-on-units-split}.3 that this map is split injective. Hence $L_0$ is trivial, thus so is $L$. We can therefore without loss of generality assume that $L=\B$. Hence $\mathcal L$ is in the image of $Q^M\to  \uPic^M_{\B,\proet}$.
		
		For part 1, it remains to prove surjectivity on the right: By \Cref{p:explicit-description-of-Pic_B}, any class in $\uPic_{\B,\proet}(S)$ comes from an invertible $\B$-module $L$ on $X\times S$. In order to lift this to $\uPic^M_{\B,\proet}(S)$, it would suffice to see that $L_{|x}\cong \B_{|x}$ for each $x\in M$. To this end, recall that by \Cref{c:Leray-seq-for-B^x}.2 and the discussion below it, using functoriality in $X$ applied to the morphism $x:\Spa(K)\to X$, we have a commutative diagram
		\[\begin{tikzcd}[column sep = 0.3cm]
			{H^1_{\et}(X\times S,B^\times)} \arrow[r] \arrow[r]\arrow[d]  & {H^1_{\proet}(X\times S,\B^\times)} \arrow[r]\arrow[d] & {H^0(X,B)\otimes_K\wtOm_S(S)\oplus H^0(X,B\otimes \wtOm)\otimes_K  \O(S)} \arrow[d].\\
			{H^1_{\et}(S,B_{|x}^\times)} \arrow[r] \arrow[r]  & {H^1_{\proet}(S,\B_{|x}^\times)} \arrow[r] & {B_{|x}\otimes_K \wtOm_S(S)}
		\end{tikzcd}\]
		in which the last vertical map is the projection to the first factor followed by $H^0(X,B)\to B_{|x}$. According to \Cref{t:rel-vPic-spectral-var}.1, after replacing $S$ by an \'etale cover and tensoring $L$ with an element in $H^1_{\proet}(S,\pi_{S,\proet \ast}\B^\times)$, we can assume that the image of $L$ in the top right corner is concentrated in the second factor  $H^0(X,B\otimes \wtOm)\otimes  \O(S)$. Hence its image in the bottom right corner is $0$. This means precisely that $L_{|x}$ is \'etale-locally isomorphic to $\B_{|x}$ on $S$. After replacing $S$ by an \'etale cover, we therefore have $L_{|x}\cong \B_{|x}$ for each $x\in M$, as desired.
		
		This proves part 1. For part 2, it is immediate that we obtain a sequence on $\SmRig_{K,\et}$ by pullback. In particular, $\mathcal P^M_{\mathbb X}\to \mathcal P_{\mathbb X}$ is a $Q^M$-torsor. But by \Cref{l:QM}, we have $Q^M\cong \G_a^r\times \G_m^{s}$ for some $r,s\in \N$, so any such torsor is already trivial in the analytic topology. Hence the sequence is exact on $\SmRig_{K,\an}$ and it follows that $\mathcal P^M_{\mathbb X}$ is representable.
	\end{proof}
	We now describe the Lie algebra of $\mathcal P^M_{\mathbb X}$. By \Cref{t:red-structure-group}, this will sit in an exact sequence
	\[ 0\to \Lie Q^M\to \Lie \mathcal P^M_{\mathbb X}\to \mathcal A_B(K)\to 0.\]
	To give an explicit and functorial description of $\Lie \mathcal P^M_{\mathbb X}$, we introduce an additive variant of rigidified Picard functors, parametrising additive $\B$-torsors instead of invertible $\B$-modules:
	\begin{Definition}
		Let $L\mathcal P^M_{\mathbb X}=L\mathcal P^M_{\mathbb X,B}$ be the presheaf \[L\mathcal P^M_{\mathbb X}\colon\SmRig_{K,\et}\to \mathrm{Ab},\quad S\mapsto (a,[(\alpha_x)_{x\in M}])\]
		defined as follows: The first entry is a morphism  $a:S\to \mathcal A_B$ over $K$, giving rise to a $\B$-torsor $a^{\ast}L_{\mathbb X,B}$ on $X\times S$ via pullback of the $\B$-torsor $L_{\mathbb X,B}$ from \Cref{d:L_XB-e_X} along $\id\times a:X\times S\to X\times \mathcal A_B$. Then $\alpha_x$ is an isomorphism of $\B_{|x}$-torsors $(a^{\ast}L_{\mathbb X,B})_{|x}\to \B_{|x}$. We call two tuples $(\alpha_x)_{x\in M}$, $(\alpha_x')_{x\in M}$  isomorphic when there is $\beta\in H^0(X\times S,\B)$ such that $\alpha_x'+ \beta_{|x} = \alpha_x$. Then $[(\alpha_x)_{x\in M}]$ denotes the associated isomorphism class of this equivalence relation.
	\end{Definition}
	
	\begin{Lemma}\label{l:LPXM}
		There is a natural short exact sequence of presheaves on $\SmRig_{K,\et}$
		\[ 0\to LQ^M\to L\mathcal P^M_{\mathbb X}\to \mathcal A_B\to 0\]
		where $LQ^M$ is the rigid vector group defined by sending $S\in \SmRig_{K,\et}$ to 
		\[ LQ^M(S)=\big((\textstyle\prod_{x\in M}\B_{|x})/ H^0(X,\B)\big)\otimes_K \O(S).\]
		In particular, $L\mathcal P^M_{\mathbb X}$ is a sheaf represented by a rigid vector group, i.e.\ $L\mathcal P^M_{\mathbb X}=L\mathcal P^M_{\mathbb X}(K)\otimes_K\G_a$.
	\end{Lemma}
	\begin{proof}
		The $\B$-torsor $L_{\mathbb X,B}$  on $X\times \mathcal A_B$ is by construction the pushout of a pullback of a $\nu^\ast\wtOm^\vee_X$-torsor on $X_\proet$. Hence $(L_{\mathbb X,B})_{|x}$ is a trivial $\B_{|x}$-torsor on $\mathcal A_B$, so $(a^{\ast}L_{\mathbb X,B})_{|x}$ is isomorphic to $\B_{|x}$ on $S_\proet$. This shows surjectivity on the right. For left-exactness, note that $a^\ast L_{\mathbb X,B}=\B$ when $a$ is the zero-section $0:S\to \mathcal A_B$. Using \Cref{eq:coh-bc} to identify $ H^0(X\times S,\B)=H^0(X,\B)\otimes \O(S)$, the isomorphism classes of rigidifications of $\B$ are then given by $LQ^M(S)$.
		
		Hence the sequence of presheaves is exact. This implies that $L\mathcal P^M_{\mathbb X}=L\mathcal P^M_{\mathbb X}(K)\otimes_K\G_a$ because the analogous statement holds for the two outer terms. Thus $L\mathcal P^M_{\mathbb X}$ is a sheaf.
	\end{proof}
	\begin{Lemma}\label{l:Lie-of-PMX}
		\begin{enumerate}
			\item There is a canonical identification, functorial in $(X,\mathbb X)$, $M$ and $B$,
			\[ \Lie \mathcal P^M_{\mathbb X}=L\mathcal P^M_{\mathbb X}(K).\]
			\item For any larger finite multiset $M'\supseteq M$ of $K$-points of $X$, the forgetful map defines a natural morphism $\mathcal P^{M'}_{\mathbb X}\to \mathcal P^M_{\mathbb X}$. Via (1), its derivative is identified with the forgetful morphism $L\mathcal P^{M'}_{\mathbb X}\to L\mathcal P^M_{\mathbb X}$ evaluated on $K$.
		\end{enumerate}
	\end{Lemma}
	\begin{proof}
		Comparing the definition of $L\mathcal P^{M}_{\mathbb X}$ to that of $\mathcal P^{M}_{\mathbb X}$ in \Cref{t:red-structure-group}, we see that pushout along $\exp:\mathcal B\to \B^\times\tf$ defines compatible morphisms of sheaves on $\SmRig_{K,\et}$, written vertically, 
		\[\begin{tikzcd}[column sep = 0.3cm,row sep = 0.5cm]
			0\arrow[r]&{ LQ^{M}}\arrow[r]\arrow[d]  & {L\mathcal P^{M}_{\mathbb X}} \arrow[r]\arrow[d] & {\mathcal A_B} \arrow[d,"\id"]\arrow[r]&0\\
			0\arrow[r]&{Q^{M}\tf} \arrow[r] \arrow[r]  & {\mathcal P^{M}_{\mathbb X}\tf }\arrow[r] & {\mathcal A_B}\arrow[r]&0.
		\end{tikzcd}\]
		Recall from \eqref{eq:Lie-as-Hom} that the functor $\Lie(-)$ on commutative rigid groups extends to a functor $\Hom(\G_a^+,-\tf)$ on abelian sheaves on $\SmRig_{K,\et}$.
		Since $Q^{M}\cong \G_a^r\times \G_m^s$ by \Cref{l:QM}, the left and right vertical maps are isomorphisms after applying $\Hom(\G_a^+,-\tf)$. Hence the same is true for the middle map by the 5-Lemma: Here we use that $\Hom(\G_a^+,-\tf)$ is left-exact, and even exact for the top row because by \Cref{l:LPXM}, this is a short exact sequence of rigid vector groups. This shows part 1.
		
		The second part about $M'$ follows immediately from this construction because the forgetful maps are clearly compatible with the above morphism of short exact sequences.
	\end{proof}
	
	\begin{Definition}\label{p:lift-induces-elt-in-LiePM}
		Let $x\in X(K)$. Assume we are given a lift $\tilde x:\Spa(\BdR^+/\xi^2)\to \mathbb X$ of $x$. Then the functoriality of $L_{\mathbb X}$ in $(X,\mathbb X)$ of \Cref{l:Faltings-ext}.2 applied to the morphism \[(x,\tilde x):(\Spa(K),\Spa(\BdR^+/\xi^2))\to (X,\mathbb X)\] induces an isomorphism of $\B_{|x}$-torsors
		$\alpha_{(x,\tilde{x})}:(L_{\mathbb X,B})_{|x}\to \B_{|x}$.
		By definition of $L\mathcal P^M_{\mathbb X}$ and \Cref{l:Lie-of-PMX}, the datum $\mathbb M$ of a lift $\tilde{x}$ for each $x\in M$ therefore induces a natural section
		\[ s_{\mathbb X,\mathbb M}:\mathcal A_B(K)\to L\mathcal P^M_{\mathbb X}(K)=\Lie \mathcal P^M_{\mathbb X}.\]
		This construction is clearly functorial in $(X,\mathbb X)$, $\mathbb M$ and $B$.
	\end{Definition}

	\subsection{The $p$-adic exponential for rigid groups}\label{s:exp-for-rigid-groups}
	As we have already used earlier, recall that the $p$-adic exponential is as usually defined as the continuous group isomorphism
	\[ \textstyle\exp:p^{\alpha}\mathfrak m_K\to 1+p^{\alpha}\m_K,\quad x\mapsto \sum_{n=0}^\infty\tfrac{x^n}{n!}\]
	where $\alpha=\frac{1}{p-1}$ for $p\geq 3$, and $\alpha=2$ for $p=2$. Second, we have the $p$-adic logarithm 
	\[\textstyle\log:1+\m_K\to K,\quad x\mapsto \sum_{n=1}^\infty(-1)^{n+1}\tfrac{(x-1)^n}{n}\]
	whose radius of convergence is larger and that is surjective because $K$ is algebraically closed.
	
	\begin{Definition}
		An exponential for $K$ is a group homomorphism $\Exp\colon K\to 1+\m_K$ s.t.\:
		\begin{enumerate}
			\item $\Exp$ splits the logarithm, meaning that $\log\circ\Exp=\id$, and
			\item $\Exp$ restricts to $\exp$ on $p^{\alpha}\mathfrak m_K$.
		\end{enumerate}
		It is clear from (2) that $\Exp$ is automatically continuous. An exponential for $K$ always exists (see e.g.\ \cite[Lemma~6.2]{HMZ}).
	\end{Definition}
	
	For any commutative rigid group $G$ over $K$, let $\wh{G}=\uHom(\underline{\Z}_p,G)\subseteq G$ be the topological $p$-torsion sub-v-sheaf from \cite[\S2]{heuer-geometric-Simpson-Pic}. By \cite[Prop.~2.14]{heuer-geometric-Simpson-Pic}, $\wh{G}$ is an open rigid subgroup. If $[p]:G\to G$ is surjective, $\wh{G}$ is an analytic $p$-divisible group in the sense of Fargues by \cite[Prop.~6.10]{HMZ}. By \cite[\S2]{Fargues-groupes-analytiques}, we then have a logarithm map that fits into an exact sequence of rigid groups 
	\[ 0\to G[p^\infty]\to \wh{G}\xrightarrow{\log_G} \Lie(G)\otimes_K \G_a\to 0\] 
	where $\Lie(G)$ is the Lie algebra of $G$, i.e.\ the tangent space at the identity. We now use:
	
	\begin{Theorem}[{\cite[Thm.~6.12]{HMZ}}]\label{t:rigid-group-exp}
		Let $G$ be a commutative rigid group such that $[p]\colon G\to G$ is surjective on $K$-points. Then any exponential $\Exp\colon K\to 1+\m$ induces a continuous splitting
		\[ \Exp_G: \Lie(G)\to \wh{G}(K)\subseteq G(K)\]
		of the logarithm, i.e.\ $\log_G\circ\Exp_G=\id$. For fixed $\Exp$, the map $\Exp_G$ is functorial in $G$.
	\end{Theorem}
	
	Applying \Cref{t:rigid-group-exp} to the identity component of the rigid group $\mathcal P_{\mathbb X}$ from \Cref{t:red-structure-group}, we obtain for any coherent $\O_X$-algebra $B$ an exponential map
	\begin{equation}\label{eq:exp-on-P_X}
	 \Exp:\mathcal A_B(K)\to \mathcal P_{\mathbb X}(K)\subseteq \uPic_{\B,\proet}(K)
	 \end{equation}
	 that is functorial in $B$.
	In the next section, we will explain that one can also apply \Cref{t:rigid-group-exp} to the rigidified Picard functors $\mathcal P^M_{\mathbb X}$ of \Cref{t:rigidifed-PicB}, thus yielding a functorial construction of ``exponential $\B$-modules'', not only isomorphism classes of such.
	
	\subsection{Invertible $\B$-modules via the exponential}
	Finally, we now consider the case that $B$ is a $\Sym^\bullet_{\O_X} \wtOm_X^\vee$-algebra on $X_{\et}$ that is coherent over $\O_X$. Then $\mathcal A_B$ has a canonical section:
	
	\begin{Definition}\label{d:can-section-tau}
		Let $X$ be any smooth rigid space. Set $T_X:=\Sym^\bullet_{\O_X} \wtOm_X^\vee$ and let $B$ be any $T_X$-algebra. As $\wtOm_X^\vee$ is projective, the induced $\O_X$-linear map $\wtOm_X^\vee\to B$ dualises to a morphism 
		$\tau_B:\O\to \wtOm_X\otimes B$. This defines a tautological section $\tau_B\in H^0(X,\wtOm_X\otimes B)=\mathcal A_B(K)$.
	\end{Definition}
	
	We can now finally combine all results that we have discussed up until this point to obtain the following, which summarises the key technical construction of this article:
	\begin{Theorem}\label{t:inv-B-module-assoc-to-HF} 
		Let $X$ be a smooth proper rigid space over $K$. Let $\nu=\nu_X:X_\proet\to X_{\et}$ be the natural map. Assume we are given:
		\begin{itemize}[-]
			\item a $\BdR^+/\xi^2$-lift $\mathbb X$ of $X$, and
			\item an exponential $\mathrm{Exp}:K\to 1+\m$.
		\end{itemize}
		Then there is a way  to associate to any $\O_X$-coherent $\O_X$-torsionfree $T_X$-algebra $B$ an invertible $\mathcal B:=\nu^{\ast}B$-module $\mathcal L_B$ on $X_{\proet}$ in such a way that the following conditions hold:
		\begin{enumerate}
			\item We have $\HTlog(\mathcal L_B)=\tau_B\in \mathcal A_B(K)$. More precisely, the class of $\mathcal L_B$ in $\uPic_{\B,\proet}(K)$ is given by the image of $\tau_B$ under the map $\Exp:\mathcal A_B(K)\to \mathcal P_{\mathbb X}(K)$ from \eqref{eq:exp-on-P_X}.
			\item $\mathcal L_B$ is natural in $B$: To any morphism $\phi:B\to B'$ of $\O_X$-coherent $\O_X$-torsionfree $T_X$-algebras  with $\B':=\nu^{\ast}B'$, we can associate an isomorphism 
			\[\psi_{\phi}:\mathcal L_B\otimes_{\B}{\B'}\isomarrow \mathcal L_{B'}\] of $\B'$-modules on $X_\proet$, in a way that is compatible with compositions.
			\item $\mathcal L_B$ is natural in $(X,\mathbb X,K,\Exp)$: 
			Let $K\hookrightarrow K'$ be any continuous homomorphism into a complete algebraically closed field $K'$ and let $\Exp'$ be an exponential for $K'$ that restricts to $\Exp$ on $K$. Let $X'$ be a smooth proper rigid space over $K'$ with lift $\mathbb X'$ to $\BdR^+(K')/\xi^2$. Let $f:X'\to X$ be a morphism of adic spaces over $K$ with lift $\widetilde{f}:\mathbb X'\to \mathbb X$ of $f$ to $\BdR^+(K)/\xi^2$. Let $B'$ be the maximal $\O_{X'}$-torsionfree quotient of $f^\ast B$ and $\B':=\nu_{X'}^\ast B'$. Then we can associate to $\tilde{f}$ an isomorphism \[\psi_{\tilde{f}}:f^{\ast}\mathcal L_B\otimes_{f^\ast \B}\B'\isomarrow \mathcal L_{B'}\] in a way that is compatible with compositions and commutes with the $\psi_\phi$ from (2).
			\item Given two morphisms of $\O_X$-coherent $\O_X$-torsionfree $T_X$-algebras $\phi_i:B\to B_i$ for $i=1,2$, let $B_3$ be the maximal $\O_X$-torsionfree quotient of $B_1\otimes_B B_2$, endowed with the $T_X$-algebra structure for which $\tau_{B_3}=\tau_{B_1}\otimes 1+1\otimes \tau_{B_2}$. Set $\B_i:=\nu^\ast B_i$ for $i=1,2,3$. Set $\mathcal L_{B_1}\boxtimes_{\B_3} \mathcal L_{B_2}:=(\mathcal L_{B_1}\otimes_{\B_1}\B_3)\otimes_{\B_3}(\mathcal B_3\otimes_{\mathcal B_2}\mathcal L_{B_2})$. Then there is an isomorphism \[\psi_{\phi_1,\phi_2}:\mathcal L_{B_1}\boxtimes_{\B_3} \mathcal L_{B_2} \isomarrow \mathcal L_{B_3}.\] This is compatible in triple products, and with the isomorphisms from (2) and (3).
		\end{enumerate}
	\end{Theorem}
	\begin{proof}
		Given $B$, we choose any $B$-rigidifying set $\mathbb M\subseteq \mathbb X(\BdR^+/\xi^2)$ in the sense of \Cref{d:rigidifying-set}.2. 
		This means that its reduction $M$ is a $B$-rigidifying multiset of $K$-points of $X$. We can always find such a set $\mathbb M$:  By \Cref{l:injective-into-fibre}, we can first find a $B$-rigidifying set $M$ and then lift this to $\mathbb X$ using that $\mathbb X\to \Spa(\BdR^+/\xi^2)$ is smooth. 
		
		By \Cref{t:rigidifed-PicB}.2, we then have a short exact sequence of rigid groups over $K$
		\[0\to Q^M\to \mathcal P^M_{\mathbb X}\to \mathcal P_{\mathbb X}\to 0.\]
		The multiplication map $[p]$ is surjective on $Q^M$ by \Cref{l:QM} and also on the identity component of $\mathcal P_{\mathbb X}$ by \Cref{t:red-structure-group}. It follows that $[p]$ is surjective on the identity component of $\mathcal P_{\mathbb X}^M$. We can therefore apply \Cref{t:rigid-group-exp} which combines with \Cref{l:Lie-of-PMX} to give an exponential map
		\[ \Exp:L\mathcal P^M_{\mathbb X}(K)=\Lie \mathcal P^M_{\mathbb X}\to \mathcal P^M_{\mathbb X}(K).\]
		Finally, \Cref{p:lift-induces-elt-in-LiePM} and \Cref{d:can-section-tau} say that $\mathbb M$ induces a canonical element 
		\[  s_{\mathbb X,\mathbb M}(\tau_B)\in L\mathcal P^M_{\mathbb X}(K).\]
		We can thus associate to $\mathbb M$ the isomorphism class 
		\[\Exp( s_{\mathbb X,\mathbb M}(\tau_B))\in \mathcal P^M_{\mathbb X}(K)\] of an $M$-rigidified invertible $\B$-module on $X_\proet$. By \Cref{l:isom-between-rigidified-modules-unique}, this class uniquely determines an $M$-rigidified invertible $\B$-module $\mathcal L^{\mathbb M}_{B}=(L^{\mathbb M}_{B},(\alpha_{B,x}^{\mathbb M})_{x\in M})$ up to unique isomorphism.
		
		We claim that $L^{\mathbb M}_B$ is independent of the choice of $\mathbb M$: To see this, let $\mathbb M'\subseteq \mathbb X(\BdR^+/\xi^2)$ be any other choice of a $B$-rigidifying set. Comparing both of them to $\mathbb M\cup \mathbb M'$, we see that it suffices to treat the case that $\mathbb M\subseteq \mathbb M'$. As an aside, we note that this is the point where the reduction $\mathbb M\cup \mathbb M'$ might become a multiset, because it is important to allow the case that $\mathbb M$ and $\mathbb M'$ contain different lifts of the same point in $X(K)$.

		By \Cref{l:Lie-of-PMX}.2, we then have forgetful maps which by functoriality in \Cref{t:rigid-group-exp} fit into a commutative diagram 
		\[\begin{tikzcd}
			\mathcal P_{\mathbb X}^{\mathbb M'}(K) \arrow[r] & \mathcal P_{\mathbb X}^{\mathbb M}(K) \\
			L\mathcal P_{\mathbb X}^{\mathbb M'}(K) \arrow[u, "\Exp"] \arrow[r] & L\mathcal P_{\mathbb X}^{\mathbb M}(K). \arrow[u, "\Exp"]
		\end{tikzcd}\]
		Using again \Cref{l:isom-between-rigidified-modules-unique}, it follows that there is a unique isomorphism \[L^{\mathbb M}_{B}\isomarrow L^{\mathbb M'}_{B}\] of $\B$-modules that identifies $\alpha_{B,x}^{\mathbb M}$ with $\alpha_{B,x}^{\mathbb M'}$ for each $x\in \mathbb M$. This proves independence of $\mathbb M$: Explicitly, we can consider the filtered direct system of $B$-rigidifying sets $\mathbb M\subseteq \mathbb X(\BdR^+/\xi^2)$ and define
		\[\mathcal L_B:=\textstyle\varinjlim_{\mathbb M} L^{\mathbb M}_{B}.\]
		
		Part 2, the functoriality in $B$, is similar: We choose $\mathbb M$ large enough such that $M$ is rigidifying for both $B$ and $B'$. It is clear that pushout along $\phi$ defines a homomorphism
		\begin{equation}\label{eq:pushout-PicMB}
		q_{B,B'}:\uPic^M_{\B,\proet}\to \uPic^M_{\B',\proet}.
		\end{equation}
		By functoriality of $e_{\mathbb X}\circ \HTlog$
		in $B$, this restricts to a homomorphism
		\[ \mathcal P^{\mathbb M}_{\mathbb X,\B}\to \mathcal P^{\mathbb M}_{\mathbb X,\B'}.\]
		By functoriality in \Cref{l:Lie-of-PMX}, this commutes with the exponentials, exactly like above. This shows that the $\mathbb M$-rigidified $\B'$-modules $(L^{\mathbb M}_{B}\otimes \B',(\alpha_{B,x}^{\mathbb M}\otimes \B')_{x\in M})$ and $(L^{\mathbb M}_{B'},(\alpha_{B',x}^{\mathbb M})_{x\in M})$ are isomorphic. By \Cref{l:isom-between-rigidified-modules-unique}, there is  a unique isomorphism $\psi_{\phi}:L_B^{\mathbb M}\otimes_{\B}{\B'}\to L^{\mathbb M}_{B'}$ that commutes with rigidifications. The uniqueness also yields compatibility with compositions.
		
		The functoriality in $(X,\mathbb X)$ and $(K,\Exp)$ is similar: Assume first that $K=K'$. Choose rigidifying sets $\mathbb M'\subseteq \mathbb X'(\BdR^+/\xi^2)$ for $\B'$ and $\mathbb M\subseteq \mathbb X(\BdR^+/\xi^2)$ for $\B$. We may assume that $\tilde f(\mathbb M')\subseteq \mathbb M$ after adding $\tilde f(\mathbb M')$ to $\mathbb M$. Then pullback along $f$ induces a pullback morphism
		\[\uPic^{\mathbb M}_{\B,\proet}\to \uPic^{\mathbb M'}_{\B',\proet}\]
		which once again restricts to a homomorphism
		$\mathcal P^{\mathbb M}_{\mathbb X,\B}\to \mathcal P^{\mathbb M'}_{\mathbb X',\B'}$.
		We can now argue as in part~2 to obtain by functoriality in \Cref{l:Lie-of-PMX} the desired morphism $\psi_{\tilde f}$, using additionally the functoriality in $(X,\mathbb X)$ in \Cref{p:lift-induces-elt-in-LiePM}.
		
	Compatibility with composition and with part 2 follows from the uniqueness in \Cref{l:isom-between-rigidified-modules-unique}. 
		
		For part 3, it remains to treat the base-change $f:X_{K'}\to X$: We first note that the base-change of $\mathbb X$ along $\BdR^+(K)/\xi^2\to \BdR^+(K')/\xi^2$  defines a lift $\mathbb X'$ of $X_{K'}$. Then the pullback $\mathbb M'$ of any $B$-rigidifying set $\mathbb M\subseteq \mathbb X(\BdR^+/\xi^2)$ is $B'$-rigidifying, where $B':=f^{\ast}B$. We thus obtain a pullback map $\mathcal P^{\mathbb M}_{\mathbb X,\B}\to \mathcal P^{\mathbb M'}_{\mathbb X',\B'}$ and from here the construction goes exactly as before.
		
		For part 4, let $\mathbb M$ be large enough such that it is rigidifying for $B_1$, $B_2$ and $B_3$. It is straightforward to see that $-\boxtimes_{\B_3}-$ defines a natural morphism
		\[ \uPic^{\mathbb M}_{\B_1,\proet}\times \uPic^{\mathbb M}_{\B_2,\proet}\to \uPic^{\mathbb M}_{\B_3,\proet},\quad ( L_1,(\alpha_{x})),( L_2,(\beta_{x}))\mapsto ( L_1\boxtimes_{\B_3} L_2,(\alpha_{x}\boxtimes \beta_{x})).\]	
		Explicitly, in terms of the pushout maps  \Cref{eq:pushout-PicMB} this can be described as the product of $q_{B_1,B_3}$ and $q_{B_2,B_3}$. It follows from this and the fact that $\HTlog$ is a homomorphism that  we have $\HTlog(L_1\boxtimes_{\B_3}L_2)=\HTlog(L_1)\otimes 1+1\otimes \HTlog(L_2)$. Since we also have $\tau_{B_3}=\tau_{B_1}\otimes 1+1\otimes \tau_{B_2}$ by construction, it follows that the above restricts to a natural morphism
		\[\cdot:\mathcal P^{\mathbb M}_{\mathbb X,\B_1}\times \mathcal P^{\mathbb M}_{\mathbb X,\B_2}\to \mathcal P^{\mathbb M}_{\mathbb X,\B_3}. \]
		By the description of $-\boxtimes_{\B_3}-$ as $q_{B_1,B_3} \cdot q_{B_2,B_3}$, its derivative has to be the map
		\[ 	L\mathcal P_{\mathbb X,\B_1}^{\mathbb M}\times 	L\mathcal P_{\mathbb X,\B_2}^{\mathbb M}\to 	L\mathcal P_{\mathbb X,\B_3}^{\mathbb M} \]
		given by the sum of the natural pushout maps. As $\Exp$ is a homomorphism, it follows that
		\[\Exp( s_{\mathbb X,\mathbb M}(\tau_{B_1}))\cdot  \Exp( s_{\mathbb X,\mathbb M}(\tau_{B_2}))=\Exp( s_{\mathbb X,\mathbb M}(\tau_{B_1}\otimes 1+1\otimes \tau_{B_2}))=\Exp( s_{\mathbb X,\mathbb M}(\tau_{B_3})).\]
		
		From here, the argument works exactly like before.
	\end{proof}
	\begin{Remark} We note that by construction, each $\mathcal B$-module $\mathcal L_B$ in \Cref{t:inv-B-module-assoc-to-HF}  comes equipped with an isomorphism
	$\alpha_{\tilde{x}}:(\mathcal L_B)_{|x}\isomarrow \mathcal B_{|x}$
	for every $x\in X(K)$ and every lift $\tilde{x}\in \mathbb X(\BdR^+/\xi^2)$. We could thus introduce a notion of ``fully rigidified $\B$-modules'' to characterise $\mathcal L_B$ uniquely up to unique isomorphism. This is why we use the notation $\mathcal L_B$  reserved for rigidifed modules rather than $L_B$. But we will not need this in the following.
	\end{Remark}
	\begin{Remark}\label{r:Simpson-gerbe}
		In upcoming work in progress, Bhatt--Zhang will show that one can reinterpret  \Cref{t:inv-B-module-assoc-to-HF} more geometrically as constructing a natural  local splitting of a certain rigid analytic gerbe over the cotangent space of $X$, which they call the Simpson gerbe.
	\end{Remark}
	\section{Local considerations on the $p$-adic Simpson correspondence via twisting}
	The final ingredient for our construction of the non-abelian Hodge correspondence are local considerations of how to pass between Higgs bundles and pro-\'etale vector bundles via twisting with invertible $\B$-modules. This is based on the following general construction:
		\begin{Definition}\label{d:wtTX-action}
		Let $(E,\theta_E)$ be a Higgs bundle on $X$. The Higgs field $\theta_E:E\to E\otimes \wtOm_X$ dualises to a morphism $\wtOm^{\vee}_X\to \uEnd(E)$ sending $\partial \mapsto (E\xrightarrow{\theta}E\otimes \wtOm_X\xrightarrow{\id\otimes \partial}E)$. Due to the Higgs field condition $\theta\wedge \theta=0$, this extends to an $\O_X$-algebra morphism on $X_{\et}$
		\[ T_X:=\Sym^\bullet_{\O_X} \wtOm_X^{\vee}\to \uEnd(E).\]
		Let $B_{\theta}$ be the image of this morphism, this is a commutative subalgebra of $\uEnd(E$). Since $\uEnd(E)$ is coherent as an $\O_X$-module, so is its submodule $B_\theta$. There is a canonical section \[\tau_{\theta}\in H^0(X,B_{\theta}\otimes \wtOm_X)\]
		defined as the image of $\id\in \wtOm^\vee_X\otimes \wtOm_X\to B_{\theta}\otimes \wtOm_X$, uniquely determined by the property that
		\begin{equation}\label{eq:tau-goes-to-theta}
		H^0(X,B_{\theta}\otimes \wtOm_X)\hookrightarrow H^0(X,\uEnd(E)\otimes \wtOm_X) \quad \text{sends}\quad \tau_{\theta}\mapsto\theta_E.
		\end{equation}
			We shall also denote $B_\theta$ just by $B$ and $\tau_\theta$ by $\tau_B$ when $\theta$ is clear from context.
	\end{Definition}
	Invoking \Cref{t:inv-B-module-assoc-to-HF},
	the idea for constructing the $p$-adic Simpson functor for proper $X$
	\[ \{\text{Higgs bundles on }X\}\isomarrow	\{\text{pro-\'etale vector bundles on $X$}\}\]
	will now be to send any Higgs bundle $(E,\theta)$ on $X$ to the pro-\'etale vector bundle $\nu^\ast E\otimes_{\B_\theta}\mathcal L_{B_\theta}$ where $\mathcal L_{B_\theta}$ is the invertible $\B_\theta:=\nu^{\ast}B_\theta$-module from \Cref{t:inv-B-module-assoc-to-HF} with $\HTlog(\mathcal L_{B_\theta})=\tau_{\theta}$. 
	\begin{Remark}\label{r:enlarge-B}
	Note that by \Cref{t:inv-B-module-assoc-to-HF}.2, we are free to enlarge $B=B_\tau$, as follows: If $T_X\to B'\to B$ is any $\O_X$-coherent sub-quotient of $T_X$, then $B'$ acts on $E$ via $B'\to B$, and
	\begin{equation} \nu^\ast E\otimes_{\B}\mathcal L_{B}=\nu^\ast E\otimes_{\B}\B\otimes_{\B'}\mathcal L_{B'}= \nu^\ast E\otimes_{\B'}\mathcal L_{B'}.
	\end{equation}
	\end{Remark}
	In order to be able to go into the other direction, we need some preparations on pro-\'etale vector bundles: For this we begin by recalling the local correspondence. As we will explain, one can reinterpret this in terms of twisting with pro-\'etale invertible $\B$-modules.
	
	 In contrast to Faltings' construction, we do not actually rely on the local correspondence for the \textit{construction} of $\mathrm S_{\mathbb X,\Exp}$, but we will use it to see that $\mathrm S_{\mathbb X,\Exp}$ is an equivalence.
	\subsection{The Local correspondence}
	\begin{Definition}\label{d:rho}
		We call an affinoid rigid space $U$ toric if there is an \'etale map $f:U\to \mathbb T^d$ to the torus over $K$ which is a composition of rational localisations and finite \'etale maps. We call $f$ a toric chart. Given a chart $f$, consider the affinoid perfectoid torus $\mathbb T^d_\infty\to \mathbb T$, a pro-\'etale Galois torsor under the group $\Delta:=\Z_p(1)^d$. We denote by $\wt U\to U$ the pullback along $f$. The chart $f$ induces parameters $T_1,\dots,T_d\in \O(U)^\times$ on $U$ which define a basis $\frac{dT_1}{T_1},\dots,\frac{dT_d}{T_d}$ of $\Omega_U$. We denote by $\partial_1,\dots,\partial_d$ the dual basis of $\Omega_U^\vee$.
		 Then $f$ induces an isomorphism
		\[\rho_f:H^0(U,\wtOm_U)\isomarrow \Hom_{\Z_p}(\Delta,\O(U)).\] which can be characterised as follows: 
		Its dual $\Z_p^d(1)\to \Omega_U^\vee(1)(U)$ is the $(1)$-twist of the map that sends the standard basis vector $\gamma_i$ of $\Z_p^d$ to $\partial_i$. For any coherent $\O_U$-module $B$, we will denote by $\rho_{f,B}$ or just by $\rho_f$ the induced isomorphism obtained by tensoring with $B(U)$.
	\end{Definition}
	
	There is for any smooth rigid space $U$ an intrinsic notion of ``smallness'' for both pro-\'etale vector bundles and Higgs bundles. As we will not need the technical details, we just refer to \cite[\S6]{heuer-sheafified-paCS} for the definition. What will be important for us is only the following:
	\begin{itemize}
		\item  Any pro-\'etale vector bundle or Higgs bundle on $U$ becomes small on an \'etale cover.
		\item For toric $U$, we have the following equivalence, the ``Local correspondence'':
	\end{itemize}
	\begin{Theorem}[{\cite[\S3]{Faltings_SimpsonI}\cite[Thm.~6.5]{heuer-sheafified-paCS}}]\label{t:local-corresp}
		Let $U$ be a toric smooth rigid space and let $f:U\to \mathbb T^d$ be a toric chart. Then $f$ induces an exact equivalence of categories
		\[\LS_f:\{\text{small pro-\'etale vector bundles on $U$}\}\isomarrow \{\text{small Higgs bundles on $U$}\}.\]
		It sends $(E,\theta)$ to the unique pro-\'etale vector bundle $V$ for which $V(\wt U)$ is the $\Delta$-module
		\[V(\wt U):= \O(\wt U)\otimes_{\O(U)}E(U), \quad \gamma \cdot (a\otimes x)=\gamma(a)\otimes  \exp(\rho_f(\theta)(\gamma))x\]
		where $\rho_f:=\rho_{f,\End(E)}\colon H^0(U,\wtOm_U\otimes \uEnd(E))\isomarrow \Hom(\Delta,\End(E))$ is the map from \Cref{d:rho}. In particular, we have a natural isomorphism of $\O(\wt U)$-modules
		$E(\wt U)=V(\wt U)$.
	\end{Theorem}
	\begin{Corollary}\label{c:End(V)-coherent}
		Let $X$ be any smooth rigid space and let $V$ be a pro-\'etale vector bundle on $X$. Then the sheaf of $\O_X$-linear endomorphisms $\nu_{\ast}\uEnd(V)$ is a coherent module on $X_{\et}$.
	\end{Corollary}
	\begin{proof}
		The statement is local, so we may assume that $X$ is toric and $V$ is small. Then by \Cref{t:local-corresp}, $\uEnd(V)\cong \uEnd(E,\theta)$ for some Higgs bundle $(E,\theta)$, and this is coherent.
	\end{proof}
	
	Let now $B$ be any coherent $\O_X$-algebra and $\mathcal B:=\nu^{\ast}B$. Choose $B^+\subseteq B$ as in the discussion after \Cref{l:coherent-integral-sumodule}. Then we also have the following variant which is essentially a weak version of a local correspondence for invertible $\B$-modules that will be enough for our purposes:
	
	\begin{Lemma}\label{l:tech-lemma-descr-Ltau}
		Let $\mathcal L$ be an invertible $\B$-module on $X_{\proet}$. Set $\tau:=\HTlog(\mathcal L)\in H^0(X,\wtOm\otimes B)$. Then there is an \'etale cover of $X$ by toric rigid spaces $U\to X$ with charts $f:U\to \mathbb T^d$ satisfying the following: Let $\wt U\to U$ be the toric $\Delta$-torsor induced by $f$. The restriction $\rho_f(\tau|_U):\Delta\to B(U)$ has image in $2pB^+(U)$ and $\mathcal L$ is isomorphic to the invertible $\B$-module $\mathcal L_{\tau|_U,f}$ on $U_\proet$ defined via descent along $\wt U\to U$ of $\B_{|\wt U}$ endowed with the following $\Delta$-action:
		\[\B(\wt U)=\O(\wt U)\otimes_{\O(U)} B(U),\quad \gamma \cdot (a\otimes x)=\gamma(a)\otimes \exp(\rho_{f,B}(\tau|_U)(\gamma)).\]
	\end{Lemma}
	\begin{proof}
		Let $U\to X$ be any \'etale map from a toric rigid space.
		By \Cref{c:Leray-seq-for-B^x}.2, we have a left-exact sequence
		\[0\to  H^1_{\et}(U,B^\times)\to  H^1_{\proet}(U,\B^\times)\xrightarrow{\HTlog_U} H^0(U,\wtOm_U\otimes B)\]
		It follows that any invertible $\B$-module $ \mathcal L'$ on $U_\proet$ with $\HTlog_U( \mathcal L')=\HTlog_U(\mathcal L)=\tau$ becomes isomorphic to $\mathcal L$ after passing to an \'etale cover.
		After passing to an open subgroup of $\Delta$ by replacing $U$ by a finite \'etale cover, we can ensure that $\rho_{f,B}(\tau)$ has image in $2pB^+(U)$ where $\exp$ is defined. Then the $\Delta$-module $\mathcal L_{\tau|_U,f}$ defined in the lemma satisfies $\HTlog(\mathcal L_{\tau|_U,f})=\tau$ by construction, hence $\mathcal L_{\tau|_U,f}\cong \mathcal L$  after passing to a further \'etale cover.
	\end{proof}
	We can use this to reinterpret the local correspondence in terms of twisting:
	\begin{Proposition}\label{l:LS-is-twisting}
		In the setting of \Cref{t:local-corresp}, let $(E,\theta)$ be any Higgs bundle. Denote by $B\subseteq \uEnd(E)$ the coherent $\O_X$-module of \Cref {d:wtTX-action} with section $\tau_\theta\in H^0(X,\wtOm\otimes B)$.
		Then over an \'etale cover of $U$, we have a natural isomorphism $\LS_{f}^{-1}(E,\theta)\isomarrow \nu^{\ast}E\otimes_{\mathcal B}\mathcal L_{\tau_\theta,f}$.
	\end{Proposition}
	\begin{proof}
		Comparing  the descriptions in \Cref{t:local-corresp} and \Cref{l:tech-lemma-descr-Ltau}, it suffices to see that for any $\gamma \in \Delta$, the element $\rho_{f,B}(\tau_\theta)(\gamma)\in B(U)$ acts on $E(U)$ as $\rho_{f,\End(E)}(\theta_E)(\gamma)\in \End(E(U))$. But by naturality of $\rho_{f,-}$ applied to the map $B\to \uEnd(E)$, we have a commutative diagram
		\[
		\begin{tikzcd}
			{H^0(U,\wtOm\otimes B)} \arrow[d] \arrow[r, "\rho_{f}"] & {\Hom(\Delta,B(U))} \arrow[d] & \tau_\theta \arrow[d, maps to] \arrow[r, maps to] & \rho_f(\tau_\theta) \arrow[d, maps to] \\
			{H^0(U,\wtOm\otimes \uEnd(E))} \arrow[r, "\rho_f"]   & {\Hom(\Delta,\End(E|_U))}     & \theta_E \arrow[r, maps to]                    & \rho_f(\theta_E)                
		\end{tikzcd}\]
		where the vertical maps send $\tau_\theta$ to $\theta_E$ by the defining property of $\tau_\theta$ in \Cref{eq:tau-goes-to-theta}.
	\end{proof}
	
	\subsection{Rodr\'iguez Camargo's Higgs field}
	Motivated by earlier results of Pan \cite[\S 3.1]{PanLocallyAnalytic}, it was first observed by Rodr\'iguez Camargo in \cite[Thm.~1.0.3]{camargo2022geometric},  that one can use the Local correspondence to endow any pro-\'etale vector bundle with a \textit{canonical} Higgs field in a natural way. Since his construction is written for a different technical setup, we now give a slight reinterpretation, which simplifies the proof somewhat in our special case of interest.
	\begin{Theorem}\label{t:Camargos-Higgs-field}
		Let $X$ be a smooth rigid space over $K$. 
		To simplify notation,  let us still denote  by $\wtOm_X$ the pullback $\nu^{\ast}\wtOm_X$ of the sheaf of \Cref{d:wtOm} along $\nu:X_\proet\to X_\et$.  Then there is a unique way to endow any pro-\'etale vector bundle $V$ on $X$ with a Higgs field
		\[ \theta_V:V\to V\otimes_{\O_X}\wtOm_X\]
		on $X_{\proet}$ in such a way that the following conditions hold:
		\begin{enumerate}
			\item The association $V\mapsto \theta_V$ is functorial in $V$ and $X$.
			\item If $X$ is toric and $f:X\to \mathbb T^d$ is a toric chart, then $\theta_V$ corresponds in terms of the associated Higgs bundle $(E,\theta_E):=\mathrm{LS}_f(V)$ to the tautological morphism of Higgs bundles
		$\theta_E: (E,\theta_E)\to (E,\theta_E)\otimes (\wtOm,0)$, where $(\wtOm,0)$ is $\wtOm$ with the trivial Higgs field.
		\item We have $\theta_V=0$ if and only if $V$ is \'etale-locally trivial on $X$.
		\end{enumerate}
	\end{Theorem}
	\begin{Remark}\label{r:RCH-field}
		For any local basis $\omega_1,\dots,\omega_d$ of $\wtOm_X$, the tensor product of Higgs bundles $(E,\theta_E)\otimes (\wtOm,0)$ in (2) is $E\otimes \wtOm$ with the Higgs field given in terms of $\theta=\sum_i\theta_i\omega_i$ by
			\[ \theta_E:\textstyle E\otimes \wtOm_X\to E\otimes \wtOm_X\otimes \wtOm_X,\quad \sum_i e_i\otimes \omega_i\mapsto \sum_i\sum_j \theta_j(e)\otimes \omega_i\otimes \omega_j.\]
		More explicitly, (2) means that  $\theta_V$ corresponds on the toric cover $\wt X\to X$ to the natural map
		\[ \O(\wt X)\otimes_{\O(X)} E(X)\to  \O(\wt X)\otimes_{\O(X)} E(X)\otimes \wtOm_X, \quad a\otimes e\mapsto a\otimes \theta(e)\]
		which commutes with the $\Delta$-action as $\exp(\theta_i)$ commutes with each $\theta_j$.
	\end{Remark}
	\begin{proof}[{Proof of \Cref{t:Camargos-Higgs-field}}]
		It is clear that (2) defines a Higgs field $\theta_V:V\to V\otimes \wtOm$ for any small pro-\'etale vector bundle $V$ if $X$ is toric. This is functorial in $V$ since any morphism of Higgs bundles $\phi:(E,\theta)\to (E',\theta')$ induces a morphism $\phi\otimes \id:(E,\theta)\otimes (\wtOm,0)\to (E',\theta')\otimes (\wtOm,0)$.
		
		It therefore suffices to prove that $\theta_V$ is independent of the choice of toric chart:
		Let $f'$ be a second toric chart and let $\theta_V':V\to V\otimes \wtOm$ be the Higgs field induced via $\LS_{f'}$. 
		 We need to show that $\theta_V=\theta_V'$. To see this, we may replace $X$ by any \'etale cover. Let $(E',\theta')=\LS_{f'}(V)$,  then there exists a non-canonical isomorphism between $(E,\theta)$ and $(E',\theta')$ after \'etale localisation on $X$, e.g.\ by \cite[Thm~1.2]{heuer-sheafified-paCS}. Indeed, we can see this via twisting: Let $B$ be the coherent quotient of $(\theta,\theta',0):T_X\to \End((E,\theta)\oplus(E',\theta')\oplus (\wtOm,0))$ from \Cref{d:wtTX-action} and let $\B=\nu^\ast B$.
		  Set $L:=\mathcal L_{\tau_B,f}$ and $L':=\mathcal L_{\tau_B,f'}$. Then by \Cref{l:LS-is-twisting}, and \Cref{r:enlarge-B}, we see that we can use $L$ and $L'$ to compute $\LS_f$ and $\LS_{f'}$, namely we can find isomorphisms
		 \[\lambda:V=\LS_f^{-1}(E,\theta)\isomarrow \nu^\ast E\otimes_\B L \text{ and }\lambda':V=\LS_{f'}^{-1}(E',\theta')\isomarrow \nu^\ast E'\otimes_\B L'.\] 
		 Since $\HTlog(L)=\tau_B=\HTlog(L')$, there is by \Cref{l:tech-lemma-descr-Ltau} after a further localisation a $\B$-linear isomorphism $\psi:L\isomarrow L'$. We can combine this to an isomorphism
		 \[ \LS_f^{-1}(E,\theta)=V\xrightarrow{\lambda'} \nu^\ast E'\otimes_\B L'\xrightarrow{\id\otimes \psi^{-1}} \nu^\ast E'\otimes_\B L\isomarrow \LS_f^{-1}(E',\theta').\]
		 Since $\LS_f^{-1}$ is fully faithful, this comes from an isomorphism
		 $\phi:(E,\theta)\isomarrow (E',\theta')$ of Higgs bundles.
		 Summarising the discussion, this shows  that the following diagram commutes:
		\[
		\begin{tikzcd}
			   & \theta_V: V \arrow[d, "\id",xshift=0.32cm] \arrow[r, "\lambda"] &  \nu^\ast E\otimes_\B L \arrow[d, "\phi\otimes \psi"] \arrow[r, "\theta_E\otimes \id"] & \nu^\ast(E\otimes \wtOm)\otimes_\B L \arrow[d, "(\phi\otimes \id)\otimes \psi"] \arrow[r, "\lambda^{-1}"] & V\otimes \wtOm \arrow[d, "\id\otimes \id"] \\
		&	\theta_{V}':  V \arrow[r, "\lambda'"]                 &  \nu^\ast E'\otimes_\B L' \arrow[r, "\theta_{E'}\otimes \id"]                 & \nu^\ast(E'\otimes \wtOm)\otimes_\B L' \arrow[r, "\lambda'^{-1}"]                          & V\otimes \wtOm                             
		\end{tikzcd}\]
		This shows that $\theta_V=\theta_{V'}$.
		Functoriality in $X$ can be seen by the same argument.
	\end{proof}
	Alternatively, even without checking that the left square at the end of the proof commutes, we could simply define $\Psi:V\to V$ as the unique isomorphism making the square commute. Then the diagram says that we have an isomorphism $\Psi:(V,\theta_V)\to (V,\theta'_V)$. But since $\theta_V$ commutes with any endomorphism of $V$ by functoriality, this implies that $\theta_V=\theta_V'$.
	\begin{Corollary}\label{c:theta_E-vs-theta_V}
		If $U$ is a smooth rigid space with toric chart $f:U\to \mathbb T^d$, then the isomorphism $E(\wt U)=V(\wt U)$ of \Cref{t:local-corresp} identifies the pullback of $\theta_E$ and $\theta_V$ to $\wt U$.
	\end{Corollary}
	\begin{proof}
		Immediate from \Cref{t:Camargos-Higgs-field}.2 and functoriality in the last part of \Cref{t:local-corresp}.
	\end{proof}
	
	\begin{Corollary}\label{p:twist-of-HB-recovers-can-Higgs-field}
		Let $X$ be a smooth rigid space and let  $(E,\theta_E)$ be a Higgs bundle on $X$. Let $B\subseteq \uEnd(E)$ and $\tau_B$ be associated to $\theta_E$ as in \Cref{d:wtTX-action} and set $\B:=\nu^\ast B$. Let $\mathcal L$ be any invertible $\B$-module on $X$ such that $\HTlog(\mathcal L)=\tau_B$. Then $V:=\nu^\ast E\otimes_{\B}\mathcal L$ is a pro-\'etale vector bundle on $X$ whose canonical Higgs field $\theta_V:V\to V\otimes \wtOm_X$ of \Cref{t:Camargos-Higgs-field} is given by
		\[\nu^\ast \theta_E\otimes_{\B} \mathcal L:V\to V\otimes \wtOm_X\]
		defined more explicitly as the composition $\Sym^\bullet \wtOm_X^{\vee}\xrightarrow{\theta} \uEnd_B(E)\xrightarrow{\nu^\ast(-)\otimes_{\B}\mathcal L}\nu_\ast\uEnd(V)$.
	\end{Corollary}
	\begin{proof}
		The conclusion is \'etale-local on $X$, so we may assume that $X$ is toric with a toric chart $f:X\to \mathbb T^d$. By \Cref{l:tech-lemma-descr-Ltau}, we can after a further \'etale localisation  assume that there is an isomorphism $\mathcal L\isomarrow \mathcal L_{\tau_B,f}$, so it suffices to prove the statement for $V= \nu^\ast E\otimes_{\B}\mathcal L_{\tau_B,f}$.  But then we have $\LS_f(V)=(E,\theta)$ by \Cref{l:LS-is-twisting}. The claim then follows from \Cref{t:Camargos-Higgs-field}.2 and \Cref{l:LS-is-twisting}, which say that $\theta_V=\LS_f(\theta_E)=\nu^\ast\theta_E\otimes_\B\mathcal L_{\tau_B,f}$.
	\end{proof}

	We now explain how to use the canonical Higgs field $\theta_V$ to pass from  pro-\'etale vector bundles to Higgs bundles  by twisting with invertible $\B$-modules. This is based on the following:

	\begin{Definition}\label{d:make-proetVB-B-module}
		As in \Cref{d:wtTX-action}, for any pro-\'etale vector bundle $V$ on $X$, we can equivalently regard the canonical Higgs field $\theta_V$ as a homomorphism 
		\[ \wtOm_X^{\vee}\to \nu_\ast\uEnd(V), \quad \partial\mapsto (V\xrightarrow{\theta_V} V\otimes \wtOm_X\xrightarrow{\id\otimes \partial}V)\]
		on $X_{\et}$ that extends to an $\O_X$-algebra homomorphism
		$\theta_V:T_X\to \nu_\ast\uEnd(V)$.
		Let $B=B_V$ be the image of this map.  By \Cref{c:End(V)-coherent}, this is a coherent $\O_X$-algebra. Set $\B:=\nu^\ast B_V$, then $V$ is a $\B$-module in a canonical way. Like in \Cref{d:wtTX-action}, the image of $\id\in \uHom(\wtOm, \wtOm)=\wtOm^{\vee}\otimes\wtOm$ under the map $\wtOm^{\vee}\to B$ then defines a canonical section $\tau_{B}:=\tau_{\theta_V}\in H^0(X,B\otimes \wtOm)$.
	\end{Definition}
	\begin{Proposition}\label{p:twist-of-vVB-is-analytic}
		Let $X$ be a smooth rigid space. Let $V$ be a pro-\'etale vector bundle on $X$. Let $B=B_V$ and $\B:=\nu^{\ast}B_V$ be as in \Cref{d:make-proetVB-B-module}. Let  $L$ be any pro-\'etale invertible \mbox{$\B$-module} on $X_\proet$ with $\HTlog(L)=\tau_B\in H^0(X,B\otimes \wtOm)$. Then
		$E:=V\otimes_{\B} L^{-1}$
		is an analytic-locally trivial vector bundle on $X$, which inherits a natural Higgs field $\theta_E:=\theta_V\otimes_{\B} L^{-1}$.
	\end{Proposition}
	\begin{proof}
		The statement is \'etale-local, so we may assume that $X$ is  toric and fix a toric chart $f:X\to \mathbb T^d$ so that  we are in the setup of \Cref{t:local-corresp}:  More precisely, by \Cref{l:LS-is-twisting}, we may assume that there is a Higgs bundle $(E',\theta')$ on $X$ such that $V=\nu^\ast E'\otimes_\B\mathcal L_{\tau_B,f}$. By \Cref{t:Camargos-Higgs-field}.2, the canonical Higgs field $\theta_V$ is then $\nu^\ast E'\otimes_\B\mathcal L_{\tau_B,f}\to \nu^\ast(E'\otimes \wtOm)\otimes_\B\mathcal L_{\tau_B,f}$. More precisely, we a priori need to be more careful and use $B=B_{\theta'}$, but the natural map
		\[ \theta_V:\wtOm_X^\vee\to \nu_\ast\uEnd(V)\]
		factors through $-\otimes_{\B_{\theta'}}\mathcal L_{\tau_{\theta'},f}:\uEnd(E',\theta')\to \nu_\ast\uEnd(V)$, which shows that $B_{\theta'}=B_V$. Hence
		\[ E=V\otimes_{\B} L^{-1}=\nu^\ast E'\otimes_{\B}\mathcal L_{\tau_B,f}\otimes_{\B} L^{-1}.\]
		To prove that $E$ is \'etale-locally trivial, it thus suffices to prove that $\mathcal L_{\tau_B,f}\otimes_{\B}L^{-1}$ is an \'etale-locally trivial invertible $\B$-module. By \Cref{c:Leray-seq-for-B^x}.2, this follows from the fact that
		\[ \HTlog(\mathcal L_{\tau_B,f}\otimes_{\B}\mathcal L^{-1})=\HTlog_B(\mathcal L_{\tau_B,f})-\HTlog_B(\mathcal L^{-1})=\tau_B-\tau_B=0.\]
		The Higgs field comes from the fact that $-\otimes_{\B} L^{-1}$ defines a natural map $\nu_\ast\uEnd(V)\to \uEnd(E)$. We can compose it with $\theta_V:T_X\to \nu_\ast\uEnd(V)$ to get the desired Higgs field $T_X\to \uEnd(E)$.
		\end{proof}
	
	\section{The $p$-adic Simpson correspondence}
	\subsection{Proof of Main Theorem}
	
	We can now prove our main result.
	\begin{Theorem}\label{t:p-adicSimpson-proper}
		Let $K$ be a complete algebraically closed extension of $\Q_p$.
		Let $X$ be a smooth proper rigid space over $K$. Let $\nu:X_\proet\to X_\et$ be the natural map. Let $\mathbb X$ be a $\BdR^+/\xi^2$-lift of $X$ and let $\Exp$ be an exponential for $K$.  Such choices always exist. According to \Cref{t:inv-B-module-assoc-to-HF}, they induce for any $\O_X$-coherent $\O_X$-torsionfree $T_X:=\Sym^\bullet_{\O_X} \wtOm^\vee$-algebra $B$ an invertible $\nu^{\ast}B$-module $\mathcal L_B$ in a way that is functorial in $B$ and $(X,\mathbb X)$.
		\begin{enumerate}
			\item The following functors define an exact tensor equivalence of categories:
		\begin{eqnarray*}
		\mathrm{S}_{\mathbb X,\Exp}:\{\text{pro-\'etale vector bundles on }X\}&\isomarrow &\{\text{Higgs bundles on }X\}\\
		V& \mapsto & \nu_{\ast}\big((V,\theta_V)\otimes_{\nu^\ast B_V} \mathcal L_{B_V}^{-1}\big)\\
		\nu^\ast E\otimes_{\nu^\ast  B_\theta} \mathcal L_{B_\theta}& \mapsfrom & (E,\theta)
		\end{eqnarray*}
		 where $B_\theta$ and $B_V$ are as defined in Definitions~\ref{d:wtTX-action} and \ref{d:make-proetVB-B-module}. 
		 \item The equivalence $\mathrm{S}_{\mathbb X,\Exp}$ is natural in $(X,\mathbb X,K,\Exp)$: Let the setup be as in \Cref{t:inv-B-module-assoc-to-HF}.3, so $X'$ is a smooth proper rigid space with a lift $\mathbb X'$ and $f:X'\to X$ is a morphism over $K$. Then any lift $\tilde{f}:\mathbb X'\to \mathbb X$ induces a natural transformation $t_{\tilde{f}}$
		 \[
		 \begin{tikzcd}[row sep = 0.8cm]
		 	\{\text{pro-\'etale vector bundles on }X\} \arrow[r, "\mathrm S_{\mathbb X,\Exp}"]  \arrow[d,"f^\ast"]& \{\text{Higgs bundles on }X\}      \arrow[d,"f^\ast"] \\
		 	\{\text{pro-\'etale vector bundles on }X'\} \arrow[r, "\mathrm S_{\mathbb X',\Exp'}"]  \arrow[ru,Rightarrow,shorten >=7ex,shorten <=7.0ex,"t_{\tilde{f}}" {xshift=-1pt,yshift=-1.5pt}]& 	\{\text{Higgs bundles on }X'\}.
		 \end{tikzcd}
		 \]
		 such that $t_{(-)}$ is compatible with compositions.
		 \end{enumerate}
	\end{Theorem}
	
	\begin{Remark}
		In general, changing either $\mathbb X$ or $\Exp$ will have a non-trivial effect on $\mathrm S_{\mathbb X,\Exp}$ already on the level of isomorphism classes of objects, and already for line bundles it is subtle to describe this effect explicitly. In fact, comparing functors for different choices of lifts $\mathbb X$ is closely related to Faltings' notion of ``twisted pullback''.
	\end{Remark}
	\begin{proof}[Proof of \Cref{t:p-adicSimpson-proper}]
		We first show that the two mappings are well-defined on objects:
		
		Let $V$ be a pro-\'etale vector bundle on $X$. Let $B_V$ be as in \Cref{d:make-proetVB-B-module}. Note that this is $\O_X$-torsionfree because $\uEnd(V)$ is.
		By \Cref{t:inv-B-module-assoc-to-HF}.1, the invertible $\B_V:=\nu^\ast B_V$-module $\mathcal L_{B_V}$ satisfies $\HTlog(\mathcal L_{B_V})=\tau_{B_v}$. Then by \Cref{p:twist-of-vVB-is-analytic}, the $\O_X$-module $(V,\theta_V)\otimes_{\B_V}\mathcal L_{B_V}^{-1}$ on $X_\proet$ is an analytic-locally trivial vector  bundle endowed with a Higgs field. Hence its restriction to $X_{\et}$ is a  Higgs bundle.
		 In the other direction, since $E$ is a vector bundle and $\mathcal L_{B_{\theta}}$ is pro-\'etale locally on $X$ isomorphic to $\B_\theta:=\nu^\ast B_\theta$, it is clear that $\nu^\ast E\otimes_{\B_\theta}\mathcal L_{B_{\theta}}$ is a pro-\'etale vector bundle. Hence the mappings are well-defined on objects.
		
		To see that they are functorial, let $\varphi\colon V\to W$ be any morphism of pro-\'etale vector bundles. If $\varphi$ is an isomorphism, then we have a canonical isomorphism $B_V=B_W$ which induces an isomorphism $\mathrm{S}_{\mathbb X,\Exp}(V)\to \mathrm{S}_{\mathbb X,\Exp}(W)$. In general, we can write $\varphi$ as a composition
		\[V\xrightarrow{(\id,0)} V\oplus W\xrightarrow{\smallmat{1}{\varphi}{0}{1}} V\oplus W\xrightarrow{(0,\id)} W\]
		to reduce to showing that the construction is compatible with direct sums.
				
		It is clear from \Cref{t:Camargos-Higgs-field} and compatibility of the local correspondence with $\oplus$ that $\theta_{V\oplus W}=\theta_V\oplus \theta_W$ on $V\oplus W$. Consequently, the restriction of the $T_X$-action on $V\oplus W$ to either subspace defines natural maps
		$B_{V\oplus W}\to B_{V}$ and $B_{V\oplus W}\to B_{W}$.
		It thus suffices to observe that for any morphism of $\O_X$-coherent $T_X$-algebras $\psi:B'\to B$ with canonical sections $\tau_{B}$ and $\tau_{B'}$ as in \Cref{d:make-proetVB-B-module}, we have by  \Cref{t:inv-B-module-assoc-to-HF}.2 a canonical identification
		\[ (-)\otimes_{\B'}\mathcal L_{\tau_{B'}}^{-1}= (-)\otimes_{\B}\B\otimes_{\B'}\mathcal L_{\tau_{B'}}^{-1}= (-)\otimes_{\B}\mathcal L_{\tau_{B}}^{-1},\]
		where as usual we write $\B=\nu^\ast B$ and $\B'=\nu^\ast B'$.
		Applying this transformation to the inclusion $\psi=V\to V\oplus W$ and the projection $\psi=V\oplus W\to W$, we obtain the desired compatibility with $\oplus$. The exactness can be seen by the same argument.
		
		The other direction works in exactly the same way. Thus both functors are well-defined. 
		
		Let us write $\mathrm{T}_{\mathbb X,\Exp}$ for the functor from right to left.
		We claim that $\mathrm{S}_{\mathbb X,\Exp}$ and $\mathrm{T}_{\mathbb X,\Exp}$ are mutual quasi-inverses to each other:  Let $(E,\theta_E)$ be a Higgs bundle on $X$ and $V=\mathrm{T}_{\mathbb X,\Exp}(E,\theta_E)$. Then by \Cref{p:twist-of-HB-recovers-can-Higgs-field}, we have a commutative diagram
		\[
		\begin{tikzcd}[row sep = 0cm]
			& {\uEnd(E,\theta_E)} \arrow[dd,hook, "-\otimes_{\B_{\theta}}\mathcal L_\theta"] \\
			T_X \arrow[rd, "\theta_V"'] \arrow[ru, "\theta_E"] &                                                                       \\
			& \nu_\ast\uEnd(V)                                                             
		\end{tikzcd}\]
		where $\theta=\theta_E$ in the subscript.
		Since the vertical map is clearly injective, we thus have a canonical identification of the respective images $B_{\theta}=B_V$ of $\theta_E$ and $\theta_V$,  which identifies $\tau_{\theta}$ with $\tau_{B_V}$ and hence also $\mathcal L_V$ with $\mathcal L_\theta$. It follows that we have
		\[\mathrm{S}_{\mathbb X,\Exp} \circ \mathrm{T}_{\mathbb X,\Exp}(E,\theta)=\nu_\ast\big(\nu^\ast(E,\theta_E)\otimes_{\B_\theta}\mathcal L_\theta\otimes_{\B_V}\mathcal L_V^{-1})\big)=(E,\theta_E).\]
		The other direction can be seen in exactly the same way, using instead that the diagram
		\[
		\begin{tikzcd}[row sep = 0cm]
			& {\uEnd(E)} \\
		T_X \arrow[rd, "\theta_V"'] \arrow[ru, "\theta_E"] &                                                                       \\
			& \nu_\ast\uEnd(V)   \arrow[uu,hook, "-\otimes_{\B_{V}}\mathcal L_{B_V}^{-1}"']                                                           
		\end{tikzcd}\]
		commutes by definition of $\theta_E$ as in \Cref{p:twist-of-vVB-is-analytic}.
		
		For part (1), it remains to show that the equivalence is compatible with $\otimes$. Let $(E_1,\theta_1)$ and $(E_2,\theta_2)$ be two Higgs bundles on $X$.
		For $i=1,2$, let $B_i$ be the coherent $\O_X$-algebra associated to $(E_i,\theta_i)$. Let $B_3$ be the maximal $\O_X$-torsionfree quotient of $B_1\otimes_{\O_X}B_2$, this acts on $E_1\otimes_{\O_X} E_2$. Let $\B_i:=\nu^\ast B_i$ for $i=1,2,3$ and to simplify notation let $\mathcal L_i:=\mathcal L_{B_i}$. Then by \Cref{t:inv-B-module-assoc-to-HF}.4, we have  a natural isomorphism of $\B_3$-modules
		\begin{align*}
		\nu^\ast(E_1\otimes_{\O_X} E_2)\otimes_{\B_3}\mathcal L_3&\isomarrow \nu^\ast(E_1\otimes_{\O_X} E_2)\otimes_{\B_3}(\mathcal L_1\otimes_{\B_1}\B_3)\otimes_{\B_3} (\B_3\otimes_{\B_2}\mathcal L_2)\\&=(\nu^\ast E_1\otimes_{\B_1}\mathcal L_1)\otimes_{\O_X} (\nu^\ast E_2\otimes_{\B_2}\mathcal L_2)
		\end{align*}
		as we wanted to see.
		
		Finally, the naturality in $(X,\mathbb X,K,\Exp)$ follows immediately from the naturality of the $\mathcal L_B$ described in \Cref{t:inv-B-module-assoc-to-HF}.3.
	\end{proof}
	
	As an application of rigidifications, we can interpret our $p$-adic Simpson functor in terms of deformations, as follows. This is useful for future applications to Chern classes.
	\begin{Proposition}
		In the setting of \Cref{t:p-adicSimpson-proper},
		let $V$ be a pro-étale vector bundle on $X$ and let $(E,\theta):=\mathrm{S}_{\mathbb X,\Exp}(V)$. Then there is a connected smooth rigid space $S$ over $K$ with two distinguished points $0\in S(K)$ and $s\in S(K)$ and a pro-\'etale vector bundle $N$ on $X\times S$ such that the fibres of $N$ over $0$ and $s$, respectively, are $N|_{X\times \{0\}}\cong \nu^\ast E$ and $N|_{X\times \{s\}}\cong V$.
	\end{Proposition}
	\begin{proof}
		Let $S:=\mathcal P^0_{\mathbb X}$ be the connected rigid group considered  in the proof of \Cref{t:inv-B-module-assoc-to-HF}. Let $s:=\Exp_S(\tau_{B})\in S(K)$ where $B:=B_\theta$ and set $\B:=\nu^\ast B$. By \Cref{p:explicit-description-of-Pic_B}, the map
		\[ H^1_\proet(X\times S,\B^\times)\to \uPic_{\B,\proet}(S)\]
		is surjective, so we can find a pro-étale invertible $\B$-module $\mathcal L$ on $X\times S$ whose isomorphism class corresponds to the natural map $S\to \uPic_{\B,\proet}$. Then $\mathcal L|_{X\times \{0\}}\cong \B$ and $\mathcal L|_{X\times \{s\}}\cong \mathcal L_{B}$ by \Cref{t:inv-B-module-assoc-to-HF}.1. So $N:=\pi^{\ast}_1\nu^\ast E\otimes_{\B}\mathcal L$ has the desired properties by definition of $\mathrm{S}_{\mathbb X,\Exp}^{-1}$.
	\end{proof}

	\subsection{The cohomological comparison}
	\begin{Definition}\label{d:Dolbeault-cohom}
	Let $X$ be a smooth rigid space and let $(E,\theta)$ be a Higgs bundle on $X$. Recall that we write $\wtOm^k_X=\wedge^k_{\O_X}\wtOm_X=\Omega^k_X(-k)$. The Higgs complex of $(E,\theta)$ is then defined as
	\[ \mathcal C^\ast_{\mathrm{Higgs}}(E,\theta):=\big [E\xrightarrow{\theta}E\otimes \wtOm^1_X\xrightarrow{\theta_1}E\otimes \wtOm^2_X\xrightarrow{\theta_2}\dots \xrightarrow{\theta_{n-1}} E\otimes \wtOm^n_X\big]\]
	where
	$\theta_k(e\otimes w):=\theta(e)\wedge w$. We then define the Dolbeault cohomology of $(E,\theta)$ as \[R\Gamma_{\mathrm{Higgs}}(X,(E,\theta)):=R\Gamma(X_{\et},\mathcal C^\ast_{\mathrm{Higgs}}(E,\theta)).\]
	\end{Definition}
	A more conceptual definition of Dolbeault cohomology is given by the observation that
	\[ R\Gamma_{\mathrm{Higgs}}(X,(E,\theta))=\Ext_{T_X}(\O_X,E)\]
	where $E$ is equipped with the $T_X:=\Sym^\bullet \wtOm^\vee$-module structure defined by $\theta$. Indeed, to compute the right hand side, we can use the resolution of $\O_X$ as a $T_X$-module
	\[K_\bullet:=\Big[T_X\otimes \wtOm_X^{d\vee}\to \dots \to T_X\otimes \wtOm_X^{2\vee}\to T_X\otimes \wtOm_X^\vee\to T_X\Big]\to \O_X\]
	that locally on $X$ in terms of any choice of basis $\partial_1,\dots,\partial_d\in \wtOm^\vee$ is the Koszul complex $\mathrm{Kos}_\bullet(T_X;\partial_1,\dots,\partial_d)$. One easily sees that $\uExt_{T_X}(\O_X,E)=R\uHom_{T_X}(K_\bullet,E)=\mathcal C^\ast_{\mathrm{Higgs}}(E,\theta)$.
	
	We can now prove the last remaining result mentioned in the introduction:
	\begin{Theorem}\label{t:cohomological-comparison}
		In the setting of \Cref{t:p-adicSimpson-proper}, let $V$ be a pro-\'etale vector bundle on $X$ and let $(E,\theta)=
		\mathrm{S}_{\mathbb X,\Exp}(V)$ be the associated Higgs bundle. Then
		there is a natural isomorphism 
		$R\nu_\ast V=\mathcal C^\ast_{\mathrm{Higgs}}(E,\theta)$ in $D(X_{\et})$.
		Its cohomology over $X$ defines a natural isomorphism
		\[R\Gamma(X_\proet,V)=R\Gamma_{\mathrm{Higgs}}(X,(E,\theta)).\]
	\end{Theorem}
	
	\begin{Remark}
		Following \cite{anschutz2023hodgetate}, a natural way to prove this would be to show that \Cref{t:p-adicSimpson-proper} extends to perfect complexes. We believe that, in principle, this should be possible by our approach if one has a canonical Higgs structure on perfect complexes on $X_\proet$. Indeed, the result about perfect complexes will appear in the upcoming work of Bhatt--Zhang on the Simpson gerbe (cf \Cref{r:Simpson-gerbe}) by a twisting construction that is related to ours.
	\end{Remark}
	\begin{proof}
		We begin by defining a natural morphism in $D(X_\et)$
		\[ \mathcal C^\ast_{\mathrm{Higgs}}(E,\theta)=\uExt_{T_X}(\O_X,E)\to \uExt_{X_\proet}(\O_X,V)=R\nu_\ast V.\]
		 For this, we combine ideas of \cite{anschutz2023hodgetate} and \cite{heuer-sheafified-paCS}: 
		 Let $J:=\ker(T_X\xrightarrow{\theta} \uEnd(E))$  and consider \[\textstyle \widehat{T}_J:=\varprojlim_n T_X/J^n.\]
		 Note that $\widehat{T}_J$ still acts on $E$ via the natural projection $\widehat{T}_J\to B:=T_X/J\to \uEnd(E)$.
		  Since the natural map $T_X\to \widehat{T}_J$ is flat, we therefore have 
		 \[\uExt_{T_X}(\O_X,E)=\uExt_{\widehat{T}_J}(\O_X\otimes_{T_X}\widehat{T}_J,E).\]
		 
		 In order to compare this to the cohomology of the pro-\'etale vector bundle $V$, we would
		  like to lift the invertible $\B$-module $\mathcal L_B$ from  \Cref{t:inv-B-module-assoc-to-HF} for $\B:=\nu^\ast B$ to a $\widehat{T}_J$-module. Setting $B_n:=T_X/J^n$ and $\mathcal B_n:=\nu^\ast B_n$ for each $n\in \N$, we would like to do this by associating to $B_n$ an invertible $\mathcal B_n$-module $\mathcal L_{n}:=\mathcal L_{B_n}$. We caution that we cannot directly invoke \Cref{t:inv-B-module-assoc-to-HF} to do so if we don't know that $B_n$ is $\O_X$-torsionfree. However, we can still use \Cref{eq:exp-on-P_X} to associate to each $B_n$ an {isomorphism class} $\Exp(\tau_{B_n})\in \mathcal P_{\mathbb X}(K)$, functorial in $B$. Choosing a representative $\mathcal L_{B_n}$ in this class for any $n\geq 2$ (we set $\mathcal L_{B_1}:=\mathcal L_B$) as well as isomorphisms $\psi_n:\mathcal L_{n+1}\otimes_{\mathcal B_{n+1}}\mathcal B_n\isomarrow \mathcal L_n$ for any $n\geq 1$, which exist by functoriality,
		we still obtain an invertible $\nu^\ast \widehat{T}_J$-module 
		 \[\textstyle\widehat{\mathcal L}:=\varprojlim_{n\in \N} \mathcal L_{n}\]
		 lifting the invertible $\B$-module $\mathcal L_B$ used in \Cref{t:p-adicSimpson-proper}. We note that $\widehat{\mathcal L}$ is uniquely determined up to isomorphism: If $(\mathcal L_{n}')_{n\in \N}$ is any other system of representatives with choices of transition maps and $\widehat{\mathcal L}'$ denotes their limit, then we can find an isomorphism  $\widehat{\mathcal L}'\to \widehat{\mathcal L}$. This follows from the fact that $\Hom_{\B_n}(\mathcal L_n,\mathcal L_n')\cong H^0(X,B_n)$ is a finite dimensional $K$-vector space, so $(\mathrm{Isom}_{\B_n}(\mathcal L_n,\mathcal L_n'))_{n\in \N}$ is a Mittag-Leffler system, hence has non-empty limit.
		 
		 \medskip
		  
		  For simplicity, let us drop the $\nu^{\ast}$ from notation in the following. We can now define a natural morphism in $D(X_\et)$
		 \[\psi:\uExt_{\widehat{T}_J}(\O_X\otimes_{T_X}\widehat{T}_J,E)=R\uHom_{\widehat T_{J}}(K_\bullet\otimes_{T_X}\widehat T_{J},E)\xrightarrow{-\otimes_{\widehat{T}_J}\widehat{\mathcal L}} R\uHom_{X_\proet}(K_\bullet\otimes_{T_X}\widehat {\mathcal L},E\otimes_{\widehat{T}_J}\widehat{\mathcal L}).\]
		 This is functorial in $E$ and $B=T_X/J$. In particular, we may now without loss of generality make $B$ bigger, i.e.\ $J$ smaller, to assume that $J$ is contained in the kernel of the projection $T_X\to \O_X$ induced by the zero map $0:\wtOm^\vee\to \O_X$.
		 
		 Since the $\widehat{T}_J$-action on $E$ factors through $B$, we have a canonical identification
		 \begin{equation}\label{eq:cohom-comparison-target-1} E\otimes_{\widehat{T}_J}\widehat{\mathcal L}=E\otimes_{\B}\mathcal L=V.
		 \end{equation}
		 	 Moreover, since $-\otimes_{\widehat{T}_J}\widehat{\mathcal L}$ is exact, we have identifications
		\begin{equation}\label{eq:cohom-comparison-target-2} 	 
		 	 K_\bullet\otimes_{\widehat{T}_J}\widehat{\mathcal L}=\O_X\otimes_{\widehat{T}_J}\widehat{\mathcal L}=\O_X\otimes_\B \mathcal L=\O_X
		\end{equation}
		 because the action of $\B$ on $\O_X$ is induced by the $T_X$-action on $\O_X$ for which $\wtOm^\vee$ acts trivially.
		  Hence the right hand side of $\psi$ is naturally identified with the desired target $\uExt_{X_\proet}(\O_X,V)$.
		 
		 We note that the map $\psi$ does not depend on the choice of representative $\mathcal L_{B_n}$ or the $\psi_n$: Indeed, \eqref{eq:cohom-comparison-target-1} and \eqref{eq:cohom-comparison-target-2} give a canonical identification of the target which is independent of $\widehat{\mathcal L}$, so the map $-\otimes_{\widehat{T}_J}\widehat{\mathcal L}$ only depends on the isomorphism class of $\widehat{\mathcal L}$.
		 	 
		 It remains to see that $\psi$ is a quasi-isomorphism. In other words, we  need to prove:
		 \begin{Lemma}
		 For any $n\in \Z$, the following map is an isomorphism on $X_{\et}$:
		 \[H^n\psi:\uExt_{\widehat{T}_J}^n(\O_X\otimes_{T_X}\widehat{T}_J,E)\to \uExt^n_{X_\proet}(\O_X,V).\]
		 
		 \end{Lemma}
		 \begin{proof}
		  As this can be checked locally on $X_{\et}$, we can now replace $X$ by a toric object $f:U\to \mathbb T^d$ of $X_{\et}$ with a fixed toric chart and assume that $E$ and $V$ are small.
		 
		 In this case, we first note that the chart $f$ defines an integral subsheaf $\wtOm_U^+\subseteq \wtOm_U$ as the preimage of $\Hom(\Delta,\O^+_U)\subseteq \Hom(\Delta,\O_U)$ under the isomorphism $\rho_f$ of \Cref{d:rho}. This is a finite locally free $\O_U^+$-submodule of $\wtOm_U$. Its dual
		 $\wtOm_U^{+,\vee}:=\Hom_{\O_U^+}(\wtOm_U^+,\O^+_U)=\Delta\otimes_{\Z_p} \O^+_U$
		 inherits from $\Delta$ a basis $\partial_1,\dots,\partial_d$. 
		  We can use this to compare $\widehat{T}_J$ to a much smaller submodule: Over $U$, we first obtain an integral submodule $T_U^+:=\Sym_{\O_U^+}^\bullet p^{-\alpha}\wtOm_U^{+,\vee}\subseteq T_U$,  where $\alpha$ is as in \S\ref{s:exp-for-rigid-groups}. Write $J^+:=J_{|U}\cap T_U^+$. We can use this to define the following diagram
		 \[
		 \begin{tikzcd}
		 	T_U=\O_U[\partial_1,\dots,\partial_d] \arrow[d] \arrow[r]                                             & \varprojlim_n T_U^+/J^{+n}\tf \arrow[d] \arrow[r] & \widehat{T}_J \\
		 	{\mathcal T_U:=\mathcal O_U\langle p^{-\alpha}\partial_1,\dots,p^{-\alpha}\partial_d\rangle} \arrow[r] & {\mathcal O^+_U[[p^{-\alpha}\partial_1,\dots,p^{-\alpha}\partial_d]]\tf}                & 
		 \end{tikzcd}\]
		 of flat $T_U$-modules,
		 where the vertical maps send $\partial_i\to \partial_i$, and the second vertical map is well-defined as our assumptions on $J$ ensure $J^+\subseteq (p^{-\alpha}\partial_1,\dots,p^{-\alpha}\partial_d) T^+_U$. By flatness, we can use either of these algebras $A$ to compute $\uExt_{T_X}^n(\O_X,E)_{|U}$ as the sheaf $\uHom_A(K_\bullet\otimes_{T_U}A,E|_U[n])$. 
		 
		 We shall use the algebra $\mathcal T_U$. To simplify notation, we will in the following also just write this as $\mathcal O_U\langle p^{-\alpha}\partial\rangle:=\mathcal O_U\langle p^{-\alpha}\partial_1,\dots,p^{-\alpha}\partial_d\rangle$, and similarly for the other algebras. 
		 
		 Note that due to the assumption that $E$ is small, the $T_U$-action on $E_{|U}$ extends to a $\mathcal T_{U}$-action on $E_{|U}$. In particular, we still have a natural map $\mathcal T_U\to B$. The relevance of the convergence condition is now that after a further \'etale localisation, we can also lift $\mathcal L_B$ to an invertible $\mathcal T_U$-module $\mathcal L$: Namely, according to \Cref{l:tech-lemma-descr-Ltau}, we can simply take $\mathcal L$ to be the $\mathcal T_U$-module whose $\Delta$-action on $\mathcal T_U(\wt U)$ is defined by the continuous 1-cocycle
		 \[ c:\Delta\to \mathcal T_U^\times(U),\quad \gamma_i\mapsto \exp(\partial_i).\]
		 This allows us to compute $\psi_{|U}$ more explicitly: It is given by  sending a homomorphism
		 \[  K_\bullet\otimes_{T_U}\mathcal T_U\to E_{|U}[n]\]
		 to the map in $D(U_\proet)$ associated to the $\Delta$-linear morphism of complexes of $\O(\wt U)$-modules
		 \begin{equation}\label{eq:morph-from-Koszul-Higgs-side}
		 (K_\bullet\otimes_{T_U}\mathcal T_U)(\wt U)\to E_{|U}\otimes_{T_U}\mathcal T_U[n](\wt U)
		 \end{equation}
		 where the action on $K_\bullet$ is trivial and the action on the second factors is via $c$. On the left hand side, we note that we can describe $(K_\bullet\otimes_{T_U}\mathcal T_U)(U)$ as the Koszul complex over $R:=\O(U)$,
		 \[  (K_\bullet\otimes_{T_X}\mathcal T_U)(U)=\mathrm{Kos}_\bullet(R\langle p^{-\alpha}\partial\rangle;\partial_1,\dots,\partial_d).\]
		  On the right hand side of \Cref{eq:morph-from-Koszul-Higgs-side}, we recover $V(\wt U)[n]$ with its natural action via  \Cref{t:local-corresp}. At this point, it suffices to show that every class in $\Ext^n_{U_\proet}(\O_U,V)$ is of the form  \Cref{eq:morph-from-Koszul-Higgs-side}.
		 
		 To see this, we use that by \cite[Lemma 6.7]{heuer-sheafified-paCS}, the Cartan--Leray sequence of $\wt U\to U$ induces an isomorphism
		 $H^n_{\cts}(\Delta,M)=\Ext^n_{U_\proet}(\O_U,V)$
		 for $M:=E(U)$ with $\Delta$-action via $c$. Unravelling the definitions of continuous cohomology $H^\ast_{\cts}$ and the Cartan--Leray sequence, this means the following: Any element of $\Ext^n_{U_\proet}(\O_U,V)$ can be represented by the morphism of complexes of pro-\'etale sheaves described over $\wt U$ by the $\Delta$-equivariant $\O(\wt U)$-linear map of complexes obtained by tensoring the $R$-linear $\Delta$-equivariant morphism
		 \[\mathrm{Kos}_{\bullet}(R[[T_1,\dots,T_d]];T_1,\dots,T_d)\to M[n]\]
		 with $\O(\wt U)$. Here the action of $\gamma_i\in \Delta$ on the left is given by multiplication with $T_i+1$ (see for example the proof of \cite[Lemma~5.5]{Scholze_p-adicHodgeForRigid}). Exactly as on the Higgs side, it follows from the fact that $V$ is small that we can instead compute this using the $R$-linear $\Delta$-equivariant maps
		 \[\mathrm{Kos}_{\bullet}(R\langle p^{-\alpha}T\rangle;T_1,\dots,T_d)\to M[n].\]
		 
		 It thus suffices to see that there is a natural $R$-linear $\Delta$-equivariant isomorphism 
		 \[\mathrm{Kos}_\bullet(R\langle p^{-\alpha}\partial\rangle;\partial_1,\dots,\partial_d)\to \mathrm{Kos}_\bullet(R\langle p^{-\alpha}T\rangle;T_1,\dots,T_d).\]
		 To construct this, we first note that the natural map
		 \[\phi:R\langle p^{-\alpha}\partial\rangle\to R\langle p^{-\alpha}T\rangle,\quad \partial\mapsto \log(T+1)\]
		 is well-defined due to the convergence condition, and
		 is $\Delta$-equivariant because
		 \[ \phi(f(\partial)\cdot \gamma_i)=\phi(f(\partial)\exp(\partial_i))=f(\log(T+1))(T_i+1)=\phi(f(\partial))\cdot \gamma_i.\]
		It is an isomorphism because $T\mapsto \exp(\partial)-1$ is an inverse. We thus obtain an isomorphism
		\[\phi:\mathrm{Kos}_{\bullet}(R\langle p^{-\alpha}\partial\rangle;\partial_1,\dots,\partial_d)\isomarrow \mathrm{Kos}_{\bullet}(R\langle p^{-\alpha}T\rangle;\log(T_1+1),\dots,\log(T_d+1)).\]
		Finally, the fact that $\log(T+1)/T$ is a unit in $R\langle p^{-\alpha}T\rangle^\times$ means that the right hand side is isomorphic to $\mathrm{Kos}_\bullet(R\langle p^{-\alpha}T\rangle;T_1,\dots,T_d)$
		by \cite[0625]{StacksProject}. 
		
		All in all, this shows that $H^n\psi_{|U}$ is indeed an isomorphism, as we wanted to see.
	\end{proof}
	This finishes the proof of \Cref{t:cohomological-comparison}.
	\end{proof}

\end{document}